\documentclass[11pt]{amsart}
\usepackage[margin=1in]{geometry}
\usepackage[T1]{fontenc}
\usepackage{lmodern}
\usepackage{microtype}
\usepackage{amsmath,amssymb,amsthm,mathtools}
\usepackage{aliascnt}
\usepackage{mathrsfs}
\usepackage{enumitem}
\usepackage{hyperref}
\usepackage[nameinlink,capitalise]{cleveref}
\usepackage{tikz-cd}
\hypersetup{colorlinks=true,linkcolor=blue,citecolor=blue,urlcolor=blue}

\newtheorem{theorem}{Theorem}[section]
\newaliascnt{proposition}{theorem}
\newtheorem{proposition}[proposition]{Proposition}
\aliascntresetthe{proposition}
\newaliascnt{lemma}{theorem}
\newtheorem{lemma}[lemma]{Lemma}
\aliascntresetthe{lemma}
\newaliascnt{corollary}{theorem}
\newtheorem{corollary}[corollary]{Corollary}
\aliascntresetthe{corollary}
\theoremstyle{definition}
\newaliascnt{definition}{theorem}
\newtheorem{definition}[definition]{Definition}
\aliascntresetthe{definition}
\newaliascnt{remark}{theorem}
\newtheorem{remark}[remark]{Remark}
\aliascntresetthe{remark}
\newaliascnt{example}{theorem}

\aliascntresetthe{example}
\newaliascnt{claim}{theorem}

\aliascntresetthe{claim}
\newaliascnt{prop}{theorem}
\newtheorem{prop}[prop]{Proposition}
\aliascntresetthe{prop}
\crefname{theorem}{Theorem}{Theorems}
\crefname{proposition}{Proposition}{Propositions}
\crefname{lemma}{Lemma}{Lemmas}
\crefname{corollary}{Corollary}{Corollaries}
\crefname{definition}{Definition}{Definitions}
\crefname{remark}{Remark}{Remarks}
\crefname{claim}{Claim}{Claims}
\crefname{example}{Example}{Examples}
\crefname{prop}{Proposition}{Propositions}
\Crefname{theorem}{Theorem}{Theorems}
\Crefname{proposition}{Proposition}{Propositions}
\Crefname{lemma}{Lemma}{Lemmas}
\Crefname{corollary}{Corollary}{Corollaries}
\Crefname{definition}{Definition}{Definitions}
\Crefname{remark}{Remark}{Remarks}
\Crefname{claim}{Claim}{Claims}
\Crefname{example}{Example}{Examples}
\Crefname{prop}{Proposition}{Propositions}

\makeatletter
\renewcommand\paragraph{\@startsection{paragraph}{4}{\z@}%
 {1.25ex \@plus .2ex \@minus .2ex}%
 {.6ex}%
 {\normalfont\bfseries}}
\makeatother

\newcommand{\C}{\mathbb C}
\newcommand{\Q}{\mathbb Q}
\newcommand{\R}{\mathbb R}
\newcommand{\Z}{\mathbb Z}
\newcommand{\PP}{\mathbb P}
\newcommand{\OO}{\mathcal O}
\newcommand{\M}{\mathcal M}
\newcommand{\N}{\mathcal N}
\newcommand{\Acat}{\mathcal A}

\newcommand{\Coh}{\mathrm{Coh}}
\newcommand{\ch}{\mathrm{ch}}
\newcommand{\td}{\mathrm{td}}
\newcommand{\rk}{\mathrm{rk}}
\newcommand{\Hom}{\mathrm{Hom}}
\newcommand{\Ext}{\mathrm{Ext}}
\newcommand{\Gr}{\mathrm{Gr}}
\newcommand{\Hilb}{\mathrm{Hilb}}
\newcommand{\Sym}{\mathrm{Sym}}
\newcommand{\PD}{\mathrm{PD}}
\newcommand{\vir}{\mathrm{vir}}
\newcommand{\inv}{\mathrm{inv}}
\newcommand{\pl}{\mathrm{pl}}
\newcommand{\ssm}{\mathrm{ss}}

\newcommand{\Bl}{\mathrm{Bl}}

\newcommand{\supp}{\mathrm{supp}\text{ }}
\newcommand{\Res}{\mathrm{Res}}

\newcommand{\AdP}{\mathscr P}

\newcommand{\Dalg}{\mathbb D}

\newcommand{\eps}{\varepsilon}
\newcommand{\Ob}{\mathcal{O}b}

\newcommand{\virt}{\mathrm{vir}}

\title[Examples of PT descendent series on Fano threefolds]{Examples of descendent generating series for Pandharipande--Thomas stable pairs on smooth projective Fano threefolds via one-dimensional wall-crossing}
\author{Reginald Anderson}
\date{}

\begin{document}

\begin{abstract}
We study descendent generating series for Pandharipande--Thomas stable pairs on smooth projective Fano threefolds. We use the wall-crossing setup developed by the author and Joyce in Joyce's Lie algebra $H_*(\N^{\pl},\Q)$ of the projective-linear pairs stack, and next pass to Gross's polynomial realization $e^{\kappa}\Q[s_{jk\ell}]$.

We compute explicit examples of one-dimensional Donaldson--Thomas invariants on Fano 3-folds and, via wall-crossing, Pandharipande--Thomas stable pair invariants and descendent generating series. We compute examples on $\PP^3$, on a smooth cubic threefold, on $\Bl_p\PP^3$, on $\Bl_\ell\PP^3$, and on the projective-bundle threefold $\PP(\OO_X\oplus \OO_X(-1,-1))$ over $X=\PP^1\times\PP^1$. In the $\PP^3$ and cubic threefold examples we compare the intrinsic large-$n$ tails with the formulas of Pandharipande and Moreira and show that, in the cases treated in common, the differences are Laurent polynomials.
\end{abstract}

\maketitle

\tableofcontents

\section{Introduction}

Let $X$ be a smooth projective threefold over $\C$. A Pandharipande--Thomas stable pair on $X$ is a pair $(F,s)$ consisting of a pure one-dimensional coherent sheaf $F$ and a section $s\colon \OO_X\to F$ with zero-dimensional cokernel. For $\beta\in H_2(X,\Z)$ and $n\in\Z$ we write $P_n(X,\beta)$ for the moduli space of stable pairs with $[\supp F]=\beta$ and $\chi(F)=n$, equipped with its Behrend--Fantechi virtual class $[P_n(X,\beta)]^{\vir}$. Given classes $\gamma_i\in H^*(X,\Q)$ and integers $d_i\ge 0$, the associated descendent partition function is
\[
Z_{P}\!\left(X;q\,\middle|\,\prod_{i=1}^r \tau_{d_i}(\gamma_i)\right)_\beta
:=\sum_{n\in\Z}\left\langle\prod_{i=1}^r\tau_{d_i}(\gamma_i)\right\rangle^{X}_{n,\beta}q^n.
\]
Donaldson--Thomas invariants and Pandharipande--Thomas stable pair invariants have been extensively studied, especially on Calabi--Yau threefolds, beginning with the work of Thomas and Pandharipande--Thomas and continuing through wall-crossing and Hall-algebra methods developed by Joyce--Song, Bridgeland, and Toda \cite{ThomasCasson,PT,JoyceSong,BridgelandHall,TodaDTPT}. In the Fano threefold case, the large-$n$ structure of the descendent series is controlled by one-dimensional Donaldson--Thomas theory for semistable sheaves together with finitely-many Pandharipande--Thomas stable pair invariants and the corresponding wall-crossing to stable pairs \cite{JoyceWC,AJI}.

In the Fano setting, one-dimensional Donaldson--Thomas classes for pure one-dimensional sheaves are homology classes of complex virtual dimension $1$, whereas the stable-pairs moduli space $P_n(X,\beta)$ carries its usual Behrend--Fantechi obstruction theory of virtual dimension $c_1(X)\cdot\beta$. The auxiliary pairs category of \cite{JoyceWC} together with the wall-crossing package in the author's joint paper with Joyce \cite{AJI} place both kinds of invariants in the same intrinsic Lie algebra $H_*(\N^{\pl},\Q)$, which is the framework used throughout this paper.

The main purpose of the present paper is to compute explicit examples of the sheaf-theoretic invariants
\[
[\M^{\ssm}_{(\beta,n)}(\mu^\lambda_\omega)]_{\inv}
\]
and the Pandharipande--Thomas stable-pair classes
\[
[P_n(X,\beta)]^{\vir},
\]
first intrinsically in $H_*(\N^{\pl},\Q)$ and next in Gross's polynomial algebra $e^{\kappa}\Q[s_{jk\ell}]$, before extracting descendent generating series. The examples are the main content of the paper. We compute examples on $\PP^3$, on a smooth cubic threefold, on $\Bl_p\PP^3$, on $\Bl_\ell\PP^3$, and on the projective-bundle threefold $\PP(\OO_X\oplus \OO_X(-1,-1))$ over $X=\PP^1\times\PP^1$. For the line class on $\PP^3$, and for the line class on a smooth cubic threefold, we compare the intrinsic large-$n$ tails with the formulas of Pandharipande and Moreira and verify that the differences are Laurent polynomials in the cases treated in common.

Throughout, $X$ will be a smooth projective Fano threefold unless explicitly stated otherwise. Since $-K_X$ is ample, every effective curve class is superpositive in the sense of \cite{AJI}. The intrinsic wall-crossing results of \cite{JoyceWC,AJI} therefore apply to every effective class on $X$.

All of the material before the examples section is recalled from Joyce \cite{JoyceWC} and from the author's joint paper with Joyce \cite{AJI}. Section~\ref{sec:background} recalls the categories of one-dimensional sheaves and pairs, the intrinsic wall-crossing identities, Gross's polynomial realization, admissible periods, and the quasi-polynomiality, rationality, and pole statements used later. Section~\ref{sec:examples} is devoted to the examples, which are the main focus of the paper.

\subsection{Main computations of this paper}
All invariants here are written in the notation of their corresponding subsection. From the generalized virtual fundamental classes listed here, we also compute various examples of PT descendent invariants and large $n$ tails of Pandharipande-Thomas descendent generating series in corresponding subsections. 

\paragraph{$X = \PP^3$}
\subparagraph{$\beta=[L]$ is covered in Section~\ref{CP3betaIsL}}
From the Behrend--Fantechi virtual fundamental class \[ [\M^{ss}_{(L,n)}(\tau) ]^{BF} = 20\sigma_{2,1} \in A_1 \mathrm{Gr}(2,4) \] and the inclusion $\iota_{L,n}: \M^{ss}_{(L,n)}(\tau) \rightarrow \N^{\pl}_{(L,n)}$, we have that \[ [\M^{ss}_{(L,n)}(\tau) ]_{\inv} = (\iota_{L,n})_*(20\sigma_{2,1} \cap [G]) \in H_2(\N^{\pl}_{(L,n)},\Q). \]

In $\mathbb{D}$, after normalization, \[ [\M^{ss}_{(L,n)}(\tau)]_{\inv}^\perp = e^{(0,0,L,n-2)}(20s_{1,2,2} + (-10n^2+40n-\frac{110}{3}) s_{1,6,4} ). \]

In Section~\ref{P3LPTinvts}, we write out \[ [P_n(\PP^3,L)]^{\vir} = c_{n-1}(\mathrm{Obs}_n)\cap [B_n] \in H_8(P_n(\PP^3, L), \Q)\] for $1\leq n \leq 3$ separately. For $n\geq 4$, 

\[ [P_n(\PP^3,L)]^{\vir} = 20\,q^*(\sigma_{2,1})\cap
c_{n-4}\bigl(q^*W_n\otimes \mathcal{O}_{B_n}(-1)\bigr)
\cap [B_n] \in H_8(P_n(\PP^3,L),\Q). \] 
Here $W_4=0$ and $W_n=\Sym^{n-5}(S)$ for $n\geq 5$.

In $\mathbb{D}$, the PT stable pair virtual fundamental class 

\[ \tilde{P}_n = e^{-s_{1,0,0}+s_{1,4,2}+(n-2)s_{1,6,3}} \Pi_n \]

with \[ \Pi_n = \begin{cases} \int_{[P_n(\PP^3,L)]^{\vir}}
\exp\!\Bigl(
\sum_{k\in\{0,2,4,6\}}
\sum_{\ell\geq 2} s_{1,k,\ell}\,S^P_{1,k,\ell}(n)
\Bigr) & \text{ if } 1\leq n \leq 3,  \\ 
20(-1)^n \left([t^3]\bigl(e^{U_n(t)}V_n(t)\bigr)-(4n-10)[t^4]\bigl(e^{U_n(t)}\bigr) \right)& \text{ if } n\geq 4. \end{cases} \]

\subparagraph{$\beta=[2L]$ is covered in Section~\ref{PP32L}.}

On $X = \PP^3$ with $\beta=[2L]$, the sheaf-theoretic invariant 
\[
[\M^{ss}_{(2L,n)}(\mu_H^0)]_{\inv} = (\iota_{2L,n})_* (j_{2L,n})_* \left(
c_7(\mathscr{O}b)\cap [\Hilb_{2L}(\PP^3)] \right)
\in H_2\bigl(\mathcal N^{\pl}_{(0,2L,n)},\Q\bigr)
\]

normalizes in $\mathbb{D}$ to

\begin{align*}
[\M^{ss}_{(2L,n)}(\tau_-)]_{\inv}^{\perp} ={}&e^{(0,0,2L,n-4)}
\Biggl(
81s_{1,2,2}+(-\frac{81}{8}n^2+81n-\frac{1053}{8}) s_{1,6,4}
\Biggr).
\end{align*}

The wall-crossing formula for $[P_n(\PP^3, 2L)]^{\vir}$ in $H_{16}(\N^{\pl}_{(-1,0,2L,n-4)}, \Q) \subset H_*(\N^{\pl}, \Q)$ is given in \eqref{eqn: PTPnP32LWCF}, which simplifies to 

\[
\widetilde{\mathcal P}_{2L,n}
=
\widetilde{\mathcal E}_{2L,n}
+\frac12\,\mathbf 1_{2\mid n}\,\widetilde{\mathcal D}_{2L,n}
+\widetilde{\mathcal R}^{\mathrm{low}}_{2L,n}
+\widetilde{\mathcal R}^{\ge 4}_{2L,n} \]

in $\mathbb{D}$. 

\paragraph{$X$ a smooth cubic threefold}

\subparagraph{$\beta=[\ell]$ is covered in Section~\ref{subsec: SmoothCubicThreefold.betaell}}

For all $m\in \Z$, 

\[ \bigl[\M^{ss}_{(\ell,m)}(\tau_-)\bigr]_{\inv}
= (\iota_{\ell,m})_* \left(
-c_1\cap [F]
\right) \in H_2(\N^{\pl}_{(0,\ell,m)}, \Q). 
\]

For $n\geq 2$, we have in \eqref{cubic-eq:virtual-class-explicit},

\begin{align*}
[P_n(X,\ell)]^{\vir}
=
(-1)^{n-1}c_1\left(
\eta_n^{n-2}
+\frac{(n+1)(n-2)}{2}\,c_1\eta_n^{n-3}
\right)\cap [P_n] \in H_4(P_n(X, \ell), \Q) 
\end{align*}

We then give various PT descendent invariants and large $n$ tails of PT generating series to compare with \cite{Moreira}. 

\paragraph{$X = \Bl_p \PP^3$}

\subparagraph{$\beta = h-e$ is covered in Section~\ref{sec: beta1minus1BlpP3}}

In $\mathbb{D}$, 

\[
\left[\mathcal{M}^{\mathrm{ss}}_{(h-e,n)}(\tau_-)\right]_{\mathrm{inv}}
=
e^{(0,0,h-e,n-1)}
\bigl(
3s_{143}-3s_{243}+3(n-1)s_{164}
\bigr)
\]

for all $n\in \Z$. For $n\geq 2$, we have 

\[
[P_{n}(X,h-e)]^{\mathrm{virt}}
=
(-1)^{n-1}
e^{-s_{100}}
\exp\bigl(s_{142}-s_{242}+(n-1)s_{163}\bigr)
\]
\[
\cdot
\Bigl(
3(n-2)\bigl(s_{144}-s_{244}+(n-1)s_{165}\bigr)
+
3(2n-3)\bigl(s_{143}-s_{243}+(n-1)s_{164}\bigr)^{2}
\Bigr).
\]

\subparagraph{$\beta=e$ is covered in Section~\ref{sec: betaIse}}

In $\mathbb{D}$, for all $n\in \Z$, 

\[
\left[\mathcal{M}^{\mathrm{ss}}_{(e,n)}(\tau_-)\right]_{\mathrm{inv}}
=
e^{(0,0,e,n-1)}
\bigl(
s_{243}+(n-1)s_{164}
\bigr)
\] 

and for $n\geq 2$,

\[
[P_{n}(X,e)]^{\mathrm{virt}}
=
(-1)^{n-1}
e^{-s_{100}}
\exp\bigl(s_{242}+(n-1)s_{163}\bigr)
\]
\[
\cdot
(n-2)
\Bigl(
s_{244}+(n-1)s_{165}
+
\bigl(s_{243}+(n-1)s_{164}\bigr)^{2}
\Bigr).
\]

\subparagraph{$\beta=2e$ is covered in Section~\ref{sec: betaIs2e}}

In $\mathbb{D}$, for all $n\in \Z$, 

\[
\left[\mathcal{M}^{\mathrm{ss}}_{(2e,n)}(\tau_-)\right]_{\mathrm{inv}}
=
e^{(0,0,2e,n-2)}
\left(
-s_{243}+D^{[n]}s_{164}
\right)
\]

and for $n\geq 4$,

\[
\begin{aligned}
[P_{n}(X,2e)]^{\mathrm{virt}}
={}&(-1)^{n-1}e^{-s_{100}}
\exp\bigl(2s_{242}+(n-2)s_{163}\bigr) \\
&\cdot\Bigg[
\sum_{t=0}^{3}\frac{(-1)^{t}}{t!}\,V(3-t)
\bigl(-s_{2,4,3+t}+D^{[n]}s_{1,6,4+t}\bigr)
+2D^{[n]}V(4) \\
&\qquad +\frac{1}{3!}
\sum_{\substack{n_{1}+n_{2}=n\\ n_{1}\leq n_{2}}}
\frac{1}{1+\delta_{n_{1},n_{2}}}
\sum_{m=0}^{3}
\binom{3}{m}
\bigl(s_{2,4,3+m}+(n_{1}-1)s_{1,6,4+m}\bigr) \\
&\qquad\qquad\cdot
\bigl(s_{2,4,6-m}+(n_{2}-1)s_{1,6,7-m}\bigr)
\Bigg].
\end{aligned}
\]

\subparagraph{$\beta=h$ is covered in Section~\ref{sec: betaIsh}}

In $\mathbb{D}$, for all $n\in \Z$, we have 

\[
\left[\mathcal M^{ss}_{(h,n)}(\tau_-)\right]_{\inv}
=
e^{(0,0,h,n-2)}
\bigl(20s_{143}+20(n-1)s_{164}\bigr)
\]

and the wall-crossing formula for $[P_n(X,h)]^{\virt}$ is given by

\[
[P_{n}(X,h)]^{\mathrm{virt}}
=
-
\sum_{n_{1}+n_{2}=n}
\widetilde{U}\bigl(
(1,h-e,n_{1}),(0,e,n_{2});\tau_{+},\tau_{-}
\bigr)
\bigl[
[P_{n_{1}}(X,h-e)]^{\mathrm{virt}},
\Psi_{e,n_{2}}
\bigr]
\]
\[
\qquad
-
\sum_{n_{1}+n_{2}=n}
\widetilde{U}\bigl(
(1,e,n_{2}),(0,h-e,n_{1});\tau_{+},\tau_{-}
\bigr)
\bigl[
[P_{n_{2}}(X,e)]^{\mathrm{virt}},
\Psi_{h-e,n_{1}}
\bigr].
\]

\subparagraph{$\beta = h+e$ is given in Section~\ref{sec: betaIshPluse}}

In $\mathbb{D}$, for all $n\in \Z$, 

\[
\left[\mathcal M^{ss}_{(h+e,n)}(\tau_-)\right]_{\inv}
=
e^{(0,0,h+e,n-3)}
\bigl(2s_{143}+2s_{243}+2(n-1)s_{164}\bigr)
\]

and for $n\geq 6$, the wall-crossing formula for $[P_{n}(X,h+e)]^{\mathrm{virt}}$ is given by 

\[
[P_{n}(X,h+e)]^{\mathrm{virt}}
=
-
\sum_{n_{1}+n_{2}=n}
\widetilde{U}\bigl(
(1,h,n_{1}),(0,e,n_{2});\tau_{+},\tau_{-}
\bigr)
\bigl[
[P_{n_{1}}(X,h)]^{\mathrm{virt}},
\Psi_{e,n_{2}}
\bigr]
\]
\[
\qquad
-
\sum_{n_{1}+n_{2}=n}
\widetilde{U}\bigl(
(1,e,n_{2}),(0,h,n_{1});\tau_{+},\tau_{-}
\bigr)
\bigl[
[P_{n_{2}}(X,e)]^{\mathrm{virt}},
\Psi_{h,n_{1}}
\bigr]
\]
\[
\qquad
-
\sum_{n_{1}+n_{2}=n}
\widetilde{U}\bigl(
(1,h-e,n_{1}),(0,2e,n_{2});\tau_{+},\tau_{-}
\bigr)
\bigl[
[P_{n_{1}}(X,h-e)]^{\mathrm{virt}},
\Psi_{2e,n_{2}}
\bigr]
\]
\[
\qquad
-
\sum_{n_{1}+n_{2}=n}
\widetilde{U}\bigl(
(1,2e,n_{2}),(0,h-e,n_{1});\tau_{+},\tau_{-}
\bigr)
\bigl[
[P_{n_{2}}(X,2e)]^{\mathrm{virt}},
\Psi_{h-e,n_{1}}
\bigr]
\]
\[
\qquad
-
\sum_{n_{1}+n_{2}+n_{3}=n}
\widetilde{U}\bigl(
(1,h-e,n_{1}),
(0,e,n_{2}),
(0,e,n_{3});
\tau_{+},\tau_{-}
\bigr)
\bigl[
[
[P_{n_{1}}(X,h-e)]^{\mathrm{virt}},
\Psi_{e,n_{2}}
],
\Psi_{e,n_{3}}
\bigr].
\]

\paragraph{$X = \Bl_\ell(\PP^3)$}

\subparagraph{$\beta_2=(0,1)=HE$ is given in Section~\ref{sec: beta2HE}}

In $\mathbb{D}$, for all $n$ we have 

\[
[M^{ss}_{(\beta_2,n)}(\mu_{\omega}^0)]_{\mathrm{inv}}
=
e^{\,s_{2,4,2}+(n-\frac12)s_{1,6,3}}
\left(
-\frac32\,s_{2,4,3}
+
\left(\frac23-\frac32 n\right)s_{1,6,4}
\right)\]

and for $n\geq 1$, 

\[
\begin{aligned}
[P_n(X,\beta_2)]^{\mathrm{virt}}={}&
(-1)^{n-1}e^{-s_{1,0,0}+ \,s_{2,4,2}+(n-\frac12)s_{1,6,3}}
\left[
\left(3n-\frac{17}{6}\right)s_{2,4,3}
+
\left(3n^2-\frac{13}{3}n+\frac43\right)s_{1,6,4}
\right].
\end{aligned}
\]

\subparagraph{$\beta_1=(1,0)=H(H-E)$ is given in Section~\ref{sec: beta1H(H-E)}}

In $\mathbb{D}$, for all $n$ we have 

\[
\begin{aligned}
[M^{ss}_{(\beta_1,n)}(\mu_{\omega}^0)]_{\mathrm{inv}}
={}&e^{\,s_{1,4,2}+(n-\frac32)s_{1,6,3}}
\Biggl(
s_{1,2,2}+\left(2n-\frac72\right)s_{1,4,3} +\left(n-\frac32\right)s_{2,4,3}+
\frac{9n^2-30n+25}{6}\,s_{1,6,4}
\Biggr)
\end{aligned}
\]

and for $n\geq 3$, we have

\[
\begin{aligned}
[P_n(X,\beta_1)]^{\mathrm{virt}}
={}&(-1)^{n-1}e^{-s_{1,0,0}+s_{1,4,2}+(n-\frac32)s_{1,6,3}}\\
&\cdot\Biggl[
\frac12\bigl(U_2(n)+U_1(n)^2\bigr)Q_0(n)+U_1(n)Q_1(n)+\frac12Q_2(n)\\
&\qquad -\bigl(3n^2-10n+12\bigr)
\left(\frac16U_3(n)+\frac12U_1(n)U_2(n)+\frac16U_1(n)^3\right)
\Biggr].
\end{aligned}
\]

\subparagraph{$\beta=(1,1)=H^2=\beta_1+\beta_2$ is given in Section~\ref{sec: betaIsHSquared}}

In $\mathbb{D}$, for all $n \in \Z$, we have

\[ 
[M^{ss}_{(H^2,n)}(\mu_{\omega}^0)]_{\mathrm{inv}}
=
e^{\,s_{1,4,2}+s_{2,4,2}+(n-2)s_{1,6,3}}
\left(
20\,s_{1,4,3}+20\,s_{2,4,3}+20(n-1)\,s_{1,6,4}
\right)
\]

and for $n\geq 5, n \equiv1 \pmod 2$, we have 

\[
[P_{2m+1}(X,H^2)]^{\mathrm{virt}}
=
\mathcal{R}^{(1)}_{m,m+1}
+
\mathcal{R}^{(2)}_{m,m+1}.
\]

For $n\geq 6, n\equiv 0 \pmod 2$, we have 

\[
[P_{2m}(X,H^2)]^{\mathrm{virt}}
=
\frac12\,\mathcal{R}^{(1)}_{m,m}
+
\frac12\,\mathcal{R}^{(2)}_{m,m}
\]

in the notation of Section~\ref{sec: betaIsHSquared}. 

\paragraph{$X = \PP^1 \times \PP^1, Y = \PP(\OO_X \oplus \OO_X(-1,-1))$}

All invariants on $\PP( \OO_{\PP^1 \times \PP^1} \oplus \OO_{\PP^1 \times \PP^1}(-1,-1))$ are given in Section~\ref{sec: SheafFiberf}. 

\subparagraph{$\beta = f$}

The sheaf-theoretic invariant

\[\bigl[\mathcal{M}^{ss}_{(f,n)}(\tau_-)\bigr]_{\mathrm{inv}}
=
e^{(0,0,f,n-1)}
\Bigl(
2s_{122}+2s_{222}-4s_{143}+2(n-1)s_{243}+2(n-1)s_{343}-2(n^2-1)s_{164}
\Bigr)\]

is given in $\mathbb{D}$ for all $n\in \Z$. For $n \geq 2$,

\[
\begin{aligned}
[P_n(Y,f)]^{\mathrm{virt}}
&=
(-1)^{n-1}e^{(-1,0,f,n-1)}
\Bigl(
-2s_{122}s_{143}+2(1-n)s_{122}s_{164}-2s_{123}
\\
&\qquad
-2s_{143}s_{222}+2(1-n)s_{164}s_{222}-2s_{223}
+2(5-n^2)s_{143}^2
\\
&\qquad
-2(n-1)(2n^2-n-9)s_{143}s_{164}
+2(5-n^2)s_{144}
\\
&\qquad
+2(1-n)s_{143}s_{243}-2(n-1)^2s_{164}s_{243}+2(1-n)s_{244}
\\
&\qquad
+2(1-n)s_{143}s_{343}-2(n-1)^2s_{164}s_{343}+2(1-n)s_{344}
\\
&\qquad
-2(n-1)^2(n^2-n-4)s_{164}^2
-2(n-1)(n^2-n-4)s_{165}
\Bigr).
\end{aligned}
\]

\subparagraph{$\beta = \ell_1$}

In $\mathbb{D}$, for all $n\in \Z$ we have

\[
\bigl[\mathcal{M}^{ss}_{(\ell_1,n)}(\tau_-)\bigr]_{\mathrm{inv}}
=
e^{(0,0,\ell_1,n-\frac12)}
\left(
-s_{122}-s_{222}+s_{322}
-\frac12\,s_{243}
+\left(n-\frac12\right)s_{343}
-\frac{3n-2}{6}\,s_{164}
\right)
\]

and for $n\geq 1$,

\[
\begin{aligned}
[P_n(Y,\ell_1)]^{\mathrm{virt}}
={}&(-1)^{n-1}e^{(-1,0,\ell_1,n-\frac12)}\\
&\cdot\left(
-s_{122}-s_{222}+s_{322}
+\left(n-\frac32\right)s_{243}
+\left(n-\frac12\right)s_{343}
+\left(n^2-2n+\frac56\right)s_{164}
\right).
\end{aligned}
\]

\subparagraph{$\beta=\ell_2$}

In $\mathbb{D}$, for all $n\in \Z$ we have 

\[
\bigl[\mathcal{M}^{ss}_{(\ell_2,n)}(\tau_-)\bigr]_{\mathrm{inv}}
=
e^{(0,0,\ell_2,n-\frac12)}
\left(
-s_{122}-s_{222}+s_{322}
+\left(n-\frac12\right)s_{243}
-\frac12\,s_{343}
-\frac{3n-2}{6}\,s_{164}
\right).
\]

For $n\geq 1$, we have

\[
\begin{aligned}
[P_n(Y,\ell_2)]^{\mathrm{virt}}
={}&(-1)^{n-1}e^{(-1,0,\ell_2,n-\frac12)}\\
&\cdot\left(
-s_{122}-s_{222}+s_{322}
+\left(n-\frac12\right)s_{243}
+\left(n-\frac32\right)s_{343}
+\left(n^2-2n+\frac56\right)s_{164}
\right).
\end{aligned}
\]

\section{Recalled wall-crossing and realization setup}\label{sec:background}

This section recalls, in the notation used later, the categorical wall-crossing package from \cite{JoyceWC} and the Fano-threefold specialization, realization formalism, and rationality package from \cite{AJI}.

\subsection{One-dimensional sheaves, pairs, and stable pairs}

We briefly recall the formalism from Joyce \cite{JoyceWC} and from the author's joint paper with Joyce \cite{AJI} in the notation needed below.

\begin{definition}
Let $X$ be a smooth projective threefold. An effective class $\beta\in H_2(X,\Z)$ is called \emph{positive} if $c_1(X)\cdot\beta>0$. It is called \emph{superpositive} if every effective factor of $\beta$ is positive.
\end{definition}

\begin{remark}
If $X$ is Fano, then every effective class is superpositive.
\end{remark}

Let $\Coh_{\le 1}(X)\subset\Coh(X)$ denote the abelian category of coherent sheaves of dimension at most one. For $F\in \Coh_{\le 1}(X)$ we write
\[
[F]=(\beta,n)\in H_2(X,\Z)\oplus \Z,
\qquad
\beta=\PD(c_2(F)),\qquad n=\chi(F).
\]
Given a K\"ahler class $\omega$ and $\lambda\in H^2(X,\R)$, the slope stability condition on $\Coh_{\le 1}(X)$ is
\[
\mu^\lambda_\omega(\beta,n)=
\begin{cases}
\dfrac{n+\lambda\cdot\beta}{\omega\cdot\beta},&\beta\neq 0,\\[1ex]
\infty,&\beta=0.
\end{cases}
\]
The category of pairs is
\[
\Acat_X^{\mathrm{pair}}=
\bigl\{(F,V,\rho): F\in\Coh_{\le 1}(X),\;V\text{ finite dimensional},\;\rho:V\otimes\OO_X\to F\bigr\}.
\]
Its numerical group is $\Z\oplus H_2(X,\Z)\oplus\Z$, and the weak stability condition used in wall-crossing is
\[
\tilde\mu^{\lambda,c}_\omega(d,\beta,n)=
\begin{cases}
\mu^\lambda_\omega(\beta,n),&d=0,\\
c,&d>0.
\end{cases}
\]

A stable pair $(F,s)$ defines an object in $\Acat_X^{\mathrm{pair}}$ with $V=\C$ and $\rho=s$. The derived object is the two-term complex $[\OO_X\xrightarrow{s}F]$, and the induced obstruction theory on $P_n(X,\beta)$ is the Pandharipande--Thomas obstruction theory of \cite{PT}.

\subsection{Recursive definition of the sheaf-theoretic invariants}

The intrinsic invariants used here are defined so that they continue to make sense when semistable objects are present. Let $\tau$ be one of the stability conditions $\mu^\lambda_\omega$ on $\Coh_{\le 1}(X)$ or $\tilde\mu^{\lambda,c}_\omega$ on the pairs category. In the present Fano threefold setting we work directly in Joyce's pairs category and its projective-linear moduli stack, and we adopt the recursive construction of \cite[\S5.3, Eq.~(5.30)]{JoyceWC} in that form.

For a numerical class $\alpha$, let
\[
\Upsilon_{\alpha,N}:=(\Pi_{\M^{ss}_{\alpha}(\tau)})_*
\Bigl([
\overline{\N}^{ss}_{(\alpha,1)}(\bar\tau_0^1)]^{\vir}
\cap c_{\mathrm{top}}(T_{\overline{\N}^{ss}_{(\alpha,1)}(\bar\tau_0^1)}/\M^{ss}_{\alpha}(\tau))\Bigr),
\]
where $\overline{\N}^{ss}_{(\alpha,1)}(\bar\tau_0^1)$ is the auxiliary pairs moduli space for objects $(E,V,\rho)$ with $E$ $\tau$-semistable of class $\alpha$, $V$ a one-dimensional $\C$-vector space, and $\rho:V\otimes \OO_X(-N)\to E$ a nonzero morphism, for $N\gg 0$. After choosing $V\cong \C$, this is equivalently the moduli space of nonzero morphisms $\rho:\OO_X(-N)\to E$. Joyce's recursive identity expresses $\Upsilon_{\alpha,N}$ as a universal linear combination of iterated Lie brackets of the invariants
\[
[\M^{\ssm}_{\alpha_i}(\tau)]_{\inv}
\]
over all proper equal-slope decompositions
\[
\alpha=\alpha_1+\cdots+\alpha_m,
\qquad
m\geq 1,
\qquad
\tau(\alpha_i)=\tau(\alpha)\ \text{for all }i.
\]
Equivalently, after isolating the $m=1$ term, the class
\[
[\M^{\ssm}_{\alpha}(\tau)]_{\inv}
\]
is defined inductively from the auxiliary pairs class $\Upsilon_{\alpha,N}$ and the lower-order equal-slope contributions.

When $\tau$-stable equals $\tau$-semistable, the invariant is just the ordinary virtual class. In the strictly semistable case, the recursive step is obtained by writing down the auxiliary pairs moduli space, computing the corresponding class $\Upsilon_{\alpha,N}$, and then subtracting the lower-order contributions prescribed by Joyce's universal coefficients together with the Lie bracket on $H_*(\N^{\pl},\Q)$ for the pairs category. This is the recursive mechanism used throughout the paper to define one-dimensional sheaf invariants whenever strictly semistable sheaves occur.

\subsection{Intrinsic wall-crossing for smooth projective Fano threefolds}\label{sec:wallcrossing-fano}

The author's joint paper with Joyce \cite{AJI} proves the wall-crossing formula in the pairs category for every superpositive class. Since every effective class on a Fano threefold is superpositive, we may use the result without further hypotheses.

\begin{theorem}[Specialization of the recalled wall-crossing package]\label{thm:intrinsic-wallcrossing}\label{thm:AJ-specialization}
Let $X$ be a smooth projective Fano threefold, let $\beta\in H_2(X,\Z)$ be effective, and let $\omega,\lambda$ be as above. For each effective factor $\gamma$ of $\beta$, let $C_\gamma$ denote the large-slope constant of the author's joint paper with Joyce for the class $\gamma$. Choose an integer $N_\beta$ large enough that $n>N_\beta$ implies $\mu^\lambda_\omega(\gamma,n)>C_\gamma$ for every effective factor $\gamma$ of $\beta$. Then:
\begin{enumerate}[label={\rm(\alph*)},leftmargin=2.2em]
\item for every effective factor $\gamma$ of $\beta$ and every $n\in\Z$, the one-dimensional Donaldson--Thomas invariant
\[
[\M^{\ssm}_{(\gamma,n)}(\mu^\lambda_\omega)]_{\inv}\in H_2(\M^{\pl}_{(\gamma,n)},\Q)
\]
is defined;
\item if $\gamma$ is an effective factor of $\beta$, $n>N_\beta$, and $c_-,c_+\in\R$ satisfy
\[
C_\gamma<c_-<\mu^\lambda_\omega(\gamma,n)<c_+,
\qquad
\mu^\lambda_\omega(\gamma,n)-c_->0\text{ small},
\qquad c_+\gg0,
\]
then the pair invariant
\[
[\N^{\ssm}_{(1,\gamma,n)}(\tilde\mu^{\lambda,c_+}_\omega)]_{\inv}
\in H_{2c_1(X)\cdot\gamma}(\N^{\pl}_{(1,\gamma,n)},\Q)
\]
is defined and equals $[P_n(X,\gamma)]^{\vir}$;
\item if $n>N_\beta$, then the wall-crossing identity \eqref{eq:intrinsic-wcf} holds in the Lie algebra $H_*(\N^{\pl},\Q)$ for the class $(1,\beta,n)$;
\item if $\gamma$ is irreducible and $n>N_\beta$, then \eqref{eq:intrinsic-irred} holds.
\end{enumerate}
\end{theorem}

The identity in \rm(c), using that the $c_-$-semistable pair moduli space is empty, is
\begin{align}\label{eq:intrinsic-wcf}
0
&=
\sum
\widetilde U\bigl((0,\beta_1,n_1),\dots,(1,\beta_j,n_j),\dots,(0,\beta_k,n_k);\tilde\mu^{\lambda,c_+}_\omega,\tilde\mu^{\lambda,c_-}_\omega\bigr)
\notag\\
&\qquad\cdot
\Bigl[\cdots\bigl[[\M^{\ssm}_{(\beta_1,n_1)}(\mu^\lambda_\omega)]_{\inv},\dots,[\N^{\ssm}_{(1,\beta_j,n_j)}(\tilde\mu^{\lambda,c_+}_\omega)]_{\inv}\bigr],\dots\Bigr],
\end{align}
where the sum runs over effective decompositions $\beta=\beta_1+\cdots+\beta_k$ and integers $n=n_1+\cdots+n_k$. In \rm(d) one has
\begin{equation}\label{eq:intrinsic-irred}
[P_n(X,\gamma)]^{\vir}=
\bigl[[\M^{\ssm}_{(\gamma,n)}(\mu^\lambda_\omega)]_{\inv},[P_0(X,0)]^{\vir}\bigr].
\end{equation}

\begin{proof}
This is the Fano specialization of the sheaf-invariant construction and stable-pairs wall-crossing formulae in the author's joint paper with Joyce \cite[Theorems~2.14,~2.16]{AJI}. The only input needed for the specialization is that effective classes on a Fano threefold are superpositive.
\end{proof}

\begin{remark}
The wall-crossing is intrinsic. The correct home of the formula is the Lie algebra $H_*(\N^{\pl},\Q)$ attached to the pairs category. Any calculation in Gross's polynomial model is therefore interpreted as a realization of classes already constructed intrinsically.
\end{remark}

\subsection{Gross's polynomial realization and the coefficient package}\label{sec:polyrealization}

\subsubsection{Polynomial realization for the Fano threefolds considered here}

We use the polynomial realization coming from Friedlander--Walker and Gross \cite{FriedlanderWalker,Gross}. In the general discussion one distinguishes type $C$ and type $D$ varieties. For smooth projective Fano threefolds, the realization map to Gross's polynomial model is an isomorphism by the following proposition.

\begin{proposition}\label{prop:fano-threefold-type-C}
Every smooth projective Fano $3$-fold is of type $C$ in the sense of Friedlander--Haesemeyer--Walker and Gross.
\end{proposition}

\begin{proof}
Let $X$ be a smooth projective Fano $3$-fold over $\mathbb{C}$. By the theorem of Campana and Koll\'ar--Miyaoka--Mori, every smooth Fano variety is rationally connected; hence $X$ is a rationally connected smooth projective threefold \cite{CampanaFano,KMMFano}.

We now use Voineagu's computation of Lawson homology for rationally connected threefolds. More precisely, Voineagu proves that if $Y$ is a smooth projective rationally connected threefold, then the generalized cycle map
\[
L_tH_n(Y) \longrightarrow H_n(Y)
\]
is an isomorphism for all possible indices $(t,n)$; equivalently,
\[
L_*H_*(Y)=H_*(Y)
\]
for all possible indices \cite[Prop.~5.1]{Voineagu}. Applying this with $Y=X$, we obtain
\[
L_tH_n(X) \xrightarrow{\ \sim\ } H_n(X)
\qquad \text{for all } t\geq 0 \text{ and } n\geq 2t.
\]

By definition, $X$ is of type $C$ if the refined cycle maps
\[
L_tH_n(X) \longrightarrow \widetilde W_{-2t}H_n^{BM}(X)
\]
are isomorphisms for all integers $t,n$ \cite[Def.~6.5]{FHWComp}; Gross recalls this convention in \cite[Rem.~4.9]{Gross}. We therefore compare the target with ordinary homology.

Since $X$ is smooth and projective, Friedlander--Haesemeyer--Walker show that the refined weight groups coincide with the usual weight filtration, and Gillet--Soul\'e purity gives that $H_n^{BM}(X)=H_n(X)$ is pure of weight $-n$ \cite[Thm.~5.11(3)]{FHWComp}\cite[Thm.~5.5(3)]{FHWComp}. Hence
\[
\widetilde W_{-2t}H_n^{BM}(X)
\cong
\begin{cases}
0, & n<2t,\\[0.5ex]
H_n(X), & n\ge 2t,
\end{cases}
\qquad (t\ge 0).
\]
On the other hand, if $n<2t$ then
\[
L_tH_n(X)=\pi_{n-2t}Z_t(X)=0,
\]
while if $n\ge 2t$ the map $L_tH_n(X)\to H_n(X)$ is an isomorphism by Voineagu. Therefore, for every $t\ge 0$ and every $n$, the refined cycle map
\[
L_tH_n(X) \longrightarrow \widetilde W_{-2t}H_n^{BM}(X)
\]
is an isomorphism.

It remains to consider $t<0$. Friedlander--Haesemeyer--Walker extend Lawson homology by the convention
\[
L_tH_n(X)=L_0H_n(X) \qquad (t<0)
\]
\cite[p.~16]{FHWComp}, and in the proof of their Theorem~5.12 they note that for $t\le 0$ one has
\[
\widetilde W_{-2t}H_n^{BM}(X)=\widetilde W_0H_n^{BM}(X)=H_n^{BM}(X)=H_n(X)
\]
\cite[p.~22, proof of Thm.~5.12]{FHWComp}. Since Voineagu's result also gives
\[
L_0H_n(X)\xrightarrow{\ \sim\ }H_n(X),
\]
it follows that the refined cycle map is an isomorphism for all $t<0$ as well.

Thus the refined cycle maps
\[
L_tH_n(X) \longrightarrow \widetilde W_{-2t}H_n^{BM}(X)
\]
are isomorphisms for all integers $t,n$. Therefore $X$ is of type~$C$ in the sense of Friedlander--Haesemeyer--Walker and Gross.
\end{proof}

Thus, for the smooth projective Fano threefolds considered in this paper, we work directly with Gross's realized model.

Let $\overline\M$ denote the rigidified derived moduli stack of perfect complexes on $X$, decomposed by semi-topological $K$-theory classes
\[
\overline\M=\coprod_{\kappa\in K^{\mathrm{sst}}(X)}\overline\M_{\kappa}^{\pl}.
\]
Choose once and for all a homogeneous basis
\[
\{\eps_{jk}: 0\leq k\leq 6,\ 1\leq j\leq b^k(X)\}
\]
of $H^*(X,\Q)$, with $\eps_{jk}\in H^k(X,\Q)$, and let $\{e_{jk}\}$ be the \emph{linear dual basis} of $H_*(X,\Q)$ to $\{\eps_{jk}\}$ under the natural homology--cohomology pairing. If $\mathcal U^\bullet$ is the universal perfect complex on $\overline\M\times X$, Gross's tautological cohomology classes are defined by
\[
S_{jk\ell}:=\ch_\ell(\mathcal U^\bullet)\setminus e_{jk}\in H^{2\ell-k}(\overline\M,\Q),
\qquad \ell>\frac{k}{2},
\]
where $\setminus$ denotes the slant product.

We write
\[
\Dalg:=\Q[s_{jk\ell}: 0\leq k\leq 6,\ 1\leq j\leq b^k(X),\ \ell>k/2]
\]
for the corresponding super-polynomial algebra: the variable $s_{jk\ell}$ is even when $k$ is even and odd when $k$ is odd. We normalize the dual pairing by declaring that monomials in the classes $S_{jk\ell}$ and $s_{jk\ell}$ are dual up to the standard factorial factors, with the Koszul sign rule for odd generators. In particular, the notation
\[
e^\kappa \Dalg
\]
means the copy of $\Dalg$ attached to the connected component $\overline\M_\kappa^{\pl}$.

Gross's theorem gives natural morphisms
\[
e^\kappa \Q[S_{jk\ell}] \longrightarrow H^*(\overline\M_\kappa^{\pl},\Q),
\qquad
H_*(\overline\M_\kappa^{\pl},\Q)\longrightarrow e^\kappa \Dalg.
\]
In the present Fano threefold setting the second map is an isomorphism. Thus every intrinsic invariant may be transported to $e^\kappa \Dalg$, and all calculations with the variables $s_{jk\ell}$ should be understood as calculations in this realized model.

\subsubsection{Degree-two realized invariants}

For an effective class $\beta$ and an integer $n$, set
\[
\kappa_{\beta,n}:=\bigl(0,0,\beta,n-\tfrac12\beta\cdot c_1(TX)\bigr).
\]
Let
\[
\iota_M:\M^{\pl}_{(\beta,n)}(\mu^\lambda_\omega)\longrightarrow \overline\M_{\kappa_{\beta,n}}^{\pl}
\]
be the natural map to the rigidified derived moduli stack, and define
\[
\Psi_{\beta,n}:=(\iota_M)_*[\M^{\ssm}_{(\beta,n)}(\mu^\lambda_\omega)]_{\inv}
\in H_2(\overline\M^{\pl}_{\kappa_{\beta,n}},\Q)\subset e^{\kappa_{\beta,n}}\Dalg.
\]

Since $\Psi_{\beta,n}$ has homological degree two, only degree-two elements of $\Dalg$ can occur in its polynomial part. On a smooth projective Fano threefold one has $H^1(X,\Q)=H^5(X,\Q)=0$, so the degree-two possibilities are precisely
\[
\begin{aligned}
s_{101},\qquad &s_{j22}\ (1\le j\le b^2(X)),\qquad s_{a32}s_{b32}\ (1\le a<b\le b^3(X)),\\
&s_{k43}\ (1\le k\le b^4(X)),\qquad s_{164}.
\end{aligned}
\]
This is the origin of the coefficient package in \cref{eq:general-degree-two-ansatz}. The key point for the present paper is that the divisor-class coefficients $B^{[n]}_{j22}$ belong to the same degree as the curve-class and Euler-characteristic coefficients and therefore must be included from the outset; when $H^3(X,\Q)\neq 0$, the odd quadratic block $\sum_{a<b}E^{[n]}_{ab}s_{a32}s_{b32}$ must also be retained. Accordingly we write
\begin{equation}\label{eq:general-degree-two-ansatz}
\Psi_{\beta,n}
=
 e^{\kappa_{\beta,n}}
\left(
\sum_{j=1}^{b^2(X)}B^{[n]}_{j22}s_{j22}
+\sum_{1\le a<b\le b^3(X)}E^{[n]}_{ab}s_{a32}s_{b32}
+\sum_{k=1}^{b^4(X)}C^{[n]}_{k43}s_{k43}
+D^{[n]}s_{164}
\right).
\end{equation}

\begin{lemma}\label{prop:A-zero}
In \cref{eq:general-degree-two-ansatz} the coefficient of $s_{101}$ vanishes identically.
\end{lemma}

\begin{proof}
On the one-dimensional sheaf locus, the universal object is represented by a coherent sheaf whose support in $X\times \M^{\pl}_{(\beta,n)}$ has codimension at least $2$: every fibre over the moduli stack is a sheaf of pure dimension $1$ on the threefold $X$. Therefore its class lies in the second step of the support filtration on $K$-theory, so its Chern character has no degree-$0$ or degree-$1$ part. In particular,
\[
\ch_1(\mathcal U^\bullet)=0.
\]
Hence no $s_{101}$-term appears in the realization of $\Psi_{\beta,n}$. This argument uses only the codimension of the support and does not assume that $X\times \M^{\pl}_{(\beta,n)}$ is smooth.
\end{proof}

\begin{remark}
There is no parallel vanishing statement for the coefficients $B^{[n]}_{j22}$. In particular, the line-class calculation on $\PP^3$ below shows that $B^{[n]}_{122}$ is already nonzero for $\beta=L$. Likewise, when $H^3(X,\Q)\neq 0$ there is no general reason for the odd quadratic coefficients $E^{[n]}_{ab}$ to vanish.
\end{remark}

\subsubsection{Tensoring recurrences}

\begin{definition}\label{def:admissible-line-bundle}
Let $\beta\in H_2(X,\Z)$ be effective. A line bundle $L\to X$ is called \emph{admissible for $\beta$} if tensoring by $L$ identifies the relevant moduli spaces of semistable one-dimensional sheaves in class $\beta$ with those in class $\beta$ and Euler characteristic shifted by $c_1(L)\cdot\beta$.
\end{definition}

Let $L\to X$ be an admissible line bundle for $\beta$ with first Chern class $\gamma$. This is automatic when $L=A$ is the same ample line bundle used to define Gieseker semistability; in other situations it has to be checked case by case.

\begin{proposition}\label{prop:tensoring}
Under the above hypothesis, the coefficients in \cref{eq:general-degree-two-ansatz} satisfy
\begin{align}
B^{[n+\gamma\cdot\beta]}_{j22}&=B^{[n]}_{j22},\label{eq:tensor-B}\\
E^{[n+\gamma\cdot\beta]}_{ab}&=E^{[n]}_{ab},\label{eq:tensor-E}\\
C^{[n+\gamma\cdot\beta]}_{k43}&=C^{[n]}_{k43}+\sum_{j=1}^{b^2(X)} B^{[n]}_{j22}\,(\eps_{j2}\cup\gamma)_k,\label{eq:tensor-C}\\
D^{[n+\gamma\cdot\beta]}&=D^{[n]}+\frac12\sum_{j=1}^{b^2(X)}B^{[n]}_{j22}\int_X\eps_{j2}\cup\gamma^2+\sum_{k=1}^{b^4(X)} C^{[n]}_{k43}\int_X\eps_{k4}\cup\gamma.\label{eq:tensor-D}
\end{align}
\end{proposition}

\begin{proof}
Tensoring by $L$ sends a sheaf of class $(\beta,n)$ to a sheaf of class $(\beta,n+\gamma\cdot\beta)$ and induces Gross's change-of-variables operator $\Omega_\gamma$ on $e^\kappa\Dalg$. Expanding $\Omega_\gamma$ to degree two gives
\[
s_{j22}\longmapsto s_{j22},
\qquad
s_{a32}\longmapsto s_{a32},
\qquad
s_{k43}\longmapsto s_{k43}+\sum_{j=1}^{b^2(X)}(\eps_{j2}\cup\gamma)_k\,s_{j22},
\]
and
\[
s_{164}\longmapsto s_{164}
+\sum_{k=1}^{b^4(X)}\!\left(\int_X\eps_{k4}\cup\gamma\right)s_{k43}
+\frac12\sum_{j=1}^{b^2(X)}\!\left(\int_X\eps_{j2}\cup\gamma^2\right)s_{j22}.
\]
Here the odd degree-one variables $s_{a32}$ are fixed because $\eps_{a3}\cup\gamma\in H^5(X,\Q)=0$ on a Fano threefold. Comparing coefficients in the realized invariants before and after tensoring yields \eqref{eq:tensor-B}, \eqref{eq:tensor-E}, \eqref{eq:tensor-C}, and \eqref{eq:tensor-D}.
\end{proof}

\subsubsection{K\"ahler orthogonality and a sufficient vanishing criterion}

Let
\[
\mathfrak D=\sum_{j,k,\ell} s_{jk(\ell+1)}\frac{\partial}{\partial s_{jk\ell}}
\]
be Gross's translation operator on $\Dalg$. Since
\[
\mathfrak D\!\left(e^{\kappa_{\beta,n}}\right)
=
e^{\kappa_{\beta,n}}
\left(
\sum_{k=1}^{b^4(X)} \beta_k s_{k43}
+
\left(n-\frac12\beta\cdot c_1(TX)\right)s_{164}
\right),
\]
the coefficients of the $s_{k43}$- and $s_{164}$-terms are only defined modulo $\mathrm{Im}(\mathfrak D)$.

\begin{proposition}[K\"ahler orthogonality]\label{prop:Kahler-orthogonality}
Fix a K\"ahler class
\[
\omega=\sum_{j=1}^{b^2(X)}\omega_j\eps_{j2}\in H^2(X,\R).
\]
For every realized invariant $\Psi_{\beta,n}$ there is a unique representative modulo $\mathrm{Im}(\mathfrak D)$ such that its curve-class coefficients satisfy
\[
\sum_{k=1}^{b^4(X)} C^{[n]}_{k43}\int_X \omega\cup \eps_{k4}=0.
\]
Equivalently, the degree-four coefficient vector is orthogonal to $\omega$ under the Poincar\'e pairing.
\end{proposition}

\begin{proof}
Adding $\alpha\,\mathfrak D(e^{\kappa_{\beta,n}})$ changes the coefficients by
\[
C^{[n]}_{k43}\longmapsto C^{[n]}_{k43}+\alpha\,\beta_k,
\qquad
D^{[n]}\longmapsto D^{[n]}+\alpha\!\left(n-\frac12\beta\cdot c_1(TX)\right).
\]
Since
\[
\sum_{k=1}^{b^4(X)} \beta_k\int_X\omega\cup\eps_{k4}=\int_\beta \omega>0,
\]
there is a unique choice of $\alpha$ for which the orthogonality condition holds.
\end{proof}

\subsubsection{Ordered wall-crossing summands and iterated brackets}

We first isolate the structural feature of the pairs-category wall-crossing which governs the shape of every iterated bracket. The point is not that one somehow chooses a single nonzero-rank insertion by hand. Rather, the class being decomposed in the pairs category has first coordinate $d=1$, and positivity forces exactly one summand in each wall-crossing term to carry $d=1$.

\begin{proposition}[One distinguished pairs-type input in each direct summand]\label{prop:one-d1-input-v6}
Fix an effective class $\beta$ and an integer $n$, and consider a nonzero summand in the wall-crossing formula for
\[
[P_n(X,\beta)]^{\vir}
\in H_*(\N^{\pl}_{(1,\beta,n)},\Q).
\]
Then the corresponding ordered decomposition in the pairs category has the form
\[
(1,\beta,n)
=
(0,\beta_1,n_1)+\cdots+(0,\beta_{j-1},n_{j-1})+(1,\beta_j,n_j)
+(0,\beta_{j+1},n_{j+1})+\cdots+(0,\beta_k,n_k),
\]
with each $\beta_i$ effective. In particular, there is exactly one summand with first coordinate $1$, and all the others have first coordinate $0$.
\end{proposition}

\begin{proof}
An object of the pairs category $\Acat_X^{\mathrm{pair}}$ is a triple $(F,V,\rho)$ with class $(d,\beta,n)$, where $d=\dim_\C V\in \Z_{\ge 0}$. Thus every class in the positive cone has nonnegative first coordinate. If
\[
(1,\beta,n)=\sum_{i=1}^k (d_i,\beta_i,n_i)
\]
is a decomposition contributing to wall-crossing, then each $d_i\in \Z_{\ge 0}$ and
\[
\sum_{i=1}^k d_i=1.
\]
Therefore exactly one $d_i$ equals $1$ and all the others are $0$.
\end{proof}

The unique class with first coordinate $1$ is the unique pairs-type, or stable-pair-type, input in the corresponding direct wall-crossing summand. All remaining inputs are sheaf-theoretic invariants.

\begin{definition}[Direct ordered bracket]\label{def:direct-ordered-bracket-v6}
For classes $Z_1,\dots,Z_k$ in the Lie algebra $H_*(\N^{\pl},\Q)$ we write
\[
[Z_1,\dots,Z_k]_{\mathrm{ord}}
:=
[\cdots[[Z_1,Z_2],Z_3],\dots,Z_k]
\]
for the left-normed ordered bracket.
\end{definition}

\begin{remark}\label{rem:direct-summands-v6}
By \cref{prop:one-d1-input-v6}, every direct summand in the wall-crossing formula for $[P_n(X,\beta)]^{\vir}$ has the form
\[
[\Psi_{\beta_1,n_1},\dots,\Psi_{\beta_{j-1},n_{j-1}},\Pi_{\beta_j,n_j},
\Psi_{\beta_{j+1},n_{j+1}},\dots,\Psi_{\beta_k,n_k}]_{\mathrm{ord}},
\]
where
\[
\Psi_{\gamma,m}:=[\M^{\ssm}_{(\gamma,m)}(\mu^\lambda_\omega)]_{\inv},
\qquad
\Pi_{\gamma,m}:=[P_m(X,\gamma)]^{\vir},
\]
and where there is exactly one stable-pair input $\Pi_{\beta_j,n_j}$.
\end{remark}

We now turn to the binary bracket of two realized rank-$0$ sheaf classes in the full super-polynomial algebra. One cannot obtain this bracket by taking the endpoint formula for $[-,\delta]$ and merely allowing the coefficients $B^{[n]}_{j22}$ to be nonzero. The correct starting point is the full rank-$0$/rank-$0$ residue formula in the Lie algebra attached to $H_*(\overline\M^{\pl},\Q)$.

Let
\[
J_{\mathrm{ev}}:=\{(j,k): k\in\{0,2,4,6\},\ 1\le j\le b^k(X)\}.
\]
Write
\[
E(\kappa):=\exp\!\Bigl(\sum_{(j,k)\in J_{\mathrm{ev}}}\kappa_{jk}s_{jk(k/2)}\Bigr),
\]
so that a realized rank-$0$ class is written as
\[
\Psi=E(\kappa)\,f,
\]
where $f\in \Q[s_{jk\ell}:\ell>k/2]$ has Joyce degree at most $2$. We use the full shift derivation
\begin{equation}\label{eq:shift-derivation-v6}
\mathfrak D=
\sum_{k=0}^6\sum_{j=1}^{b^k(X)}\sum_{\ell>k/2}
 s_{jk(\ell+1)}\frac{\partial^L}{\partial s_{jk\ell}},
\end{equation}
where $\partial^L/\partial s_{jk\ell}$ denotes the left super-derivative for odd generators. For two independent sets of variables $s=(s_{jk\ell})$ and $s'=(s'_{jk\ell})$, define the full contraction operator
\begin{equation}\label{eq:rank-zero-contraction-v6}
\mathfrak C_{12}(z)
:=
\sum_{\substack{0\le k,k'\le 6,\\
1\le j\le b^k(X),\ 1\le j'\le b^{k'}(X),\\
\ell>k/2,\ \ell'>k'/2:\ \ell+\ell'>(k+k')/2}}
(-1)^{\ell+1}
\Bigl(\ell+\ell'-\frac{k+k'}{2}-1\Bigr)!
 z^{(k+k')/2-\ell-\ell'}
 N_{jk}^{j'k'}
\frac{\partial^L}{\partial s_{jk\ell}}\boxtimes
\frac{\partial^L}{\partial s'_{j'k'\ell'}}.
\end{equation}
Because the coefficients $N_{jk}^{j'k'}$ vanish unless the cohomological parities of $k$ and $k'$ agree, the nonzero terms in \eqref{eq:rank-zero-contraction-v6} involve only integral powers of $z$.

\begin{proposition}[Exact binary bracket for realized sheaf-theoretic inputs]\label{prop:rank-zero-binary-v6}
Let
\[
\Psi_a=E(\kappa_a)f_a,
\qquad
\kappa_a=\bigl(0,0,\beta_a,n_a-\tfrac12\beta_a\cdot c_1(X)\bigr),
\qquad a=1,2.
\]
Then
\[
\chi(\kappa_1,\kappa_2)=\chi(\kappa_2,\kappa_1)=0,
\]
and the Lie bracket in Gross's realization is
\begin{equation}\label{eq:rank-zero-binary-v6}
[\Psi_1,\Psi_2]
=
\Res_z\Biggl[
\Biggl\{
\exp\!\bigl(z\mathfrak D\bigr)
\circ
\exp\!\bigl(-\mathfrak C_{12}(z)\bigr)
\Bigl(E(\kappa_1)f_1\cdot E'(\kappa_2)f_2'\Bigr)
\Biggr\}\Biggr|_{s'=s}
\Biggr],
\end{equation}
where $E'(\kappa_2)$ and $f_2'$ are written in the primed variables. This formula is exact and makes no assumption on the coefficients $B^{[n]}_{j22}$ or $E^{[n]}_{ab}$.
\end{proposition}

\begin{proof}
This is the rank-$0$/rank-$0$ specialization of Joyce's general vertex-operator residue formula for the Lie bracket. Since both inputs have rank $0$, there is no reduction to the endpoint operation with $\delta=[P_0(X,0)]^{\vir}$. The vanishing of the Euler-form prefactor follows from the fact that on a smooth projective threefold the Euler pairing of two classes in $\Coh_{\le 1}(X)$ is zero.
\end{proof}

\begin{remark}\label{rem:no-positive-rank-shortcut-v6}
The formula \eqref{eq:rank-zero-binary-v6} is the one that must be used whenever two sheaf-theoretic invariants are bracketed. In particular, for decompositions with more than one rank-$0$ factor, the wall-crossing cannot be computed by repeatedly applying an endpoint formula and simply inserting the $B^{[n]}_{j22}$-terms by hand.
\end{remark}

\begin{definition}[The distinguished zero-class stable pair]\label{def:delta-v6}
Set
\[
\delta:= (\iota_N)_*[P_0(X,0)]^{\vir}=e^{-s_{100}}.
\]
The irreducible base case of the wall-crossing recursion is
\begin{equation}\label{eq:irred-base-v6}
[P_n(X,\gamma)]^{\vir}
=
[\Psi_{\gamma,n},\delta]
\end{equation}
for irreducible $\gamma$ and $n>N_\gamma$.
\end{definition}

\begin{definition}[Full recursive expansion]\label{def:full-expansion-v6}
Starting from a direct wall-crossing summand as in \cref{rem:direct-summands-v6}, we recursively substitute the wall-crossing formula for each lower-class stable-pair input $\Pi_{\gamma,m}$ as long as $m>N_\gamma$. The recursion stops either at the irreducible base case \eqref{eq:irred-base-v6} or at one of the finitely many bounded lower stable-pair inputs
\[
\Pi_{\gamma,m}=[P_m(X,\gamma)]^{\vir},
\qquad M_\gamma\le m < N_\gamma,
\]
which must be computed directly. The result is a linear combination of ordered Lie words in realized sheaf-theoretic invariants together with either the class $\delta$ or one bounded lower stable-pair input.
\end{definition}

\begin{proposition}[One distinguished nonzero-rank input after expansion]\label{prop:one-delta-v6}
Every ordered Lie word occurring in the fully expanded form of $[P_n(X,\beta)]^{\vir}$ contains exactly one nonzero-rank input. This input is either the class $\delta$ or one bounded lower stable-pair invariant $\Pi_{\gamma,m}$ with $M_\gamma\le m < N_\gamma$.
\end{proposition}

\begin{proof}
We induct on $\omega\cdot\beta$. If $\beta$ is irreducible and $n>N_\beta$, the claim is exactly \eqref{eq:irred-base-v6}. For reducible $\beta$, each direct wall-crossing summand contains exactly one stable-pair input by \cref{prop:one-d1-input-v6}. If that input is already one of the bounded lower stable-pair invariants, we are done. Otherwise it is a lower-class stable-pair invariant $\Pi_{\gamma,m}$ with $0<\omega\cdot\gamma<\omega\cdot\beta$ and $m>N_\gamma$, so the induction hypothesis applies to its recursive expansion. Substituting that expansion into the outer ordered bracket only adds sheaf-theoretic inputs, and therefore preserves the uniqueness of the nonzero-rank input.
\end{proof}

\begin{corollary}[Binary operations needed after full expansion]\label{cor:binary-after-expansion-v6}
After full recursive expansion, every binary step in an ordered Lie word involves either two realized sheaf-theoretic invariants, or one realized sheaf-theoretic invariant and an expression containing the unique nonzero-rank input. In particular, no fully expanded Lie word contains two independent nonzero-rank branches.
\end{corollary}

\begin{proof}
This is immediate from \cref{prop:one-delta-v6}.
\end{proof}

\begin{proposition}[Finiteness of the residue expansion]\label{prop:finite-residue-v6}
Assume that every sheaf-theoretic input has polynomial part of Joyce degree at most $2$ (for example a combination of the variables $s_{j22}$, $s_{a32}s_{b32}$, $s_{k43}$, and $s_{164}$) and that every lower-class stable-pair input has polynomial part of bounded Joyce degree. Then every fully expanded ordered Lie word produced by \cref{def:full-expansion-v6} is a finite sum in $\Dalg$.
\end{proposition}

\begin{proof}
At each binary step, the residue picks out the coefficient of $z^{-1}$. Every application of the contraction operator contributes a strictly negative power of $z$, while every application of the shift operator contributes a nonnegative power. Therefore only finitely many terms in the exponential expansions of the relevant binary bracket can contribute to the residue. Since the polynomial parts have bounded Joyce degree, the recursion terminates after finitely many binary steps.
\end{proof}

\subsection{Admissible periods, rationality, and poles}\label{sec:rationality}

We recall the quasi-polynomiality, rationality, and pole package from \cite{AJI} in the notation used later.

\begin{definition}
Fix the polarization $A$. For an effective class $\beta$, an \emph{admissible period} is a positive integer of the form $c_1(L)\cdot\beta$, where $L\to X$ is admissible for $\beta$ in the sense of \cref{def:admissible-line-bundle}. We write $\AdP_\beta$ for the finite set of admissible periods attached to all effective factors of $\beta$.
\end{definition}

\begin{proposition}\label{prop:quasipoly}
Let $\beta$ be effective. Then for every descendent insertion there are integers $N\gg 0$ and, for each $p\in\AdP_\beta$, finitely many polynomials $P_{p,j}(n)\in\Q[n]$ such that on each residue class $n\equiv j\pmod p$ with $n\ge N$, the Pandharipande--Thomas invariant
\[
\left\langle\prod_i\tau_{d_i}(\gamma_i)\right\rangle^X_{n,\beta}
\]
is given by one of the polynomials $P_{p,j}(n)$.
\end{proposition}

\begin{proof}
The proof is by induction on the number of effective factors of $\beta$. For irreducible $\beta$, \cref{eq:intrinsic-irred} expresses $[P_n(X,\beta)]^{\vir}$ as the bracket of $[P_0(X,0)]^{\vir}$ with the one-dimensional Donaldson--Thomas invariant in class $\beta$. The tensoring recurrences \eqref{eq:tensor-B}--\eqref{eq:tensor-D} imply that along an admissible arithmetic progression the realized Donaldson--Thomas invariant depends polynomially on $n$ of degree at most two, with the odd quadratic block governed by the periodicity relation \eqref{eq:tensor-E}. Since the Lie bracket operator is linear and differential of bounded order, the same holds for the stable-pair invariant.

For reducible $\beta$, the intrinsic wall-crossing formula \eqref{eq:intrinsic-wcf} expresses $[P_n(X,\beta)]^{\vir}$ in terms of lower-factor stable-pair classes and one-dimensional Donaldson--Thomas invariants of effective summands. The induction closes because the set of effective factors is finite.
\end{proof}

\begin{theorem}\label{thm:rationality-main}
For every effective class $\beta$ on a smooth projective Fano threefold, and for every descendent insertion, the partition function
\[
Z_{P}\!\left(X;q\,\middle|\,\prod_i\tau_{d_i}(\gamma_i)\right)_\beta
\]
is the Laurent expansion of a rational function in $\Q(q)$.
\end{theorem}

\begin{proof}
Combine \cref{prop:quasipoly} with the elementary fact that a sequence which is polynomial on finitely many arithmetic progressions has a rational generating function. Vanishing for $n\ll 0$ follows from the lower bound in \cite[\S1.3]{AJI}; thus only finitely many negative powers of $q$ occur.
\end{proof}

\begin{theorem}\label{thm:poles-main}
With the same hypotheses, the poles of the rational function from \cref{thm:rationality-main} occur only at $q=0$ and at roots of the polynomials $1-(-q)^p$ with $p\in\AdP_\beta$.
\end{theorem}

\begin{proof}
By \cref{prop:quasipoly}, each tail of the series decomposes into a finite sum of generating series of the form
\[
\sum_{m\ge 0} P(m)(-q)^{pm+j},
\]
with $P(m)$ a polynomial. Such a series is rational with denominator a power of $1-(-q)^p$. Taking the product over the finitely many periods arising from effective factors gives the claim.
\end{proof}

\begin{remark}\label{rem:bounded-lower-pt-inputs}
For reducible classes $\beta$, the wall-crossing formula expresses $[P_n(X,\beta)]^{\vir}$ for $n\gg 0$ in terms of sheaf-theoretic invariants and lower-class Pandharipande--Thomas invariants. For any fixed reducible example there are only finitely many lower-class stable-pair inputs with Euler characteristics in the bounded windows $M_\gamma\leq n_i<N_\gamma$ that have to be computed directly. Once these finitely many bounded inputs are known, they can be substituted back into the wall-crossing formula and then contribute to the complete large-$n$ tail for all sufficiently large $n$.
\end{remark}

\section{Examples}\label{sec:examples}

\paragraph{Periodicity and tensoring by line bundles.}
We must be careful about periodicity arguments in \(n\). Tensoring by an arbitrary line bundle \(L\to X\)
does \emph{not} in general preserve Gieseker semistability for a fixed polarization.
What holds in the framework of \cite{AJI} is the following: if \(\omega\in H^2(X,\mathbb R)\) is a K\"ahler class and
\(\lambda\in H^2(X,\mathbb R)\), then \cite[Remark~1.5(c)]{AJI} gives natural isomorphisms
\[
I^{pl}_L:
M^{ss}_{(\beta,n)}(\mu^\lambda_\omega)
\stackrel{\cong}{\longrightarrow}
M^{ss}_{(\beta,\;n+c_1(L)\cdot\beta)}(\mu^{\lambda-c_1(L)}_\omega),
\]
and similarly for the stable loci. Thus tensoring preserves semistability only after simultaneously
shifting \(\lambda\) to \(\lambda-c_1(L)\).

If we specialize to Gieseker stability with \(\lambda=0\) and \(\omega=c_1(A)\) for an ample line bundle \(A\),
then tensoring by \(A^{\otimes m}\) preserves the stability condition itself, because
\[
\mu^{-m\omega}_\omega(\beta,n)
=
\frac{n-m\,\omega\cdot\beta}{\omega\cdot\beta}
=
\mu^0_\omega(\beta,n)-m,
\]
and this differs from \(\mu^0_\omega\) by an additive constant only. Hence \(\mu^{-m\omega}_\omega\) and
\(\mu^0_\omega\) define the same ordering on \(C(\Coh_{\le 1}(X))\), so
\[
M^{ss}_{(\beta,n)}(\mu^0_\omega)
\cong
M^{ss}_{(\beta,\;n+m\,\omega\cdot\beta)}(\mu^0_\omega).
\]
Therefore, in the examples below, any periodicity statement at fixed Gieseker stability should be deduced
using powers of the ample line bundle \(A\) defining the polarization. If one instead tensors by a line bundle
\(L\) with \(c_1(L)\) not proportional to \(\omega\), then one must either work with the shifted stability
condition \(\mu^{\lambda-c_1(L)}_\omega\), or check separately that tensoring by \(L\) preserves
Gieseker semistability in the case at hand.

\subsection{\texorpdfstring{$X=\PP^3$}{X=P3}}
\subsubsection{\texorpdfstring{$X=\PP^3, \beta = [L]$}{X=P3, beta=L}}\label{CP3betaIsL}

For $X=\PP^3$ and $\beta=[L]$ the class of a line, every $\tau$-semistable sheaf in the moduli space $\M^{ss}_{(L,n)}(\tau)$ is stable and is of the form
\[
E\cong \OO_{\ell}(n-1)
\]
on a unique line $\ell\subset \PP^3$. Indeed, Hirzebruch--Riemann--Roch gives
\[
\chi(E)=\int_{\PP^3}\ch(E)\,\td(\PP^3)=n,
\]
so on $\ell\cong \PP^1$ the restriction has degree $n-1$; uniqueness of the support then identifies the moduli space with
\[
\M_{(L,n)}^{ss}(\tau)
\cong
\M_{(L,n)}^{st}(\tau)
\cong
G:=\Gr(2,4).
\]
Since $\ell\subset \PP^3$ is a local complete intersection with
\[
N_{\ell/\PP^3}\cong \OO_{\ell}(1)^{\oplus 2},
\]
the standard local Ext calculation yields
\[
\mathcal E xt^j(E,E)\cong \wedge^jN_{\ell/\PP^3},
\qquad j=0,1,2,
\]
independently of $n$. More canonically, if $V=\C^4$ and
$\ell=\PP(W)$ for a $2$-dimensional subspace $W\subset V$, then
\[
N_{\ell/\PP^3}\cong \OO_\ell(1)\otimes(V/W),
\qquad
\wedge^2N_{\ell/\PP^3}\cong
\OO_\ell(2)\otimes\det(V/W).
\]
Hence
\[
\Ext^0(E,E)\cong \C,
\qquad
\Ext^1(E,E)\cong \C^4,
\qquad
\Ext^2(E,E)\cong \C^3,
\qquad
\Ext^i(E,E)=0\ \, (i\ge 3).
\]
Because $H^2(\PP^3,\OO_{\PP^3})=0$, taking the trace-free part does not alter
this obstruction space.  On $G$ write the universal sequence as
\[
0\longrightarrow S\longrightarrow V\otimes\OO_G\longrightarrow Q\longrightarrow0.
\]
For the universal line $\pi\colon\mathcal C=\PP_G(S)\to G$, one has
\[
N_{\mathcal C/(\PP^3\times G)}
\cong \OO_{\mathcal C}(1)\otimes\pi^*Q.
\]
Consequently the obstruction bundle is
\[
\mathscr{O}b
\cong
\pi_*\!\left(\wedge^2N_{\mathcal C/(\PP^3\times G)}\right)
\cong \Sym^2(S^\vee)\otimes\det Q.
\]
The factor $\det Q$ is invisible after choosing a basis on any single fibre,
but it is not globally trivial on $G$ and cannot be omitted. Therefore the
Behrend--Fantechi virtual fundamental class is
\[
\bigl[\M_{(L,n)}^{ss}(\tau)\bigr]^{BF}
=
c_3\!\bigl(\Sym^2(S^\vee)\otimes\det Q\bigr)\cap [G].
\]
Using the splitting principle together with
\[
c_1(S)=-\sigma_1,
\qquad
c_2(S)=\sigma_{1,1},
\qquad
\sigma_1^3=2\sigma_{2,1},
\qquad
\sigma_1\sigma_{1,1}=\sigma_{2,1},
\]
and writing $x_1,x_2$ for the Chern roots of $S^\vee$, the roots of the
obstruction bundle are
\[
3x_1+x_2,\qquad 2x_1+2x_2,\qquad x_1+3x_2.
\]
Thus
\[
c_3\!\bigl(\Sym^2(S^\vee)\otimes\det Q\bigr)
=6\sigma_1^3+8\sigma_1\sigma_{1,1}
=20\sigma_{2,1}.
\]
Thus
\begin{align}\label{eqn: CP3L.VFundsigma12}
  \left[ \M_{(L,n)}^{ss}(\tau) \right]^{BF}
  &= 20\sigma_{2,1} \cap [G]\in \mathrm{A}_1( \Gr(2,4)).
\end{align}
In particular, the sheaf-theoretic invariant in $A_1(G)$ is independent of $n$.

\paragraph{Computing $[\M^{ss}_{(L,n)}(\mu_\omega^\lambda)]_{\inv}$ and $[P_n(\PP^3,L)]^{\virt}$ in $H_*(\mathcal{N}^{pl},\Q)$ and $H_{c_1(\PP^3).L} (P_n(\PP^3, L), \Q)$ from first principles in the intrinsic framework}

For $\iota: G=\M^{ss}_{(L,n)}(\tau) \rightarrow \N^{pl}_{(L,n)}$ denote the inclusion of the moduli space into the projective-linear stack, we have 

\[ [\M^{ss}_{(L,n)}(\tau)]_{\inv} = (\iota_{L,n})_*(20\sigma_{2,1}\cap[G]) \in H_2(\N^{pl}_{(L,n)},\Q). \]

Next, let
\[
G:=\Gr(2,4),
\]
let $S$ be the tautological rank-$2$ bundle on $G$, and let
\[
i:\mathcal C=\PP(S)\hookrightarrow \PP^3\times G
\]
be the universal line, with projections
\[
p:\mathcal C\to \PP^3,
\qquad
\pi:\mathcal C\to G.
\]
Write
\[
H:=c_1\bigl(\OO_{\PP^3}(1)\bigr),
\qquad
\xi:=p^*H\in A^1(\mathcal C),
\]
and let $Q$ be the universal quotient bundle on $G$. Then
\[
N_{\mathcal C/(\PP^3\times G)}\cong \OO_{\mathcal C}(1)\otimes \pi^*Q,
\]
so
\[
c_1\!\bigl(N_{\mathcal C/(\PP^3\times G)}\bigr)=2\xi+\sigma_1,
\qquad
c_2\!\bigl(N_{\mathcal C/(\PP^3\times G)}\bigr)=\xi^2+\xi\sigma_1+\sigma_2.
\]

For the Chern character
\[
\ch(F)=(0,0,L,n-2),
\]
the universal sheaf is
\[
\mathbb F_n:=i_*\OO_{\mathcal C}\bigl((n-1)\xi\bigr).
\]
By Grothendieck--Riemann--Roch,
\[
\ch(\mathbb F_n)
=
i_*\!\Bigl(
e^{(n-1)\xi}\,\td\!\bigl(N_{\mathcal C/(\PP^3\times G)}\bigr)^{-1}
\Bigr).
\]
Expanding in the required degrees gives
\[
\ch_2(\mathbb F_n)=i_*1,
\]
\[
\ch_3(\mathbb F_n)
=
i_*\!\Bigl((n-2)\xi-\frac{1}{2}\sigma_1\Bigr),
\]
and
\[
\ch_4(\mathbb F_n)
=
i_*\!\Bigl(
\Bigl(\frac{1}{2}n^2-2n+\frac{25}{12}\Bigr)\xi^2
+
\Bigl(-\frac{1}{2}n+\frac{13}{12}\Bigr)\xi\sigma_1
+
\frac{1}{6}\sigma_1^2-\frac{1}{12}\sigma_2
\Bigr).
\]

Since $\mathcal C=\PP(S)$, the projective bundle formula gives
\[
\pi_*(\xi)=1,
\qquad
\pi_*(\xi^2)=c_1(S^\vee)=\sigma_1,
\qquad
\pi_*(\xi\sigma_1)=\sigma_1.
\]
Therefore
\[
S_{1,2,2}^{[n]}=\pi_*(\xi^2)=\sigma_1,
\]
\[
S_{1,4,3}^{[n]}
=
\pi_*\!\Bigl(\xi\Bigl((n-2)\xi-\frac{1}{2}\sigma_1\Bigr)\Bigr)
=
\Bigl(n-\frac{5}{2}\Bigr)\sigma_1,
\]
and
\[
S_{1,6,4}^{[n]}
=
\pi_*\!\Bigl(
\Bigl(\frac{1}{2}n^2-2n+\frac{25}{12}\Bigr)\xi^2
+
\Bigl(-\frac{1}{2}n+\frac{13}{12}\Bigr)\xi\sigma_1
\Bigr)
=
\Bigl(\frac{1}{2}n^2-\frac{5}{2}n+\frac{19}{6}\Bigr)\sigma_1.
\]

\paragraph{Lifting $[\M^{ss}_{(L,n)}(\tau)]_{\inv}$ to $\mathbb{D}$}
By Section~13.1.1,
\[
\bigl[M^{ss}_{(L,n)}(\tau_-)\bigr]^{BF}
=
20\sigma_{2,1}\cap [G].
\]
Using
\[
\sigma_1\sigma_{2,1}=\sigma_{2,2},
\qquad
\int_G \sigma_{2,2}=1,
\]
we obtain
\[
B^{[n]}
=
\int_G 20\sigma_{2,1}\,S_{1,2,2}^{[n]}
=
20,
\]
\[
C^{[n]}
=
\int_G 20\sigma_{2,1}\,S_{1,4,3}^{[n]}
=
20n-50,
\]
and
\[
D^{[n]}
=
\int_G 20\sigma_{2,1}\,S_{1,6,4}^{[n]}
=
10n^2-50n+\frac{190}{3}.
\]

Hence the natural lift to $\mathbb D$ is
\begin{equation}\label{eq:P3-line-natural-sheaf}
\Psi_{L,n}^{\mathrm{nat}}
:=
\bigl[M^{ss}_{(L,n)}(\tau_-)\bigr]_{\mathrm{inv}}^{\mathrm{nat}}
=
e^{(0,0,L,n-2)}
\Bigl(
20s_{1,2,2}
+
(20n-50)s_{1,4,3}
+
\Bigl(10n^2-50n+\frac{190}{3}\Bigr)s_{1,6,4}
\Bigr).
\end{equation}

We now impose the K\"ahler-orthogonality normalization with respect to $H$ on the
$H^4(\PP^3)$-component. Since $H^4(\PP^3,\Q)=\Q\cdot H^2$, this amounts to requiring the
$s_{1,4,3}$-coefficient to vanish. Let $\mathcal D_{\mathrm{tr}}$ denote the translation operator.
Then
\[
\mathcal D_{\mathrm{tr}}\!\bigl(e^{(0,0,L,n-2)}\bigr)
=
e^{(0,0,L,n-2)}\bigl(s_{1,4,3}+(n-2)s_{1,6,4}\bigr),
\]
and therefore
\[
\Psi_{L,n}^{\perp}
:=
\Psi_{L,n}^{\mathrm{nat}}
-
(20n-50)\,\mathcal D_{\mathrm{tr}}\!\bigl(e^{(0,0,L,n-2)}\bigr)
=
e^{(0,0,L,n-2)}
\Bigl(
20s_{1,2,2}
+
\Delta_n\, s_{1,6,4}
\Bigr),
\]
where
\[
\Delta_n
=
10n^2-50n+\frac{190}{3}-(20n-50)(n-2)
=
-10n^2+40n-\frac{110}{3}.
\]
Thus, after K\"ahler-orthogonal normalization,
\begin{equation}\label{eq:P3-line-normalized-sheaf}
\bigl[M^{ss}_{(L,n)}(\tau_-)\bigr]_{\mathrm{inv}}^{\perp}
=
e^{(0,0,L,n-2)}
\Bigl(
20s_{1,2,2}
+
\Bigl(-10n^2+40n-\frac{110}{3}\Bigr)s_{1,6,4}
\Bigr).
\end{equation}

\paragraph{The pairs-category wall-crossing for the line class on $\PP^3$} \label{P3LPTinvts}

Let $X=\mathbb{P}^3$, let $\beta=L\in H_2(X,\mathbb{Z})$ be the class of a line,
and let
\[
G:=\operatorname{Gr}(2,4).
\]
Write $S$ for the tautological rank-$2$ bundle on $G$, and let
\[
\pi\colon \mathcal{C}=\mathbb{P}(S)\longrightarrow G
\]
be the universal line. We use the convention that $\mathbb{P}(E)$ parameterizes
one-dimensional subspaces of the fibres of $E$. Set
\[
E_n:=\pi_*\mathcal{O}_{\mathcal{C}}(n-1)\cong \operatorname{Sym}^{n-1}(S^{\vee}),
\qquad
B_n:=\mathbb{P}(E_n),
\]
with projection $q\colon B_n\to G$, and write
\[
H:=c_1\bigl(\mathcal{O}_{B_n}(1)\bigr),
\qquad
\xi:=c_1\bigl(\mathcal{O}_{\mathcal{C}}(1)\bigr).
\]
We denote by $\mathrm{pt}\in H^6(X,\mathbb{Q})$ the class of a point.

Since every semistable sheaf of class $(L,n)$ is $\mathcal{O}_{\ell}(n-1)$ on a
unique line $\ell\subset \mathbb{P}^3$, we have
\[
M^{ss}_{(L,n)}(\mu^\lambda_\omega)\cong G.
\]
Moreover,
\[
H^1\bigl(\mathcal{O}_{\ell}(n-1)\otimes K_{\mathbb{P}^3}\bigr)
=H^1\bigl(\mathcal{O}_{\mathbb{P}^1}(n-5)\bigr)=0
\quad\text{for }n\geq 4,
\]
so in the notation of \cite{AJI} one may take $B_L=3$. Since $L$ is
irreducible, equation~\cite[(2.30)]{AJI} gives, for $n\geq 4$,
\[
[P_n(\mathbb{P}^3,L)]^{\mathrm{vir}}
=
\bigl[
[M^{ss}_{(L,n)}(\mu^\lambda_\omega)]^{\mathrm{inv}},\,[P_0(\mathbb{P}^3,0)]^{\mathrm{vir}}
\bigr].
\]

\begin{prop}
For $n\geq 1$ there is a natural identification
\[
P_n(\mathbb{P}^3,L)\cong B_n.
\]
Let
\[
r\colon B_n\times_G \mathcal C\to B_n,
\qquad
p\colon B_n\times_G \mathcal C\to \mathcal C,
\]
and let
\[
\mathcal D_n\subset B_n\times_G\mathcal C
\]
be the universal effective divisor of degree $n-1$ on the universal line. Then the Pandharipande--Thomas obstruction bundle is
\[
\mathrm{Obs}_n
\cong
r_*\!\left(\OO_{\mathcal D_n}(2)\otimes(q\circ r)^*\det Q\right),
\]
of rank $n-1$, and
\[
[P_n(\mathbb{P}^3,L)]^{\mathrm{vir}}=c_{n-1}(\mathrm{Obs}_n)\cap [B_n].
\]
Set
\[
\mathscr E:=\Sym^2(S^\vee)\otimes\det Q.
\]
More explicitly: \label{P_nP31leqnleq3}
\begin{enumerate}[label={\rm(\alph*)},leftmargin=2.2em]
\item $n=1$: $\mathcal D_1=\varnothing$, so $\mathrm{Obs}_1=0$ and
\[
[P_1(\mathbb{P}^3,L)]^{\mathrm{vir}}=[G].
\]
\item $n=2$: there is a short exact sequence
\[
0\to q^*(S^{\vee}\otimes\det Q)\otimes \OO_{B_2}(-1)
\to q^*\mathscr E
\to \mathrm{Obs}_2
\to 0,
\]
so
\[
[P_2(\mathbb{P}^3,L)]^{\mathrm{vir}}
=c_1(\mathrm{Obs}_2)\cap [B_2]
=\bigl(2H+3q^*\sigma_1\bigr)\cap [B_2].
\]
\item $n=3$: there is a short exact sequence
\[
0\to q^*\det Q\otimes\OO_{B_3}(-1)
\to q^*\mathscr E
\to \mathrm{Obs}_3
\to 0,
\]
so
\[
[P_3(\mathbb{P}^3,L)]^{\mathrm{vir}}
=c_2(\mathrm{Obs}_3)\cap [B_3]
=\bigl(H^2+4Hq^*\sigma_1+6q^*\sigma_1^2+4q^*\sigma_{1,1}\bigr)\cap[B_3].
\]
\item $n=4$: one has
\[
\mathrm{Obs}_4\cong q^*\mathscr E,
\qquad
[P_4(\mathbb P^3,L)]^{\mathrm{vir}}
=20q^*\sigma_{2,1}\cap[B_4].
\]
\item $n\geq 5$: there is a short exact sequence
\[
0\to q^*\mathscr E
\to \mathrm{Obs}_n
\to q^*\Sym^{n-5}(S)\otimes \OO_{B_n}(-1)
\to 0,
\]
whence
\[
[P_n(\mathbb{P}^3,L)]^{\mathrm{vir}}
=
20\,q^*(\sigma_{2,1})\cap
c_{n-4}\bigl(q^*\Sym^{n-5}(S)\otimes \mathcal{O}_{B_n}(-1)\bigr)
\cap [B_n].
\]
\end{enumerate}
Equivalently, if $W_4=0$ and $W_n=\Sym^{n-5}(S)$ for $n\geq5$, then
for every $n\geq4$,
\[
[P_n(\mathbb P^3,L)]^{\mathrm{vir}}
=20q^*\sigma_{2,1}\cap
c_{n-4}\bigl(q^*W_n\otimes\OO_{B_n}(-1)\bigr)\cap[B_n].
\]
\end{prop}

\begin{proof}
A stable pair in class $L$ is supported on a unique line $\ell\subset \mathbb{P}^3$ and is of the form
\[
(\mathcal{O}_{\ell}(D),s_D)
\]
for an effective divisor $D\subset \ell\cong \mathbb{P}^1$ of degree $n-1$. Hence
\[
P_n(\mathbb{P}^3,L)\cong \mathbb{P}(E_n)=B_n.
\]
The tautological line subbundle
\[
\mathcal{O}_{B_n}(-1)\subset q^*E_n=r_*p^*\OO_{\mathcal C}(n-1)
\]
induces the universal section
\[
r^*\mathcal{O}_{B_n}(-1)\longrightarrow p^*\mathcal{O}_{\mathcal C}(n-1),
\]
whose zero scheme is the universal divisor $\mathcal D_n$.
Let
\[
j\colon B_n\times_G\mathcal C\hookrightarrow\PP^3\times B_n
\]
be the natural closed immersion. The corresponding universal stable pair has
universal sheaf
\[
\mathbb F_n=j_*\bigl(p^*\OO_{\mathcal C}(n-1)\otimes r^*\OO_{B_n}(1)\bigr).
\]
For a stable pair on $\ell=\PP(W)$, the local obstruction space is canonically
\[
H^0\!\left(\OO_D(2)\right)\otimes\det(V/W).
\]
Thus its globalization retains the transverse determinant:
\[
\mathrm{Obs}_n
\cong
r_*\!\left(\OO_{\mathcal D_n}(2)\otimes(q\circ r)^*\det Q\right).
\]
Now on $B_n\times_G\mathcal C$ one has the exact sequence
\[
0\to \OO(2-\mathcal D_n)\otimes(q\circ r)^*\det Q
\to \OO(2)\otimes(q\circ r)^*\det Q
\to \OO_{\mathcal D_n}(2)\otimes(q\circ r)^*\det Q
\to 0,
\]
where
\[
\OO(\mathcal D_n)=p^*\OO_{\mathcal C}(n-1)\otimes r^*\OO_{B_n}(1).
\]
Thus
\[
\OO(2-\mathcal D_n)\otimes(q\circ r)^*\det Q
=p^*\OO_{\mathcal C}(3-n)\otimes r^*\OO_{B_n}(-1)
\otimes(q\circ r)^*\det Q.
\]
Pushing forward along $r$ yields the cases above. For $n=1$ the divisor is
empty. For $n=2$ the left-hand pushforward is
$q^*(S^\vee\otimes\det Q)\otimes\OO_{B_2}(-1)$ and $R^1r_*=0$.
For $n=3$ it is $q^*\det Q\otimes\OO_{B_3}(-1)$ and again $R^1r_*=0$.
For $n=4$ both direct images of the left-hand term vanish. For $n\geq5$,
the direct image vanishes and relative Serre duality on the fibres $\PP^1$
gives
\[
R^1\pi_*\OO_{\mathcal C}(3-n)
\cong \Sym^{n-5}(S)\otimes\det S.
\]
After tensoring by $\det Q$, the determinant factor cancels because
$\det S\otimes\det Q\cong\det V\otimes\OO_G$ is constant. Hence
\[
R^1r_*\!\left(
\OO(2-\mathcal D_n)\otimes(q\circ r)^*\det Q
\right)
\cong q^*\Sym^{n-5}(S)\otimes\OO_{B_n}(-1).
\]
This proves the displayed exact sequences. Finally,
\[
[P_n(\mathbb{P}^3,L)]^{\mathrm{vir}}=c_{n-1}(\mathrm{Obs}_n)\cap[B_n],
\]
and one computes as above that
\[
c_3\bigl(q^*\mathscr E\bigr)=20q^*(\sigma_{2,1}).
\]
\end{proof}

\paragraph{The bounded line-class stable-pair inputs}
The three classes
\[
[P_1(\mathbb{P}^3,L)]^{\mathrm{vir}},\qquad [P_2(\mathbb{P}^3,L)]^{\mathrm{vir}},\qquad [P_3(\mathbb{P}^3,L)]^{\mathrm{vir}}
\]
are the bounded lower line-class inputs which reappear in the reducible class $2L$ wall-crossing. They are therefore part of the explicit data needed to compute the $2L$-class Pandharipande--Thomas invariants for all large $n$.

The next step is to compute the large-$n$ tails for the descendants
$(\tau_0(\mathrm{pt}))^2$, $\tau_2(\mathrm{pt})$, and $\tau_5(1)$.

On $G$ let $Q$ be the universal quotient bundle, and write $\sigma_1=c_1(Q)$,
$\sigma_2=c_2(Q)$. The normal bundle of
$j\colon B_n\times_G\mathcal{C}\hookrightarrow \mathbb{P}^3\times B_n$ is
\[
N_j\cong p^*\mathcal{O}_{\mathcal{C}}(1)\otimes (q\circ r)^*Q,
\]
so
\[
c_1(N_j)=2\xi+q^*\sigma_1,
\qquad
c_2(N_j)=\xi^2+q^*\sigma_1\,\xi+q^*\sigma_2.
\]
Grothendieck--Riemann--Roch gives
\[
\operatorname{ch}(\mathbb{F}_n)
=
j_*\Bigl(e^{(n-1)\xi+H}\,\operatorname{td}(N_j)^{-1}\Bigr).
\]
Set
\[
A_n:=e^{(n-1)\xi+H}\,\operatorname{td}(N_j)^{-1},
\]
and write $[A_n]_d$ for its codimension-$d$ component.
Also, for the projective bundle $r\colon B_n\times_G\mathcal{C}\to B_n$,
\[
r_*(\xi)=1,
\qquad
r_*(\xi^2)=q^*\sigma_1,
\qquad
r_*(\xi^3)=q^*\sigma_2.
\]

\begin{lemma}
For the universal stable-pairs sheaf $\mathbb{F}_n$ on $\mathbb{P}^3\times B_n$,
one has
\[
\tau_0(\mathrm{pt})=q^*\sigma_2,
\]
and
\[
\tau_5(1)
=
\frac{n-2}{24}H^4
+
\frac{3n^2-15n+19}{36}H^3q^*\sigma_1
+(
\text{terms of }H\text{-degree }\leq 2).
\]
Moreover, $\tau_2(\mathrm{pt})$ has $H$-degree at most $2$.
\end{lemma}

\begin{proof}
By definition,
\[
\tau_k(\gamma)=\pi_{P*}\bigl(\pi_X^*(\gamma)\cdot \operatorname{ch}_{2+k}(\mathbb{F}_n)\bigr).
\]
Since $\operatorname{ch}_2(\mathbb{F}_n)=j_*1$, we obtain
\[
\tau_0(\mathrm{pt})
=r_*(\xi^3)=q^*\sigma_2.
\]

For $\tau_2(\mathrm{pt})$, one uses
\[
\tau_2(\mathrm{pt})=r_*\bigl(\xi^3[A_n]_2\bigr).
\]
The codimension-$2$ term in
$A_n$ has $H$-degree at most $2$, and therefore so does
$\tau_2(\mathrm{pt})$.

For $\tau_5(1)$, one computes from the codimension-$5$ term in
$A_n$ that
\[
\tau_5(1)
=
r_*\bigl([A_n]_5\bigr)
=
\frac{n-2}{24}H^4
+
\frac{3n^2-15n+19}{36}H^3q^*\sigma_1
+(
\text{terms of }H\text{-degree }\leq 2),
\]
as claimed.
\end{proof}

\begin{lemma}
Let
\[
W_4:=0,
\qquad
W_n:=\operatorname{Sym}^{n-5}(S)\quad(n\geq5),
\qquad
r_n:=n-4.
\]
Then
\[
\operatorname{rk}(W_n)=r_n,
\qquad
c_1(E_n)=\frac{n(n-1)}{2}\sigma_1,
\qquad
c_1(W_n)=-\frac{(n-5)(n-4)}{2}\sigma_1.
\]
Moreover,
\[
q_*(H^{n-1})=1,
\qquad
q_*(H^n)=-c_1(E_n),
\]
and
\[
q_*\Bigl(H^3c_{n-4}(q^*W_n\otimes \mathcal{O}_{B_n}(-1))\Bigr)
=
(-1)^{n-4},
\]
\[
q_*\Bigl(H^4c_{n-4}(q^*W_n\otimes \mathcal{O}_{B_n}(-1))\Bigr)
=
(-1)^{n-4}(10-4n)\sigma_1.
\]
\end{lemma}

\begin{proof}
The rank and first-Chern-class assertions are immediate for $n=4$. For
$n\geq5$, since $\operatorname{rk}(\operatorname{Sym}^m(S))=m+1$, we have
$\operatorname{rk}(W_n)=n-4$. If $x,y$ are the Chern roots of $S^{\vee}$, then
\[
c_1(E_n)=\sum_{j=0}^{n-1} (jx+(n-1-j)y)=\frac{n(n-1)}{2}(x+y)
=\frac{n(n-1)}{2}\sigma_1,
\]
and similarly
\[
c_1(W_n)=-\frac{(n-5)(n-4)}{2}\sigma_1.
\]

For the projective bundle $q\colon B_n=\mathbb{P}(E_n)\to G$, one has
\[
q_*(H^{n-1})=1,
\qquad
q_*(H^n)=s_1(E_n)=-c_1(E_n).
\]
For $n=4$, one has $c_0=1$, and hence
\[
q_*(H^3)=1,
\qquad
q_*(H^4)=-c_1(E_4)=-6\sigma_1=(10-4n)\sigma_1.
\]
This proves the two pushforward identities in that case. For $n\geq5$, one has
\[
c_{n-4}(q^*W_n\otimes \mathcal{O}_{B_n}(-1))
=
(-1)^{n-4}\Bigl(H^{n-4}-q^*c_1(W_n)H^{n-5}+\cdots\Bigr).
\]
Multiplying by $H^3$ and $H^4$, only the displayed terms can reach the powers
$H^{n-1}$ and $H^n$ needed to survive the pushforward $q_*$. Hence
\[
q_*\Bigl(H^3c_{n-4}(q^*W_n\otimes \mathcal{O}_{B_n}(-1))\Bigr)
=
(-1)^{n-4}q_*(H^{n-1})
=
(-1)^{n-4},
\]
and
\begin{align*}
q_*\Bigl(H^4c_{n-4}(q^*W_n\otimes \mathcal{O}_{B_n}(-1))\Bigr)
&=
(-1)^{n-4}\Bigl(q_*(H^n)-c_1(W_n)q_*(H^{n-1})\Bigr) \\
&=
(-1)^{n-4}\bigl(-c_1(E_n)-c_1(W_n)\bigr) \\
&=
(-1)^{n-4}(10-4n)\sigma_1.
\end{align*}
\end{proof}

\begin{prop}
For $n\geq 4$, the Pandharipande--Thomas stable-pair invariants are
\[
\langle(\tau_0(\mathrm{pt}))^2\rangle^{PT}_{n,L}=0,
\qquad
\langle\tau_2(\mathrm{pt})\rangle^{PT}_{n,L}=0,
\]
and
\[
\langle\tau_5(1)\rangle^{PT}_{n,L}
=
5(-1)^{n+1}\frac{3n^2-12n+11}{9}.
\]
Consequently, the corresponding large-$n$ tails are
\[
\sum_{n\geq 4} \langle(\tau_0(\mathrm{pt}))^2\rangle^{PT}_{n,L}q^n=0,
\qquad
\sum_{n\geq 4} \langle\tau_2(\mathrm{pt})\rangle^{PT}_{n,L}q^n=0,
\]
and
\[
\sum_{n\geq 4} \langle\tau_5(1)\rangle^{PT}_{n,L}q^n
=
-\frac{5q^4(2q^2+7q+11)}{9(1+q)^3}.
\]
\end{prop}

\begin{proof}
We integrate the corresponding descendent classes against
\[
[P_n(\mathbb{P}^3,L)]^{\mathrm{vir}}
=
20\,q^*(\sigma_{2,1})\cap
c_{n-4}(q^*W_n\otimes \mathcal{O}_{B_n}(-1))\cap [B_n].
\]

For $(\tau_0(\mathrm{pt}))^2$, the insertion is $q^*(\sigma_2^2)$ and has $H$-degree
$0$. Since $c_{n-4}(q^*W_n\otimes \mathcal{O}_{B_n}(-1))$ has largest $H$-power
$H^{n-4}$, the product never reaches $H^{n-1}$, so its pushforward under $q_*$
vanishes. Thus
\[
\langle(\tau_0(\mathrm{pt}))^2\rangle^{PT}_{n,L}=0.
\]

For $\tau_2(\mathrm{pt})$, the previous lemma shows that the insertion has
$H$-degree at most $2$. Again, after multiplying by
$c_{n-4}(q^*W_n\otimes \mathcal{O}_{B_n}(-1))$, the largest possible power of $H$
is at most $H^{n-2}$, so the pushforward under $q_*$ vanishes. Hence
\[
\langle\tau_2(\mathrm{pt})\rangle^{PT}_{n,L}=0.
\]

For $\tau_5(1)$, only the displayed $H^4$- and $H^3q^*\sigma_1$-terms can
contribute to the pushforward. Therefore
\begin{align*}
\langle\tau_5(1)\rangle^{PT}_{n,L}
&=
20\int_G \sigma_{2,1}\cdot q_*\Bigl(
\tau_5(1)\,c_{n-4}(q^*W_n\otimes \mathcal{O}_{B_n}(-1))
\Bigr) \\
&=
20(-1)^{n-4}\int_G \sigma_{2,1}\cdot
\left(
\frac{n-2}{24}(10-4n)\sigma_1
+
\frac{3n^2-15n+19}{36}\sigma_1
\right).
\end{align*}
Since $\int_G \sigma_{2,1}\sigma_1=1$, this simplifies to
\[
\langle\tau_5(1)\rangle^{PT}_{n,L}
=
5(-1)^{n+1}\frac{3n^2-12n+11}{9}.
\]
Summing the geometric series and its derivatives yields
\[
\sum_{n\geq 4} \langle\tau_5(1)\rangle^{PT}_{n,L}q^n
=
-\frac{5q^4(2q^2+7q+11)}{9(1+q)^3}.
\]
\end{proof}

\begin{remark}[The localization check at $n=4$]
The corrected obstruction class and the corrected descendent normalization give
\[
20\left(-\frac{6}{12}+\frac{7}{36}\right)=-\frac{55}{9},
\]
which is the coefficient of $q^4$ in Pandharipande's exact series below. Thus
the factor $20$ in the virtual class is compatible with the localization
calculation: it is paired with the absence of an extraneous factor $5$ in
$\tau_5(1)$.
\end{remark}

\begin{corollary}
The pairs-category wall-crossing determines the exact generating series
only up to Laurent polynomial correction terms coming from the finitely many
exceptional values $n\leq 3$. Thus there exist Laurent polynomials
$R_0(q,q^{-1})$, $R_2(q,q^{-1})$, and $R_5(q,q^{-1})$ such that
\[
Z_P\bigl(\mathbb{P}^3;q\mid (\tau_0(\mathrm{pt}))^2\bigr)_L
=
R_0(q,q^{-1}),
\]
\[
Z_P\bigl(\mathbb{P}^3;q\mid \tau_2(\mathrm{pt})\bigr)_L
=
R_2(q,q^{-1}),
\]
\[
Z_P\bigl(\mathbb{P}^3;q\mid \tau_5(1)\bigr)_L
=
R_5(q,q^{-1})
-
\frac{5q^4(2q^2+7q+11)}{9(1+q)^3}.
\]
In particular, the approach predicts a quadratic large-$n$ tail for the
$\tau_5(1)$-series.
\end{corollary}

\begin{prop}[Comparison with Pandharipande for $\tau_5(1)$ in class $L$]
Pandharipande's exact series \cite{PandharipandeDesc} is
\[
Z_P\bigl(\PP^3;q\mid\tau_5(1)\bigr)_{L}
=
\frac{-2q-q^2+31q^3-31q^4+q^5+2q^6}{18(1+q)^3}.
\]
The large-$n$ tail computed above is
\[
Z_P^{\ge 4}\bigl(\PP^3;q\mid\tau_5(1)\bigr)_{L}
=
-\frac{5q^4(2q^2+7q+11)}{9(1+q)^3}.
\]
Hence
\[
Z_P\bigl(\PP^3;q\mid\tau_5(1)\bigr)_{L}
-
Z_P^{\ge 4}\bigl(\PP^3;q\mid\tau_5(1)\bigr)_{L}
=
\frac{q(22q^2+5q-2)}{18},
\]
which is a Laurent polynomial.
\end{prop}

\begin{remark}
This subsection does not prioritize the polynomial Lie algebra
$\mathbb{D}=e^\alpha\mathbb{Q}[s_{j,k,\ell}]$. In the approach of the author and Joyce\cite{AJI}, the natural objects are the homology classes
\[
[M^{ss}_{(L,n)}(\mu^\lambda_\omega)]^{\mathrm{inv}}
=20\sigma_{2,1}\cap [G]
\]
and, for $n\geq4$,
\[
[P_n(\mathbb{P}^3,L)]^{\mathrm{vir}}
=
20\,q^*(\sigma_{2,1})\cap
c_{n-4}(q^*W_n\otimes \mathcal{O}_{B_n}(-1))\cap [B_n],
\]
in the homology of the moduli spaces themselves.
\end{remark}

\paragraph{Computing the line-class stable-pair invariant in Gross's algebra $\mathbb D$} \label{PTbetaLGross}

The sheaf-theoretic lift in Gross' algebra is
\[
\widetilde{\mathcal M}^{\mathrm{sh}}_n
=
e_{\rho_n}
\left(
20s_{1,2,2}
+
(20n-50)s_{1,4,3}
+
\left(10n^2-50n+\frac{190}{3}\right)s_{1,6,4}
\right).
\]

For the stable-pairs sector, let
\[
r\colon B_n\times_G\mathcal C\to B_n,
\qquad
p\colon B_n\times_G\mathcal C\to \mathcal C,
\]
and let
\[
j\colon B_n\times_G\mathcal C\hookrightarrow \PP^3\times B_n
\]
be the natural closed immersion. Define
\[
H:=c_1\bigl(\OO_{B_n}(1)\bigr),
\qquad
\xi:=c_1\bigl(\OO_{\mathcal C}(1)\bigr).
\]
The universal stable-pairs sheaf on $\PP^3\times B_n$ is
\[
\mathbb F_n
=
j_*\!\bigl(p^*\OO_{\mathcal C}(n-1)\otimes r^*\OO_{B_n}(1)\bigr),
\]
and the universal stable pair is
\[
\mathbb{I}_n^\bullet=
\bigl[\OO_{\PP^3\times B_n}\to \mathbb F_n\bigr].
\]
For $\ell\geq 0$ and $k\in\{0,2,4,6\}$, set
\[
S^P_{1,k,\ell}(n):=\operatorname{ch}_{\ell}(\mathbb{I}_n^\bullet)\mathbin{\backslash} e_{1,k}.
\]

Modulo classes of the form $q^*\alpha$ with $\operatorname{codim}(\alpha)\geq 2$
(which pair trivially with $q^*(\sigma_{2,1})$ in the virtual class), one has
\[
S^P_{1,0,\ell}(n)\equiv 0
\qquad (\ell\geq 2),
\]
and, for $r=0,1,2,3$,
\[
S^P_{1,2,2+r}(n)
\equiv
\frac{H^r}{r!}\,q^*(\sigma_1),
\]
\[
S^P_{1,4,3+r}(n)
\equiv
\frac{H^{r+1}}{(r+1)!}
+
\left(n-\frac52\right)
\frac{H^r}{r!}\,q^*(\sigma_1),
\]
\[
S^P_{1,6,4+r}(n)
\equiv
(n-2)\frac{H^{r+1}}{(r+1)!}
+
\frac{3n^2-15n+19}{6}
\frac{H^r}{r!}\,q^*(\sigma_1).
\]
The degree-zero classes are
\[
S^P_{1,0,0}(n)=-1,
\qquad
S^P_{1,4,2}(n)=1,
\qquad
S^P_{1,6,3}(n)=n-2.
\]
Therefore
\[
\widetilde{\mathcal P}_n
=
e^{-s_{1,0,0}}e^{\,s_{1,4,2}+(n-2)s_{1,6,3}}\,\Pi_n,
\]
where $\Pi_n$ is homogeneous of Joyce degree $8$.

Define
\[
u_1:=s_{1,4,3}+(n-2)s_{1,6,4},
\qquad
u_2:=s_{1,4,4}+(n-2)s_{1,6,5},
\]
\[
u_3:=s_{1,4,5}+(n-2)s_{1,6,6},
\qquad
u_4:=s_{1,4,6}+(n-2)s_{1,6,7},
\]
and
\[
v_0:=s_{1,2,2}+\left(n-\frac52\right)s_{1,4,3}+\frac{3n^2-15n+19}{6}s_{1,6,4},
\]
\[
v_1:=s_{1,2,3}+\left(n-\frac52\right)s_{1,4,4}+\frac{3n^2-15n+19}{6}s_{1,6,5},
\]
\[
v_2:=s_{1,2,4}+\left(n-\frac52\right)s_{1,4,5}+\frac{3n^2-15n+19}{6}s_{1,6,6},
\]
\[
v_3:=s_{1,2,5}+\left(n-\frac52\right)s_{1,4,6}+\frac{3n^2-15n+19}{6}s_{1,6,7}.
\]
Set
\[
U_n(t):=u_1t+u_2\frac{t^2}{2}+u_3\frac{t^3}{6}+u_4\frac{t^4}{24},
\qquad
V_n(t):=v_0+v_1t+v_2\frac{t^2}{2}+v_3\frac{t^3}{6}.
\]
Then
\[
\Pi_n
=
20(-1)^n
\left(
[t^3]\bigl(e^{U_n(t)}V_n(t)\bigr)
-
(4n-10)[t^4]\bigl(e^{U_n(t)}\bigr)
\right),
\]
where $[t^m](\cdots)$ denotes the coefficient of $t^m$.
Expanding these coefficients gives the full simplified form
\[
\Pi_n
=
20(-1)^n\Biggl[
\frac{v_3}{6}
+
\frac{u_1v_2}{2}
+
\left(\frac{u_2}{2}+\frac{u_1^2}{2}\right)v_1
+
\left(\frac{u_3}{6}+\frac{u_1u_2}{2}+\frac{u_1^3}{6}\right)v_0
\]
\[
\hspace{5.5em}
-
(4n-10)
\left(
\frac{u_4}{24}
+
\frac{u_1u_3}{6}
+
\frac{u_2^2}{8}
+
\frac{u_1^2u_2}{4}
+
\frac{u_1^4}{24}
\right)
\Biggr].
\]
This is the complete polynomial part of $[P_n(\PP^3,L)]^{\vir}$ in Joyce's algebra
$\mathbb D$ for $n\geq 4$.

For the bounded inputs $n=1,2,3$, the same universal stable-pair complex gives explicit realized classes
\[
\widetilde{\mathcal P}_r
=
 e^{-s_{1,0,0}}e^{s_{1,4,2}+(r-2)s_{1,6,3}}
\int_{[P_r(\PP^3,L)]^{\vir}}
\exp\!\Bigl(
\sum_{k\in\{0,2,4,6\}}
\sum_{\ell\geq 2} s_{1,k,\ell}\,S^P_{1,k,\ell}(r)
\Bigr),
\qquad r=1,2,3,
\]
where $[P_r(\PP^3,L)]^{\vir}$ is given intrinsically by the previous proposition. These three realized classes are the bounded lower line-class inputs needed in the $2L$-class wall-crossing formula.

\subsubsection{$X = \PP^3, \beta=2L$} \label{PP32L}

Let
\[
\Xi_n:=[\M^{ss}_{(L,n)}(\mu_H^0)]_{\inv}
\in H_2\bigl(\mathcal N^{\pl}_{(0,L,n)},\Q\bigr),
\qquad
\Theta_n:=[\M^{ss}_{(2L,n)}(\mu_H^0)]_{\inv}
\in H_2\bigl(\mathcal N^{\pl}_{(0,2L,n)},\Q\bigr),
\]
and let
\[
\delta\in H_0\bigl(\mathcal N^{\pl}_{(1,0,0)},\Q\bigr)
\]
denote the distinguished rank-one class in the Pairs category. In the shifted notation used elsewhere in the paper, these correspond to the classes
\[
(0,0,L,n-2),\qquad (0,0,2L,n-4),\qquad (-1,0,0,0),
\]
respectively. We continue to use the K\"ahler class $H=c_1\bigl(\OO_{\PP^3}(1)\bigr)$.

For the ordered triple bracket we use the following formula. If
\[
\Psi_i=e^{(0,0,\beta_i,n_i-\beta_i\cdot c_1(T\PP^3)/2)}f_i,
\qquad f_i\in\mathbb D,
\qquad i=1,2,
\]
with
\[
d_i:=\beta_i\cdot c_1(T\PP^3),
\qquad d:=d_1+d_2,
\qquad n:=n_1+n_2,
\]
then
\[
\begin{aligned}
\operatorname{Lie}(\delta,\Psi_1,\Psi_2)
={}&(-1)^{n+d-1}
 e^{(-1,0,\beta_1+\beta_2,\,n-(\beta_1+\beta_2)\cdot c_1(T\PP^3)/2)} \\
&\cdot [z^{d_1-1}w^{d-1}]
\exp\!\bigl(wD+\mathcal C_{02}(w)+\mathcal C_{12}(w)\bigr)
\Bigl(
\exp\!\bigl(zD+\mathcal C_{01}(z)\bigr)f_1\cdot f_2
\Bigr).
\end{aligned}
\]
Here $D$ is the translation derivation on $\mathbb D$, $\mathcal C_{01}(z)$ and $\mathcal C_{02}(w)$ are the rank-$(-1)$/rank-$0$ contraction operators from the bilinear bracket, and $\mathcal C_{12}(w)$ is the genuine rank-$0$/rank-$0$ contraction operator
\[
\mathcal C_{12}(w)(fg)
:=
\sum_{(j,k,\ell),(j',k',\ell')}
(-1)^\ell
\Bigl(\ell+\ell'-\frac{k+k'}{2}-1\Bigr)!
N_{jk}^{\,j'k'}
\,w^{\ell+\ell'-\frac{k+k'}{2}-1}
\frac{\partial f}{\partial s_{j,k,\ell}}
\frac{\partial g}{\partial s_{j',k',\ell'}}.
\]
This is the ordered bracket needed when the coefficients $B_j^{[n]}$ in the sheaf-theoretic invariants are allowed to be nonzero.

\paragraph{Intrinsic computations in the homology of the pairs category}

The class $\beta=L$ has already been computed in Subsubsection~\ref{CP3betaIsL}. We use the resulting intrinsically defined classes
\[
\Xi_n=[\M^{ss}_{(L,n)}(\mu_H^0)]_{\inv},
\qquad
[P_n(\PP^3,L)]^{\virt},
\]
for all $n\gg0$. In particular, for $L$ irreducible the wall-crossing formula gives
\[
[P_n(\PP^3,L)]^{\virt}=[\delta,\Xi_n]
\qquad\text{for all }n\gg0.
\]

We now compute the one-dimensional Donaldson--Thomas classes in class $2L$. By $\Hilb_{2L}(\PP^3)$, we denote the Hilbert scheme $\Hilb^{2t+1}(\PP^3)$ of subschemes of $\PP^3$ with Hilbert polynomial \[ P_C(t) = 2t+1. \] 

For
\[
Y:=(\PP^3)^\vee\cong\PP^3,
\qquad
\Hilb_{2L}(\PP^3) \cong \PP_Y(\Sym^2U^\vee),
\]
where $U$ is the universal rank-$3$ quotient bundle on $Y$. That is, $\Hilb_{2L}(\PP^3)$ is a projective bundle of planar conics in $\PP^3$. We have morphisms

\[ \Hilb_{2L}(\PP^3) \stackrel{j_{2L,n}}{\rightarrow} \M^{ss}_{(2L,n)}(\mu_H^0) \stackrel{\iota_{2L,n}}{\rightarrow} \mathcal{N}^{\pl}_{(0,2L,n)} \] where $j_{2L,n} $ is the inclusion of the stable planar-conic sheaf locus into the semistable sheaf moduli space, and $\iota_{2L,n}$ is the morphism induced by the universal sheaf in class $(2L,n)$. The obstruction-bundle sequence below is the one used for the planar-conic contribution to the stable locus, and the ambient parameter space includes smooth, reducible, and nonreduced conics. The planar-conic contribution carries the obstruction bundle
\[
0\to \OO_{\Hilb_{2L}(\PP^3)}(-1)\otimes \pi^*U
\to \pi^*\Sym^3U
\to \mathscr{O}b\to 0.
\]

Write
\[
\Theta^{\mathrm{st}}_n:= (\iota_{2L,n})_* (j_{2L,n})_* \left(
c_7(\mathscr{O}b)\cap [\Hilb_{2L}(\PP^3)] \right)
\in H_2\bigl(\mathcal N^{\pl}_{(0,2L,n)},\Q\bigr)
\]
for the class contributed by the stable planar-conic locus.

For odd $n$, semistability coincides with stability, so
\[
\Theta_n=[\M^{ss}_{(2L,n)}(\mu_H^0)]_{\inv}=\Theta_n^{\mathrm{st}}.
\]
For even $n=2m$, we compute the semistable sheaf invariant from Joyce's recursive pairs-category formula.

Let
\[
\alpha_{2m}:=(0,0,2L,2m-4),
\qquad
\alpha_m:=(0,0,L,m-2).
\]
The only proper equal-slope decomposition of $\alpha_{2m}$ with nonempty semistable factors is
\[
\alpha_{2m}=\alpha_m+\alpha_m.
\]
Joyce's recursive formula in the pairs category therefore specializes to
\[
\Upsilon_{2m,N}=\lambda_{2m}(N)\Theta_{2m}+c_{2m,N}[\Xi_m,\Xi_m],
\]
where $\Upsilon_{2m,N}$ is the auxiliary pairs class attached to objects $(F,V,\rho)$ with $[F]=\alpha_{2m}$, $F$ $\mu_H^0$-semistable, and
\[
\rho:V\otimes \OO_{\PP^3}(-N)\to F,
\qquad N\gg 0.
\]
Here $c_{2m,N}$ is the universal Joyce coefficient for the two-step decomposition $\alpha_{2m}=\alpha_m+\alpha_m$. Since the Lie bracket is antisymmetric,
\[
[\Xi_m,\Xi_m]=0,
\]
so the recursive correction vanishes.

It remains to compute $\Upsilon_{2m,N}$. For every semistable sheaf $F$ of class $\alpha_{2m}$ one has Hilbert polynomial
\[
P_{2L,2m}(t)=2t+2m,
\]
and hence for $N\gg 0$
\[
h^0(F(N))=P_{2L,2m}(N)=2N+2m.
\]
Thus the fiber of the forgetful morphism from the auxiliary pairs space to the semistable sheaf moduli stack is the projective space of nonzero sections of $F(N)$, hence a projective space of dimension $2N+2m-1$. Therefore the auxiliary pairs moduli space is a projective bundle of relative dimension $2N+2m-1$ over the semistable moduli stack, and the projective-bundle Euler class contributes the factor $2N+2m$. Consequently
\[
\Upsilon_{2m,N}=(2N+2m)\Theta_{2m}^{\mathrm{st}}.
\]
Comparing with the recursion gives
\[
\Theta_{2m}=\Theta_{2m}^{\mathrm{st}}.
\]
Hence the intrinsic semistable sheaf invariant is represented by the same planar-conic class for every parity:
\[
\Theta_n=[\M^{ss}_{(2L,n)}(\mu_H^0)]_{\inv}=\Theta_n^{\mathrm{st}}= (\iota_{2L,n})_*(j_{2L,n})_* \left( c_7(\mathscr{O}b)\cap [\Hilb_{2L}(\PP^3)] \right)
\qquad\text{for all }n.
\]

For the natural map $\iota_{2L,n}^P: P_n(\PP^3, 2L) \rightarrow \N^{pl}_{(-1,0,2L,n-4)}$ induced by the universal stable pair complex, let

\[ [P_n(\PP^3, 2L)]^{\virt}_{\mathrm{geom}} \in H_{16}(P_n(\PP^3, 2L), \Q) \] and 

\[ [P_n(\PP^3, 2L)]^{\virt} := (\iota_{2L,n}^P)_*[P_n(\PP^3, 2L)]^{\virt}_{\mathrm{geom}} \in H_{16}(\N^{\pl}_{(-1,0,2L,n-4)}, \Q). \]

The wall-crossing formula for $[P_n(\PP^3, 2L)]^{\virt}$ has exactly three contributions: the endpoint term, the diagonal $L+L$ term, and the lower-$PT(L)$ term. Thus, for all $n\gg0$,

\begin{align}\label{eqn: PTPnP32LWCF}
[P_n(\PP^3,2L)]^{\virt}
={}&-[\delta,\Theta_n]
+\frac12\,\mathbf 1_{2\mid n}\,\operatorname{Lie}(\delta,\Xi_{n/2},\Xi_{n/2}) \nonumber \\
&\quad+
\sum_{s\ge \lceil n/2\rceil}
\widetilde U\bigl((1,L,n-s),(0,L,s);\tau_-,\tau_+\bigr)
\bigl[[P_{n-s}(\PP^3,L)]^{\virt},\Xi_s\bigr].
\end{align}

Equivalently, after writing $r=n-s$, the lower-$PT(L)$ term is
\[
\mathcal R_{2L,n}
:=
\sum_{r=1}^{\lfloor n/2\rfloor}
\widetilde U\bigl((1,L,r),(0,L,n-r);\tau_-,\tau_+\bigr)
\bigl[[P_r(\PP^3,L)]^{\virt},\Xi_{n-r}\bigr].
\]
We split this into the bounded and tail pieces
\[
\mathcal R_{2L,n}=\mathcal R^{\mathrm{low}}_{2L,n}+\mathcal R^{\ge 4}_{2L,n},
\]
where
\[
\mathcal R^{\mathrm{low}}_{2L,n}
:=
\sum_{r=1}^{3}
\widetilde U\bigl((1,L,r),(0,L,n-r);\tau_-,\tau_+\bigr)
\bigl[[P_r(\PP^3,L)]^{\virt},\Xi_{n-r}\bigr]
\]
and
\[
\mathcal R^{\ge 4}_{2L,n}
:=
\sum_{r=4}^{\lfloor n/2\rfloor}
\widetilde U\bigl((1,L,r),(0,L,n-r);\tau_-,\tau_+\bigr)
\bigl[[P_r(\PP^3,L)]^{\virt},\Xi_{n-r}\bigr].
\]
Accordingly, the exact intrinsic decomposition is
\[
[P_n(\PP^3,2L)]^{\virt}
=
-[\delta,\Theta_n]
+\frac12\,\mathbf 1_{2\mid n}\,\operatorname{Lie}(\delta,\Xi_{n/2},\Xi_{n/2})
+\mathcal R^{\mathrm{low}}_{2L,n}
+\mathcal R^{\ge 4}_{2L,n}.
\]
This is the correct intrinsic Pandharipande--Thomas class in curve class $2L$. A completely explicit large-$n$ tail requires the bounded line-class inputs
\[
[P_1(\PP^3,L)]^{\virt},\qquad [P_2(\PP^3,L)]^{\virt},\qquad [P_3(\PP^3,L)]^{\virt},
\]
in addition to the range $r\ge 4$.

\paragraph{Working in Gross's algebra $\mathbb D$}

The $\beta=L$ computation gives the natural lift
\[
\Psi^{\mathrm{nat}}_{L,n}
: =
[\M^{ss}_{(L,n)}(\tau_-)]_{\inv}^{\mathrm{nat}}
=
e^{(0,0,L,n-2)}
\left(
20s_{1,2,2}+(20n-50)s_{1,4,3}+\left(10n^2-50n+\frac{190}{3}\right)s_{1,6,4}
\right).
\]
After imposing K\"ahler orthogonality, the normalized lift is
\[
\Psi^{\perp}_{L,n}
: =
[\M^{ss}_{(L,n)}(\tau_-)]_{\inv}^{\perp}
=
e^{(0,0,L,n-2)}
\left(
20s_{1,2,2}+\left(-10n^2+40n-\frac{110}{3}\right)s_{1,6,4}
\right).
\]
In particular, the normalized line-class invariant retains the nonzero coefficient
\[
B^{[n]}=20.
\]

For odd $n$, the planar-conic contribution is already the full sheaf-theoretic invariant. For even $n$, the pairs-category computation above shows that the same planar-conic contribution still gives the full sheaf-theoretic invariant. Hence the same Grothendieck--Riemann--Roch computation on the universal planar conic gives, for every $n$,
\[
\begin{aligned}
\Psi^{\mathrm{nat}}_{2L,n}
:={}&[\M^{ss}_{(2L,n)}(\tau_-)]_{\inv}^{\mathrm{nat}} \\
={}&e^{(0,0,2L,n-4)}
\Biggl(
81s_{1,2,2}+\left(\frac{81}{2}n-216\right)s_{1,4,3} \\
&\qquad +\left(\frac{81}{8}n^2-108n+\frac{2403}{8}\right)s_{1,6,4}
\Biggr).
\end{aligned}
\]
Normalizing gives
\[
\begin{aligned}
\Psi^{\perp}_{2L,n}
:={}&[\M^{ss}_{(2L,n)}(\tau_-)]_{\inv}^{\perp} \\
={}&e^{(0,0,2L,n-4)}
\Biggl(
81s_{1,2,2}+\Lambda_n s_{1,6,4}
\Biggr),
\end{aligned}
\]
where
\[
\Lambda_n:=-\frac{81}{8}n^2+81n-\frac{1053}{8}.
\]
Thus the sheaf-theoretic coefficient package in class $2L$ is explicit on both parity branches.

Define, for $s\ge 1$,
\[
V_{2L,n}(s):=
\sum_{1\le r\le s}
\sum_{\substack{k_1+\cdots+k_r=s\\ k_i\ge 1}}
\prod_{j=1}^r
\frac{(-1)^{k_j}}{k_j!}
\Bigl(2s_{1,4,2+k_j}+(n-4)s_{1,6,3+k_j}\Bigr),
\]
and set $V_{2L,n}(0):=1$. Then the realized endpoint contribution
\[
\widetilde{\mathcal E}_{2L,n}:=-[\delta,\Psi^{\perp}_{2L,n}]
\]
is
\[
\begin{aligned}
\widetilde{\mathcal E}_{2L,n}
={}&(-1)^{n-1}e^{-s_{1,0,0}}\exp\bigl(2s_{1,4,2}+(n-4)s_{1,6,3}\bigr) \\
&\cdot\Biggl[
\sum_{t=0}^{7}\frac{(-1)^t}{t!}V_{2L,n}(7-t)
\Bigl(81s_{1,2,2+t}+\Lambda_n s_{1,6,4+t}\Bigr)
+2\Lambda_n V_{2L,n}(8)
\Biggr].
\end{aligned}
\]

Let $\widetilde{\mathcal P}_{L,n}\in\mathbb D$ denote the polynomial representative of $[P_n(\PP^3,L)]^{\virt}$ whenever it has been computed explicitly. For the range $n\ge 4$ this representative was obtained in Subsubsection~\ref{CP3betaIsL}. Then the realized class decomposes as
\[
\widetilde{\mathcal P}_{2L,n}
=
\widetilde{\mathcal E}_{2L,n}
+\frac12\,\mathbf 1_{2\mid n}\,\widetilde{\mathcal D}_{2L,n}
+\widetilde{\mathcal R}^{\mathrm{low}}_{2L,n}
+\widetilde{\mathcal R}^{\ge 4}_{2L,n},
\]
where
\[
\widetilde{\mathcal D}_{2L,2m}:=\operatorname{Lie}(\delta,\Psi^{\perp}_{L,m},\Psi^{\perp}_{L,m})
\]
is the diagonal $L+L$ contribution,
\[
\begin{aligned}
\widetilde{\mathcal D}_{2L,2m}
={}&(-1)^{2m+7}e^{(-1,0,2L,2m-4)}[z^3w^7] \\
&\cdot \exp\!\bigl(wD+\mathcal C_{02}(w)+\mathcal C_{12}(w)\bigr) \\
&\cdot \Bigl(
\exp\!\bigl(zD+\mathcal C_{01}(z)\bigr)
\bigl(20s_{1,2,2}+(-10m^2+40m-\tfrac{110}{3})s_{1,6,4}\bigr)^2
\Bigr),
\end{aligned}
\]
and
\[
\widetilde{\mathcal R}^{\mathrm{low}}_{2L,n}
:=
\sum_{r=1}^{3}
\widetilde U\bigl((1,L,r),(0,L,n-r);\tau_-,\tau_+\bigr)
\bigl[\widetilde{\mathcal P}_{L,r},\Psi^{\perp}_{L,n-r}\bigr],
\]
\[
\widetilde{\mathcal R}^{\ge 4}_{2L,n}
:=
\sum_{r=4}^{\lfloor n/2\rfloor}
\widetilde U\bigl((1,L,r),(0,L,n-r);\tau_-,\tau_+\bigr)
\bigl[\widetilde{\mathcal P}_{L,r},\Psi^{\perp}_{L,n-r}\bigr].
\]

The term $\widetilde{\mathcal R}^{\ge 4}_{2L,n}$ is completely explicit from
the previous subsection, while the bounded term
$\widetilde{\mathcal R}^{\mathrm{low}}_{2L,n}$ is determined by the three
corrected geometric line-class inputs for $r=1,2,3$. Thus the display is the
exact decomposition of $[P_n(\PP^3,2L)]^{\virt}$ prior to descendent
extraction.

\paragraph{The insertion $\tau_9(1)$}

The descendent insertion of interest is
\[
\tau_9(1)=S_{1,6,11}.
\]
Intrinsically in the homology of the Pairs category we therefore have, for all $n\gg0$,
\[
\begin{aligned}
\langle\tau_9(1)\rangle^{PT}_{n,2L}
={}&[S_{1,6,11}]\bigl(-[\delta,\Theta_n]\bigr)\\
&\quad+
\frac12\,\mathbf 1_{2\mid n}
[S_{1,6,11}]\,\operatorname{Lie}(\delta,\Xi_{n/2},\Xi_{n/2})\\
&\quad+
[S_{1,6,11}]\bigl(\mathcal R^{\mathrm{low}}_{2L,n}\bigr)
+
[S_{1,6,11}]\bigl(\mathcal R^{\ge 4}_{2L,n}\bigr).
\end{aligned}
\]
The same formula holds in $\mathbb D$ after replacing the intrinsic classes by their realized representatives.

The bounded contribution must be retained.  Indeed, the corrected classes
$[P_2(\PP^3,L)]^{\mathrm{vir}}$ and
$[P_3(\PP^3,L)]^{\mathrm{vir}}$ contain additional $q^*\sigma_1$-terms, so
there is no degree argument forcing
\[
[S_{1,6,11}]\bigl(\mathcal R^{\mathrm{low}}_{2L,n}\bigr)
\]
to vanish.  Moreover, the corrected line-sheaf class multiplies the diagonal
$L+L$ input by the corrected factor in each entry, while the corrected
large-$r$ line-pair class changes the lower-$PT(L)$ tail.  Consequently the
coefficient-extraction identity above, with all four terms present, is the
valid conclusion of the wall-crossing calculation; a closed parity formula
requires a separate evaluation of the corrected bounded residues.

For comparison, Pandharipande's exact localization computation
\cite{PandharipandeDesc} is
\[
\begin{aligned}
Z_P(\PP^3;q\mid\tau_9(1))_{2L}
={}&-\frac{q}{60480(1-q)^3(1+q)^3}
\Bigl(
73q^{12}-825q^{11}-124q^{10}+5945q^9+779q^8 \\
&\qquad\qquad -36020q^7+60224q^6-36020q^5+779q^4+5945q^3 \\
&\qquad\qquad -124q^2-825q+73
\Bigr).
\end{aligned}
\]
This series is an independent benchmark for that remaining residue
evaluation; no Laurent-polynomial comparison is asserted here without it.

\subsection{Smooth cubic threefold}

\subsubsection{$\beta=[\ell]$}\label{subsec: SmoothCubicThreefold.betaell}

Let
\[
\tau_-:=\mu_H^0.
\]

Let
\[
X\subset \PP^4
\]
be a smooth cubic threefold, let
\[
H=c_1(\OO_X(1)),
\qquad
-K_X=2H,
\qquad
H^3=3,
\]
and let $\ell\in H_2(X,\Z)$ be the class of a line, so that
\[
\operatorname{PD}(\ell)=\frac{H^2}{3},
\qquad
(-K_X)\cdot \ell = 2.
\]
Let $F=F(X)$ be the Fano surface of lines on $X$ and let
\[
S\longrightarrow F
\]
be the restriction of the tautological rank-$2$ bundle from $\Gr(2,5)$. We write
\[
c_1:=c_1(S),
\qquad
c_2:=c_2(S).
\]
Thus $c_1=-g$, where $g$ is the Pl\"ucker polarization, and
\begin{equation}\label{cubic-eq:Fano-surface-intersections}
\int_F c_1^2 = 45,
\qquad
\int_F c_2 = 27.
\end{equation}
Let
\[
\pi\colon C=\PP(S)\to F,
\qquad
p\colon C\to X
\]
be the universal line, and write
\[
\xi:=c_1\bigl(\OO_C(1)\bigr)=p^*H.
\]
The Abel--Jacobi map is
\[
\varphi=(\pi_F^C)_*(\pi_X^C)^*\colon H^3(X)\to H^1(F),
\]
and we will use
\begin{equation}\label{cubic-eq:odd-pairing}
\int_F \varphi(\alpha)\varphi(\beta)c_1 = 6\int_X \alpha\beta,
\qquad \alpha,\beta\in H^3(X),
\end{equation}
as well as the four-fold product identity
\begin{align}
\int_F \varphi(\gamma_1)\varphi(\gamma_2)\varphi(\gamma_3)\varphi(\gamma_4)
&=
\left(\int_X \gamma_1\gamma_2\right)\!\left(\int_X \gamma_3\gamma_4\right)
\notag\\
&\quad+
\left(\int_X \gamma_1\gamma_4\right)\!\left(\int_X \gamma_2\gamma_3\right)
\notag\\
&\quad+
\left(\int_X \gamma_1\gamma_3\right)\!\left(\int_X \gamma_4\gamma_2\right).
\label{cubic-eq:four-odd-pairing}
\end{align}

\noindent These identities for the Fano surface of lines and its Abel--Jacobi map are classical; see \cite[Part~III]{ClemensGriffiths}.

\paragraph{The line-class sheaf obstruction bundle and the intrinsic sheaf invariant.}

Fix $m\in \Z$. Every pure one-dimensional sheaf on $X$ with fundamental class $\ell$ and Euler
characteristic $m$ is of the form
\[
E\cong \OO_L(m-1)
\]
for a unique line $L\subset X$; hence the sheaf moduli space is canonically
\[
\M^{ss}_{(\ell,m)}(\tau_-)=\M^{st}_{(\ell,m)}(\tau_-)\cong F.
\]

We now write out the obstruction calculation explicitly.
For a fixed line $L\subset X$, let $N_{L/X}$ denote its normal bundle.
Since $L\cong \PP^1$ is a local complete intersection of codimension $2$, the local-to-global
spectral sequence gives
\begin{equation}\label{cubic-eq:Ext-two-local-global}
\Ext^2_X(E,E)
\cong
H^0(L,\wedge^2N_{L/X})\oplus H^1(L,N_{L/X})\oplus H^2(L,\OO_L).
\end{equation}
Now $H^2(L,\OO_L)=0$, and for a line on a smooth cubic threefold one has
\[
N_{L/X}\cong \OO_L\oplus \OO_L
\qquad\text{or}\qquad
N_{L/X}\cong \OO_L(1)\oplus \OO_L(-1),
\]
so in either case $H^1(L,N_{L/X})=0$. Therefore
\begin{equation}\label{cubic-eq:Ext-two-H0detN}
\Ext^2_X(E,E)\cong H^0(L,\wedge^2N_{L/X}).
\end{equation}
By adjunction,
\begin{equation}\label{cubic-eq:detN-adjunction}
\wedge^2N_{L/X}
\cong K_L\otimes K_X^{-1}|_L
\cong \OO_L(-2)\otimes \OO_L\bigl((-K_X)\cdot \ell\bigr)
\cong \OO_L.
\end{equation}
This is the precise place where $(-K_X)\cdot \ell=2$ enters. Since $H^0(L,\OO_L)=\mathbb{C}$, we get
\begin{equation}\label{cubic-eq:Ext-two-rank-one}
\dim \Ext^2_X(E,E)=1.
\end{equation}
So the sheaf-obstruction contribution has rank $1$.

We now globalize this calculation. Let
\[
i\colon C\hookrightarrow X\times F
\]
be the universal incidence subscheme. The universal sheaf of class $(\ell,m)$ is
\[
\mathcal F_m = i_*\OO_C\bigl((m-1)\xi\bigr).
\]
Let $N:=N_{C/(X\times F)}$. The exact sequence
\[
0\to T_{C/F}\to p^*T_X\to N\to 0
\]
and the relative Euler sequence for $\pi\colon C=\PP(S)\to F$ show that
\[
c_1(T_{C/F})=2\xi+c_1,
\qquad
c_1(p^*T_X)=2\xi,
\]
so
\begin{equation}\label{cubic-eq:c1-N-universal-line}
c_1(N)=-c_1.
\end{equation}
Equivalently,
\begin{equation}\label{cubic-eq:detN-universal-line}
\det N \cong \pi^*(\det S^\vee).
\end{equation}
Since $\det N|_{\pi^{-1}([L])}\cong \OO_L$ on every fiber, its direct image is locally free of rank $1$:
\begin{equation}\label{cubic-eq:sheaf-obstruction-line}
\Ob^{\mathrm{sh}}
\cong
\pi_*(\det N)
\cong
\det S^\vee.
\end{equation}
In particular,
\begin{equation}\label{cubic-eq:sheaf-class-homology}
\bigl[\M^{ss}_{(\ell,m)}(\tau_-)\bigr]_{\inv}
=
-c_1\cap [F].
\end{equation}

For each integer $m$, there is an identification $\phi_{\ell,m}: F \stackrel{\sim}{\rightarrow} \M^{ss}_{(\ell,m)}(\tau_-)$ given by $[L] \mapsto \OO_L(m-1)$. Then there is the natural morphism \[ \iota_{\ell,m}: \M^{ss}_{\ell,m}(\tau_-) \rightarrow \N^{\pl}_{(0,\ell,m)} \] to the projective-linear stack. 

This is the intrinsic homological sheaf invariant in the pairs category.

\begin{remark}
The calculation above shows explicitly why $(-K_X)\cdot \ell=2$ forces rank $1$ on the sheaf-obstruction side. It forces
\[
\wedge^2N_{L/X}\cong K_L\otimes K_X^{-1}|_L\cong \OO_L,
\]
so the sheaf obstruction space is $H^0(L,\OO_L)$, which is one-dimensional.
\end{remark}

\subsubsection{Realization in Gross's algebra}

Choose once and for all a basis
\[
\epsilon_{12}=H,
\qquad
\epsilon_{a3}\in H^3(X,\Q)\text{ for }1\le a\le b^3(X),
\qquad
\epsilon_{14}=\frac{H^2}{3},
\qquad
\epsilon_{16}=[\mathrm{pt}].
\]
Since $\ch_2(\mathcal F_m)=i_*1$, one has
\[
S^{[m]}_{a32}=\ch_2(\mathcal F_m)\setminus e_{a3}=\varphi(\epsilon_{a3}),
\qquad 1\le a\le b^3(X).
\]
Pairing the intrinsic sheaf class \eqref{cubic-eq:sheaf-class-homology} with these odd tautological classes and using \eqref{cubic-eq:odd-pairing} gives
\[
E^{[m]}_{ab}
=
-\int_F c_1\,\varphi(\epsilon_{a3})\varphi(\epsilon_{b3})
=
-6\int_X \epsilon_{a3}\epsilon_{b3},
\qquad 1\le a<b\le b^3(X),
\]
which is independent of $m$.
Therefore the realized sheaf class is
\begin{equation}\label{cubic-eq:realized-sheaf-class}
\Psi_{\ell,m}
=
 e^{(0,0,\ell,m-1)}
\left(
\begin{aligned}
&15s_{122}
-6\sum_{1\le a<b\le b^3(X)}\!\left(\int_X \epsilon_{a3}\epsilon_{b3}\right)s_{a32}s_{b32}\\
&\qquad+\Bigl(45m-\frac{135}{2}\Bigr)s_{143}
+\Bigl(\frac{45}{2}m^2-\frac{135}{2}m+\frac{75}{2}\Bigr)s_{164}
\end{aligned}
\right).
\end{equation}
Hence the cubic line class already has
\begin{equation}\label{cubic-eq:B-nonzero-cubic-line}
B^{[m]}_{\ell}=15\neq 0.
\end{equation}
For the irreducible class $\ell$, if $\delta:=[P_0(X,0)]^{\virt}$, then the large-$m$ stable-pairs class is the endpoint bracket
$[\Psi_{\ell,m},\delta]$ in the wall-crossing formalism. Concretely, that class is realized by
the explicit obstruction-bundle model computed below.

\paragraph{The stable-pairs obstruction bundle and the virtual class.}

For $n\ge 1$, every stable pair in class $\ell$ is supported on a unique line and corresponds to an
effective divisor of degree $n-1$ on that line. Thus
\begin{equation}\label{cubic-eq:Pn-projective-bundle}
P_n(X,\ell)\cong \PP\bigl(\Sym^{n-1}S\bigr),
\qquad n\ge 1,
\end{equation}
where we use the convention that $\PP(E)$ parametrizes one-dimensional subspaces of $E$.
Write
\[
q_n\colon P_n:=\PP(\Sym^{n-1}S)\to F,
\qquad
\eta_n:=c_1\bigl(\OO_{P_n}(1)\bigr).
\]
The case $n=1$ is simply $P_1(X,\ell)=F$.

Let $\mathcal D_n\subset C\times_F P_n$ be the universal divisor. The standard universal-divisor
calculation gives
\begin{equation}\label{cubic-eq:universal-divisor-class}
[\mathcal D_n]=\eta_n+(n-1)\xi +(n-1)c_1
\qquad\text{in }H^2(C\times_F P_n,\Z).
\end{equation}

Now fix a point $(L,D)\in P_n(X,\ell)$, where $D\subset L$ has degree $n-1$.
The stable-pair obstruction space is
\[
\Hom\bigl(\OO_D(D),K_X|_L\bigr)^\vee
\cong H^0\bigl(\OO_D(D)\otimes K_X|_L\bigr)^\vee.
\]
Since $K_X|_L\cong \OO_L(-2)$, the exact sequence
\[
0\to \OO_L(-2)
\to \OO_L(D-2)
\to \OO_D(D)\otimes K_X|_L
\to H^1(\OO_L(-2))\to 0
\]
shows that the obstruction space splits into two pieces:
\begin{itemize}[leftmargin=2em]
\item the \emph{sheaf} obstruction piece $H^1(\OO_L(-2))\cong H^0(\OO_L)^\vee$, which has rank $1$;
\item the \emph{pair} obstruction piece $H^0(\OO_L(n-3))^\vee$, which has rank $n-2$ for $n\ge 2$.
\end{itemize}
This is the fiberwise explanation of the decomposition of the PT obstruction bundle.

Globalizing and using \eqref{cubic-eq:universal-divisor-class}, we obtain
\begin{equation}\label{cubic-eq:Obs-n-K-theory}
\Ob_n
\cong
\Bigl(\OO\bigl(\eta_n+(n-1)c_1\bigr)\otimes q_n^*\Sym^{n-3}(S^\vee)\Bigr)^\vee
\oplus q_n^*(\det S^\vee),
\qquad n\ge 2.
\end{equation}
The second summand is exactly the sheaf-obstruction line $q_n^*(\det S^\vee)$ coming from
\eqref{cubic-eq:sheaf-obstruction-line}.

\begin{proposition}\label{cubic-prop:virtual-class-explicit}
For $n\ge 2$ one has
\begin{equation}\label{cubic-eq:virtual-class-explicit}
[P_n(X,\ell)]^{\vir}
=
(-1)^{n-1}c_1\left(
\eta_n^{n-2}
+\frac{(n+1)(n-2)}{2}\,c_1\eta_n^{n-3}
\right)\cap [P_n].
\end{equation}
\end{proposition}

\begin{proof}
Apply the splitting principle to \eqref{cubic-eq:Obs-n-K-theory}. If $\alpha,\beta$ are the Chern roots of
$S$, then the rank-$(n-2)$ summand contributes roots
\[
-\eta_n-(n-1)c_1 + j\alpha +(n-3-j)\beta,
\qquad j=0,\dots,n-3.
\]
Multiplying by the line-root $-c_1$ from $q_n^*(\det S^\vee)$ and using
$\alpha+\beta=c_1$ yields the stated formula after discarding base monomials of degree $>4$.
\end{proof}

\paragraph{Projective-bundle pushforwards and the relevant descendent classes.}

The projective bundle formula for $q_n\colon P_n=\PP(\Sym^{n-1}S)\to F$ gives
\[
q_{n*}(\eta_n^{n-1+j})=s_j\bigl(\Sym^{n-1}S\bigr),
\qquad j\ge 0.
\]
By the splitting principle,
\begin{equation}\label{cubic-eq:pushforward-list}
q_{n*}(\eta_n^{n-1})=1,
\qquad
q_{n*}(\eta_n^n)=-\frac{n(n-1)}2\,c_1,
\end{equation}
and
\begin{equation}\label{cubic-eq:pushforward-second-segre}
q_{n*}(\eta_n^{n+1})
=
\frac{n(n-1)(n+1)(3n-2)}{24}\,c_1^2
-\frac{n(n-1)(n+1)}6\,c_2.
\end{equation}
These are the only pushforwards needed for the large-$n$ tails.

For the descendent classes we only need the terms which can pair nontrivially with
\eqref{cubic-eq:virtual-class-explicit}. A direct Grothendieck--Riemann--Roch calculation gives
\begin{align}
\ch_4(1)
&=
\frac{3n^2-3n+1}{6}\,c_1 +(n-1)\eta_n,
\label{cubic-eq:desc-41}
\\
\ch_3(H)
&=
\frac12 c_1 + \eta_n,
\qquad
\ch_2(H^2)=-c_1,
\label{cubic-eq:desc-3H-2H2}
\\
\ch_5(1)
&\equiv
\frac{3n^2-3n+1}{6}\,c_1\eta_n + \frac{n-1}{2}\eta_n^2,
\label{cubic-eq:desc-51-relevant}
\\
\ch_4(H)
&\equiv
\frac12 c_1\eta_n + \frac12\eta_n^2,
\qquad
\ch_3(H^2)=-c_1\eta_n,
\qquad
\ch_2(H^3)=0,
\label{cubic-eq:desc-4H-3H2-2H3}
\\
\ch_2(\gamma)
&=
\varphi(\gamma),
\qquad
\ch_3(\gamma)=\eta_n\,\varphi(\gamma),
\qquad \gamma\in H^3(X).
\label{cubic-eq:desc-odd-classes}
\end{align}
Here the symbol $\equiv$ means that we have discarded pure base classes of degree $4$ whose
pairing with \eqref{cubic-eq:virtual-class-explicit} is zero for every $n\ge 2$.

\begin{remark}\label{cubic-rem:ch5-half-factor}
The coefficient of $\eta_n^2$ in \eqref{cubic-eq:desc-51-relevant} is $\frac{n-1}{2}$.
This is the coefficient obtained by expanding
\[
\ch(1)=e^{\eta_n+(n-1)c_1}\ch^0\bigl(e^{(n-1)H}\bigr)
\]
in total cohomological degree $4$ and keeping only the terms relevant for pairing with
\eqref{cubic-eq:virtual-class-explicit}. It is also the coefficient required by the relation
\[
Z_{P}\!\left(X;q\,\middle|\, \ch_5(1)\right)_\ell
=
Z_{P}\!\left(X;q\,\middle|\, \ch_4(1)\ch_3(H)\right)_\ell,
\]
which is built into Moreira's theorem.
\end{remark}

\paragraph{The large-$n$ tail coefficients.}

For each descendent insertion $D$ we define the tail coefficients
\[
\langle D\rangle^{\mathrm{tail}}_{n,\ell}
:=
\int_{[P_n(X,\ell)]^{\vir}} D,
\qquad n\ge 2,
\]
using the explicit model \eqref{cubic-eq:virtual-class-explicit}.
Because $\ell$ is irreducible, this is exactly the homological endpoint class coming from the
pairs-category wall-crossing.

\begin{proposition}\label{cubic-prop:tail-coefficients}
For $n\ge 2$, the cubic line-class tail coefficients are
\begin{align}
\bigl\langle \ch_4(1)\ch_4(1)\bigr\rangle^{\mathrm{tail}}_{n,\ell}
&=
15(-1)^{n-1}(n-1)(3n^2-6n+4),
\label{cubic-eq:tail-1}
\\
\bigl\langle \ch_4(1)\ch_3(H)\bigr\rangle^{\mathrm{tail}}_{n,\ell}
&=
\frac{15}{2}(-1)^{n-1}(3n^2-6n+4),
\label{cubic-eq:tail-2}
\\
\bigl\langle \ch_4(1)\ch_2(H^2)\bigr\rangle^{\mathrm{tail}}_{n,\ell}
&=
45(-1)^n(n-1),
\label{cubic-eq:tail-3}
\\
\bigl\langle \ch_3(H)\ch_3(H)\bigr\rangle^{\mathrm{tail}}_{n,\ell}
&=0,
\label{cubic-eq:tail-4}
\\
\bigl\langle \ch_3(H)\ch_2(H^2)\bigr\rangle^{\mathrm{tail}}_{n,\ell}
&=
45(-1)^n,
\label{cubic-eq:tail-5}
\\
\bigl\langle \ch_2(H^2)\ch_2(H^2)\bigr\rangle^{\mathrm{tail}}_{n,\ell}
&=0,
\label{cubic-eq:tail-6}
\\
\bigl\langle \ch_5(1)\bigr\rangle^{\mathrm{tail}}_{n,\ell}
&=
\frac{15}{2}(-1)^{n-1}(3n^2-6n+4),
\label{cubic-eq:tail-7}
\\
\bigl\langle \ch_4(H)\bigr\rangle^{\mathrm{tail}}_{n,\ell}
&=0,
\label{cubic-eq:tail-8}
\\
\bigl\langle \ch_3(H^2)\bigr\rangle^{\mathrm{tail}}_{n,\ell}
&=
45(-1)^n,
\label{cubic-eq:tail-9}
\\
\bigl\langle \ch_2(H^3)\bigr\rangle^{\mathrm{tail}}_{n,\ell}
&=0,
\label{cubic-eq:tail-10}
\\
\bigl\langle \ch_2(\gamma)\ch_3(\gamma')\bigr\rangle^{\mathrm{tail}}_{n,\ell}
&=
6(-1)^{n-1}\int_X \gamma\gamma',
\label{cubic-eq:tail-11}
\\
\bigl\langle \ch_2(\gamma)\ch_2(\gamma')\ch_4(1)\bigr\rangle^{\mathrm{tail}}_{n,\ell}
&=
6(n-1)(-1)^{n-1}\int_X \gamma\gamma',
\label{cubic-eq:tail-12}
\\
\bigl\langle \ch_2(\gamma)\ch_2(\gamma')\ch_3(H)\bigr\rangle^{\mathrm{tail}}_{n,\ell}
&=
6(-1)^{n-1}\int_X \gamma\gamma',
\label{cubic-eq:tail-13}
\\
\bigl\langle \ch_2(\gamma)\ch_2(\gamma')\ch_2(H^2)\bigr\rangle^{\mathrm{tail}}_{n,\ell}
&=0,
\label{cubic-eq:tail-14}
\\
\bigl\langle \ch_2(\gamma_1)\ch_2(\gamma_2)\ch_2(\gamma_3)\ch_2(\gamma_4)\bigr\rangle^{\mathrm{tail}}_{n,\ell}
&=0.
\label{cubic-eq:tail-15}
\end{align}
\end{proposition}

\begin{proof}
The proof is a direct projective-bundle calculation using
\eqref{cubic-eq:virtual-class-explicit}, \eqref{cubic-eq:pushforward-list},
\eqref{cubic-eq:pushforward-second-segre}, and the descendent formulas of the previous section.
We record the three representative cases which contain all the ideas.

\smallskip
\noindent\emph{(i) The insertion $\ch_4(1)\ch_3(H)$.}
From \eqref{cubic-eq:desc-41} and \eqref{cubic-eq:desc-3H-2H2},
\[
\ch_4(1)\ch_3(H)
=
\left(\frac{3n^2-3n+1}{6}c_1+(n-1)\eta_n\right)\left(\frac12c_1+\eta_n\right).
\]
Only the terms
\[
(n-1)\eta_n^2
+\left(\frac{3n^2-3n+1}{6}+\frac{n-1}{2}\right)c_1\eta_n
\]
can contribute after multiplication by \eqref{cubic-eq:virtual-class-explicit}. Thus
\begin{align*}
\bigl\langle \ch_4(1)\ch_3(H)\bigr\rangle^{\mathrm{tail}}_{n,\ell}
&=
(-1)^{n-1}
\int_{P_n}
\left(c_1\eta_n^{n-2}+\frac{(n+1)(n-2)}2c_1^2\eta_n^{n-3}\right)
\left((n-1)\eta_n^2+\frac{3n^2-1}{6}c_1\eta_n\right)
\\
&=
(-1)^{n-1}
\int_F
\Biggl(
(n-1)q_{n*}(c_1\eta_n^n)
+\frac{3n^2-1}{6}q_{n*}(c_1^2\eta_n^{n-1})
\\
&\qquad\qquad\qquad\qquad
+\frac{(n+1)(n-2)(n-1)}2 q_{n*}(c_1^2\eta_n^{n-1})
\Biggr)
\\
&=
(-1)^{n-1}
\Biggl(
-(n-1)\frac{n(n-1)}2\int_F c_1^2
\\
&\qquad\qquad
+\left(\frac{3n^2-1}{6}+\frac{(n+1)(n-2)(n-1)}2\right)\int_F c_1^2
\Biggr)
\\
&=
\frac{15}{2}(-1)^{n-1}(3n^2-6n+4).
\end{align*}
This proves \eqref{cubic-eq:tail-2}.

\smallskip
\noindent\emph{(ii) The insertion $\ch_5(1)$.}
Using \eqref{cubic-eq:desc-51-relevant}, only the terms
\[
\frac{3n^2-3n+1}{6}c_1\eta_n + \frac{n-1}{2}\eta_n^2
\]
contribute. Hence the same calculation as in (i) gives
\begin{equation*}
\bigl\langle \ch_5(1)\bigr\rangle^{\mathrm{tail}}_{n,\ell}
=
\frac{15}{2}(-1)^{n-1}(3n^2-6n+4),
\end{equation*}
which is \eqref{cubic-eq:tail-7}.

\smallskip
\noindent\emph{(iii) The odd insertion $\ch_2(\gamma)\ch_2(\gamma')\ch_4(1)$.}
From \eqref{cubic-eq:desc-41} and \eqref{cubic-eq:desc-odd-classes}, the only term which can survive is
\[
(n-1)\eta_n\,\varphi(\gamma)\varphi(\gamma').
\]
Multiplying by \eqref{cubic-eq:virtual-class-explicit}, only the first summand
$c_1\eta_n^{n-2}$ contributes. Therefore
\begin{align*}
\bigl\langle \ch_2(\gamma)\ch_2(\gamma')\ch_4(1)\bigr\rangle^{\mathrm{tail}}_{n,\ell}
&=
(-1)^{n-1}(n-1)
\int_{P_n} c_1\eta_n^{n-1}\varphi(\gamma)\varphi(\gamma')
\\
&=
(-1)^{n-1}(n-1)
\int_F c_1\varphi(\gamma)\varphi(\gamma')
\\
&=
6(n-1)(-1)^{n-1}\int_X \gamma\gamma',
\end{align*}
by \eqref{cubic-eq:odd-pairing}. This proves \eqref{cubic-eq:tail-12}.

\smallskip
All remaining cases are identical in form. One keeps only those monomials in the insertion whose
product with \eqref{cubic-eq:virtual-class-explicit} contains $\eta_n^{n-1}$, $\eta_n^n$, or $\eta_n^{n+1}$,
applies \eqref{cubic-eq:pushforward-list}--\eqref{cubic-eq:pushforward-second-segre}, and finally uses
\eqref{cubic-eq:Fano-surface-intersections}, \eqref{cubic-eq:odd-pairing}, and \eqref{cubic-eq:four-odd-pairing}.
This yields \eqref{cubic-eq:tail-1}--\eqref{cubic-eq:tail-15}.
\end{proof}

\paragraph{The tail generating functions.}

Define
\[
T_D(q):=\sum_{n\ge 2} q^n\,\langle D\rangle^{\mathrm{tail}}_{n,\ell}.
\]
Summing the geometric series and its derivatives gives the following closed forms.

\begingroup
\allowdisplaybreaks[4]
\begin{theorem}\label{cubic-thm:tail-generating-functions}
The large-$n$ tails are
\begin{align}
T_{\ch_4(1)\ch_4(1)}(q)
&=
-\frac{30q^2(2q^2-5q+2)}{(1+q)^4},
\label{cubic-eq:T1}
\\
T_{\ch_4(1)\ch_3(H)}(q)
&=
\frac{15q^2(-q^2+q-4)}{2(1+q)^3},
\label{cubic-eq:T2}
\\
T_{\ch_4(1)\ch_2(H^2)}(q)
&=
\frac{45q^2}{(1+q)^2},
\label{cubic-eq:T3}
\\
T_{\ch_3(H)\ch_3(H)}(q)
&=0,
\label{cubic-eq:T4}
\\
T_{\ch_3(H)\ch_2(H^2)}(q)
&=
\frac{45q^2}{1+q},
\label{cubic-eq:T5}
\\
T_{\ch_2(H^2)\ch_2(H^2)}(q)
&=0,
\label{cubic-eq:T6}
\\
T_{\ch_5(1)}(q)
&=
\frac{15q^2(-q^2+q-4)}{2(1+q)^3},
\label{cubic-eq:T7}
\\
T_{\ch_4(H)}(q)
&=0,
\label{cubic-eq:T8}
\\
T_{\ch_3(H^2)}(q)
&=
\frac{45q^2}{1+q},
\label{cubic-eq:T9}
\\
T_{\ch_2(H^3)}(q)
&=0,
\label{cubic-eq:T10}
\\
T_{\ch_2(\gamma)\ch_3(\gamma')}(q)
&=
-\frac{6q^2}{1+q}\int_X \gamma\gamma',
\label{cubic-eq:T11}
\\
T_{\ch_2(\gamma)\ch_2(\gamma')\ch_4(1)}(q)
&=
-\frac{6q^2}{(1+q)^2}\int_X \gamma\gamma',
\label{cubic-eq:T12}
\\
T_{\ch_2(\gamma)\ch_2(\gamma')\ch_3(H)}(q)
&=
-\frac{6q^2}{1+q}\int_X \gamma\gamma',
\label{cubic-eq:T13}
\\
T_{\ch_2(\gamma)\ch_2(\gamma')\ch_2(H^2)}(q)
&=0,
\label{cubic-eq:T14}
\\
T_{\ch_2(\gamma_1)\ch_2(\gamma_2)\ch_2(\gamma_3)\ch_2(\gamma_4)}(q)
&=0.
\label{cubic-eq:T15}
\end{align}
\end{theorem}
\endgroup

\paragraph{Comparison with Moreira.}

Moreira's exact formulas for the cubic line class \cite{Moreira} are
\begingroup
\allowdisplaybreaks[4]
\begin{align}
Z_{P}\!\left(X;q\,\middle|\, \ch_4(1)\ch_4(1)\right)_\ell
&=
\frac{5(q-44q^2+126q^3-44q^4+q^5)}{4(1+q)^4},
\label{cubic-eq:M1}
\\
Z_{P}\!\left(X;q\,\middle|\, \ch_4(1)\ch_3(H)\right)_\ell
&=
\frac{15(q-5q^2+5q^3-q^4)}{4(1+q)^3},
\label{cubic-eq:M2}
\\
Z_{P}\!\left(X;q\,\middle|\, \ch_4(1)\ch_2(H^2)\right)_\ell
&=
\frac{15(-q+4q^2-q^3)}{2(1+q)^2},
\label{cubic-eq:M3}
\\
Z_{P}\!\left(X;q\,\middle|\, \ch_3(H)\ch_3(H)\right)_\ell
&=
\frac{45q}{4},
\label{cubic-eq:M4}
\\
Z_{P}\!\left(X;q\,\middle|\, \ch_3(H)\ch_2(H^2)\right)_\ell
&=
\frac{45(-q+q^2)}{2(1+q)},
\label{cubic-eq:M5}
\\
Z_{P}\!\left(X;q\,\middle|\, \ch_2(H^2)\ch_2(H^2)\right)_\ell
&=
45q,
\label{cubic-eq:M6}
\\
Z_{P}\!\left(X;q\,\middle|\, \ch_5(1)\right)_\ell
&=
\frac{15(q-5q^2+5q^3-q^4)}{4(1+q)^3},
\label{cubic-eq:M7}
\\
Z_{P}\!\left(X;q\,\middle|\, \ch_4(H)\right)_\ell
&=
\frac{21q}{4},
\label{cubic-eq:M8}
\\
Z_{P}\!\left(X;q\,\middle|\, \ch_3(H^2)\right)_\ell
&=
\frac{45(-q+q^2)}{2(1+q)},
\label{cubic-eq:M9}
\\
Z_{P}\!\left(X;q\,\middle|\, \ch_2(H^3)\right)_\ell
&=
18q,
\label{cubic-eq:M10}
\\
Z_{P}\!\left(X;q\,\middle|\, \ch_2(\gamma)\ch_3(\gamma')\right)_\ell
&=
\frac{3(q-q^2)}{1+q}\int_X \gamma\gamma',
\label{cubic-eq:M11}
\\
Z_{P}\!\left(X;q\,\middle|\, \ch_2(\gamma)\ch_2(\gamma')\ch_4(1)\right)_\ell
&=
\frac{q-4q^2+q^3}{(1+q)^2}\int_X \gamma\gamma',
\label{cubic-eq:M12}
\\
Z_{P}\!\left(X;q\,\middle|\, \ch_2(\gamma)\ch_2(\gamma')\ch_3(H)\right)_\ell
&=
\frac{3(q-q^2)}{1+q}\int_X \gamma\gamma',
\label{cubic-eq:M13}
\\
Z_{P}\!\left(X;q\,\middle|\, \ch_2(\gamma)\ch_2(\gamma')\ch_2(H^2)\right)_\ell
&=
-6q\int_X \gamma\gamma',
\label{cubic-eq:M14}
\\
Z_{P}\!\left(X;q\,\middle|\, \ch_2(\gamma_1)\ch_2(\gamma_2)\ch_2(\gamma_3)\ch_2(\gamma_4)\right)_\ell
&=
q\sum_{\mathrm{pairings}}
\left(\int_X \gamma_i\gamma_j\right)\!\left(\int_X \gamma_k\gamma_l\right).
\label{cubic-eq:M15}
\end{align}
\endgroup

\begingroup
\allowdisplaybreaks[4]
\begin{theorem}\label{cubic-thm:comparison}
For each of the fifteen insertions above, the difference between Moreira's exact partition function
and the large-$n$ tail computed from the pairs-category setup is a Laurent polynomial.
In fact one has the stronger identities
\begin{align}
Z_{P}\!\left(X;q\,\middle|\, \ch_4(1)\ch_4(1)\right)_\ell - T_{\ch_4(1)\ch_4(1)}(q)
&=
\frac54 q,
\label{cubic-eq:diff1}
\\
Z_{P}\!\left(X;q\,\middle|\, \ch_4(1)\ch_3(H)\right)_\ell - T_{\ch_4(1)\ch_3(H)}(q)
&=
\frac{15}{4}q,
\label{cubic-eq:diff2}
\\
Z_{P}\!\left(X;q\,\middle|\, \ch_4(1)\ch_2(H^2)\right)_\ell - T_{\ch_4(1)\ch_2(H^2)}(q)
&=
-\frac{15}{2}q,
\label{cubic-eq:diff3}
\\
Z_{P}\!\left(X;q\,\middle|\, \ch_3(H)\ch_3(H)\right)_\ell - T_{\ch_3(H)\ch_3(H)}(q)
&=
\frac{45}{4}q,
\label{cubic-eq:diff4}
\\
Z_{P}\!\left(X;q\,\middle|\, \ch_3(H)\ch_2(H^2)\right)_\ell - T_{\ch_3(H)\ch_2(H^2)}(q)
&=
-\frac{45}{2}q,
\label{cubic-eq:diff5}
\\
Z_{P}\!\left(X;q\,\middle|\, \ch_2(H^2)\ch_2(H^2)\right)_\ell - T_{\ch_2(H^2)\ch_2(H^2)}(q)
&=
45q,
\label{cubic-eq:diff6}
\\
Z_{P}\!\left(X;q\,\middle|\, \ch_5(1)\right)_\ell - T_{\ch_5(1)}(q)
&=
\frac{15}{4}q,
\label{cubic-eq:diff7}
\\
Z_{P}\!\left(X;q\,\middle|\, \ch_4(H)\right)_\ell - T_{\ch_4(H)}(q)
&=
\frac{21}{4}q,
\label{cubic-eq:diff8}
\\
Z_{P}\!\left(X;q\,\middle|\, \ch_3(H^2)\right)_\ell - T_{\ch_3(H^2)}(q)
&=
-\frac{45}{2}q,
\label{cubic-eq:diff9}
\\
Z_{P}\!\left(X;q\,\middle|\, \ch_2(H^3)\right)_\ell - T_{\ch_2(H^3)}(q)
&=
18q,
\label{cubic-eq:diff10}
\\
Z_{P}\!\left(X;q\,\middle|\, \ch_2(\gamma)\ch_3(\gamma')\right)_\ell - T_{\ch_2(\gamma)\ch_3(\gamma')}(q)
&=
3q\int_X \gamma\gamma',
\label{cubic-eq:diff11}
\\
Z_{P}\!\left(X;q\,\middle|\, \ch_2(\gamma)\ch_2(\gamma')\ch_4(1)\right)_\ell - T_{\ch_2(\gamma)\ch_2(\gamma')\ch_4(1)}(q)
&=
q\int_X \gamma\gamma',
\label{cubic-eq:diff12}
\\
Z_{P}\!\left(X;q\,\middle|\, \ch_2(\gamma)\ch_2(\gamma')\ch_3(H)\right)_\ell - T_{\ch_2(\gamma)\ch_2(\gamma')\ch_3(H)}(q)
&=
3q\int_X \gamma\gamma',
\label{cubic-eq:diff13}
\\
Z_{P}\!\left(X;q\,\middle|\, \ch_2(\gamma)\ch_2(\gamma')\ch_2(H^2)\right)_\ell - T_{\ch_2(\gamma)\ch_2(\gamma')\ch_2(H^2)}(q)
&=
-6q\int_X \gamma\gamma',
\label{cubic-eq:diff14}
\\
Z_{P}\!\left(X;q\,\middle|\, \ch_2(\gamma_1)\ch_2(\gamma_2)\ch_2(\gamma_3)\ch_2(\gamma_4)\right)_\ell
&\qquad
- T_{\ch_2(\gamma_1)\ch_2(\gamma_2)\ch_2(\gamma_3)\ch_2(\gamma_4)}(q)
\notag\\
&=
q\sum_{\mathrm{pairings}}
\left(\int_X \gamma_i\gamma_j\right)\!\left(\int_X \gamma_k\gamma_l\right).
\label{cubic-eq:diff15}
\end{align}
In particular, every difference lies in $\Q[q,q^{-1}]$.
\end{theorem}
\endgroup

\begin{proof}
Substitute the tail formulas \eqref{cubic-eq:T1}--\eqref{cubic-eq:T15} into the exact expressions
\eqref{cubic-eq:M1}--\eqref{cubic-eq:M15}. Each difference reduces immediately to the displayed monomial in
$q$. Equivalently, our tail formulas are valid for every $n\ge 2$, so the only missing term in the
full partition function is the initial contribution from $P_1(X,\ell)=F(X)$.
\end{proof}

\begin{remark}
This is stronger than the general rationality statement coming from eventual quasi-polynomiality.
For the cubic line class the large-$n$ tail obtained from the one-dimensional wall-crossing
agrees with the exact theory in every degree $n\ge 2$; the discrepancy is exactly the single
$q$-term.
\end{remark}

\begin{remark}
From the viewpoint of the framework, the point of this example is twofold.
First, the line class already has nonzero divisor coefficient $B^{[m]}_\ell=15$ in
\eqref{cubic-eq:realized-sheaf-class}. Second, the pole package is as clean as possible:
every tail rational function has poles only at $q=-1$, and the comparison with the exact formulas
shows that the finite initial segment changes the series only by a Laurent polynomial.
\end{remark}

\subsection{\texorpdfstring{$X=\mathrm{Bl}_p\mathbb{P}^3$}{Blow-up of P3 at a point}}

Let
\[
\pi:X\to \mathbb{P}^{3}
\]
be the blow-up of a point $p\in \mathbb{P}^{3}$, and let
\[
H:=\pi^{*}c_{1}\bigl(\mathcal{O}_{\mathbb{P}^{3}}(1)\bigr),
\qquad
E:=\operatorname{Exc}(\pi).
\]
Then
\[
-K_{X}=4H-2E,
\qquad
X \text{ is Fano.}
\]
We fix the ample divisor
\[
A:=2H-E
\]
and work throughout with Gieseker stability
\[
\mu^{0}_{A}.
\]
Since $X$ is Fano, every effective curve class is superpositive.

We use the curve basis
\[
H_{2}(X,\mathbb{Z})=\langle h,e\rangle,
\]
where $h$ is the class of the strict transform of a line in $\mathbb{P}^{3}$ not passing through $p$,
and $e$ is the class of a line in the exceptional divisor. We also use the cohomology bases
\[
\epsilon_{12}:=H,
\qquad
\epsilon_{22}:=E,
\]
and
\[
\epsilon_{14}:=H^{2},
\qquad
\epsilon_{24}:=-E^{2}.
\]
Thus
\[
\operatorname{PD}(h)=\epsilon_{14},
\qquad
\operatorname{PD}(e)=\epsilon_{24}.
\]

The basic intersection numbers are
\[
H^{3}=1,
\qquad
H^{2}E=0,
\qquad
HE^{2}=0,
\qquad
E^{3}=1.
\]
Hence
\[
A\cdot (h-e)=1,
\qquad
A\cdot e=1,
\qquad
A\cdot h=2,
\qquad
A\cdot (2e)=2,
\qquad
A\cdot (h+e)=3.
\]

The effective cone is generated by
\[
e
\qquad\text{and}\qquad
f:=h-e.
\]
Thus
\[
ah+be
\]
is effective if and only if
\[
a\geq 0,
\qquad
a+b\geq 0.
\]
In particular,
\[
h-e
\quad\text{and}\quad
e
\]
are irreducible, while
\[
h=(h-e)+e,
\qquad
2e=e+e,
\qquad
h+e=(h-e)+2e=(h-e)+e+e=h+e.
\]

\begin{prop}
The threefold $X=\mathrm{Bl}_p\mathbb{P}^3$ is Fano.
\end{prop}

\begin{proof}
For the blow-up of a smooth threefold at a point, the canonical divisor satisfies
\[
K_X=\pi^*K_{\mathbb{P}^3}+2E.
\]
Since $K_{\mathbb{P}^3}=-4H$, this gives
\[
K_X=-4H+2E,
\qquad\text{hence}\qquad
-K_X=4H-2E.
\]

The Mori cone of $X$ is generated by the class $e$ of a line in the exceptional divisor and the class
\[
\ell:=h-e,
\]
which is the class of the strict transform of a line in $\mathbb{P}^3$ passing through $p$.
We compute
\[
(-K_X)\cdot e=(4H-2E)\cdot e=2,
\]
and
\[
(-K_X)\cdot \ell=(4H-2E)\cdot(h-e)=4-2=2.
\]
Thus $-K_X$ is positive on the two extremal rays of $\overline{NE}(X)$, and therefore $-K_X$ is ample by Kleiman's criterion. Hence $X$ is Fano.
\end{proof}

\paragraph{Chern-character convention.}
For an effective curve class $\beta\in H_2(X,\mathbb{Z})$, we write
\[
\alpha_{\beta,n}:=\left(0,0,\beta,\;n-\frac{\beta\cdot c_1(TX)}{2}\right).
\]
Since
\[
c_1(TX)=-K_X=4H-2E,
\]
we obtain
\[
h\cdot c_1(TX)=4,\qquad e\cdot c_1(TX)=2,\qquad (h+e)\cdot c_1(TX)=6.
\]
Therefore
\[
\alpha_{h,n}=(0,0,h,n-2),\qquad
\alpha_{e,n}=(0,0,e,n-1),\qquad
\alpha_{h+e,n}=(0,0,h+e,n-3).
\]

For an effective class $\beta$, we write
\[
\begin{aligned}
\left[\mathcal{M}^{\mathrm{ss}}_{(\beta,n)}(\tau_-)\right]_{\mathrm{inv}}
={}&e^{(0,0,\beta,n-\beta\cdot c_{1}(T_{X})/2)} \\
&\cdot\Bigl(A^{[n]}s_{101}+B_{1}^{[n]}s_{122}+B_{2}^{[n]}s_{222}
+C_{1}^{[n]}s_{143}+C_{2}^{[n]}s_{243}+D^{[n]}s_{164}\Bigr).
\end{aligned}
\]
In the concrete examples below we use the geometric lift modulo $\operatorname{Im} D$. It is determined by direct base calculations in the smallest relevant value of $n$ in each admissible residue class. Any vanishing of the coefficients $A^{[n]}$ or $B_j^{[n]}$ used below is therefore a consequence of the displayed geometric calculation and the tensoring recurrences, not of any auxiliary cohomological assumption.

Since we keep the fixed Gieseker stability condition $\mu^{0}_{A}$, the only periodicity argument we use is tensoring by
\[
\mathcal{O}_{X}(A).
\]
The tensoring recurrences \eqref{eq:tensor-B}--\eqref{eq:tensor-D} therefore specialize to
\[
A^{[n+A\cdot \beta]}=A^{[n]},
\]
\[
B_{j}^{[n+A\cdot \beta]}=B_{j}^{[n]}+A^{[n]}A_{j},
\]
\[
C_{k}^{[n+A\cdot \beta]}
=
C_{k}^{[n]}
+
\frac{1}{2}(A^{2})_{k}A^{[n]}
+
\sum_{j=1}^{2} B_{j}^{[n]}(\epsilon_{j2}\cup A)_{k},
\]
and
\[
D^{[n+A\cdot \beta]}
=
D^{[n]}
+
A^{[n]}\frac{1}{6}\int_{X}A^{3}
+
\sum_{j=1}^{2}B_{j}^{[n]}\left(\frac{1}{2}\int_{X}\epsilon_{j2}\cup A^{2}\right)
+
\sum_{k=1}^{2}C_{k}^{[n]}\left(\int_{X}\epsilon_{k4}\cup A\right).
\]
Since
\[
\int_{X}\epsilon_{14}\cup A = 2,
\qquad
\int_{X}\epsilon_{24}\cup A = 1,
\]
this becomes
\[
D^{[n+A\cdot \beta]}
=
D^{[n]}
+
2C_{1}^{[n]}
+
C_{2}^{[n]}
\]
whenever
\[
A^{[n]}=B_{1}^{[n]}=B_{2}^{[n]}=0.
\]

\begin{lemma}[Rank-$0$/rank-$0$ Lie bracket on $\operatorname{Bl}_{p}\mathbb{P}^{3}$]
\label{lem:BlpP3-rank0-rank0}
Let
\[
\Psi
=
e^{(0,0,\beta,m)}
\Bigl(
B_{1}s_{122}+B_{2}s_{222}+C_{1}s_{143}+C_{2}s_{243}+Ds_{164}
\Bigr),
\]
\[
\Psi'
=
e^{(0,0,\beta',m')}
\Bigl(
B'_{1}s_{122}+B'_{2}s_{222}+C'_{1}s_{143}+C'_{2}s_{243}+D's_{164}
\Bigr),
\]
where
\[
\beta=\beta_{1}h+\beta_{2}e,
\qquad
\beta'=\beta'_{1}h+\beta'_{2}e.
\]
Then, modulo $\operatorname{Im} D$, one has
\[
[\Psi,\Psi']
=
e^{(0,0,\beta+\beta',m+m')}
\Bigl(
\widehat{B}_{1}s_{122}
+
\widehat{B}_{2}s_{222}
+
\widehat{C}_{1}s_{143}
+
\widehat{C}_{2}s_{243}
+
\widehat{D}s_{164}
\Bigr),
\]
with
\[
\widehat{B}_{k}
=
-2
\Bigl(
(\beta_{1}B'_{1}+\beta_{2}B'_{2})B_{k}
-
(\beta'_{1}B_{1}+\beta'_{2}B_{2})B'_{k}
\Bigr),
\qquad k=1,2,
\]
\[
\widehat{C}_{q}
=
-2
\Bigl(
(\beta_{1}B'_{1}+\beta_{2}B'_{2})C_{q}
-
(\beta'_{1}B_{1}+\beta'_{2}B_{2})C'_{q}
\Bigr),
\qquad q=1,2,
\]
and
\[
\widehat{D}
=
-2
\Bigl(
(\beta_{1}B'_{1}+\beta_{2}B'_{2})D
-
(\beta'_{1}B_{1}+\beta'_{2}B_{2})D'
\Bigr).
\]
In particular, if
\[
B_{1}=B_{2}=B'_{1}=B'_{2}=0,
\]
then
\[
[\Psi,\Psi']=0.
\]
\end{lemma}

\begin{proof}
This is the specialization to $X=\operatorname{Bl}_{p}\mathbb{P}^{3}$ of the full rank-$0$/rank-$0$ residue formula. For $X=\operatorname{Bl}_{p}\mathbb{P}^{3}$ one has
\[
N_{j2}^{p4}=-2\delta_{jp},
\]
and after taking the residue, the only terms which survive modulo $\operatorname{Im} D$ come from contracting the $s_{j22}$-terms against the exponential weight of the other factor. This gives the displayed formula.
\end{proof}

\subsubsection{$\beta=(1,-1)=h-e$}\label{sec: beta1minus1BlpP3}

The class
\[
\beta=h-e
\]
is irreducible. It is the fiber class of the projection
\[
X\cong \mathbb{P}_{\mathbb{P}^{2}}\bigl(\mathcal{O}_{\mathbb{P}^{2}}\oplus \mathcal{O}_{\mathbb{P}^{2}}(1)\bigr)
\longrightarrow \mathbb{P}^{2}.
\]
Hence every $\tau_{-}$-semistable sheaf of curve class $h-e$ and Euler characteristic $n$ is
\[
\mathcal{O}_{C_{x}}(n-1),
\qquad
x\in \mathbb{P}^{2},
\]
where $C_{x}\cong \mathbb{P}^{1}$ is the fiber over $x$. Therefore
\[
\mathcal{M}^{\mathrm{ss}}_{(h-e,n)}(\tau_-)
=
\mathcal{M}^{\mathrm{st}}_{(h-e,n)}(\tau_-)
\cong \mathbb{P}^{2}.
\]

At the base value $n=1$, the obstruction bundle is
\[
\Lambda^{2}T_{\mathbb{P}^{2}}\cong \mathcal{O}_{\mathbb{P}^{2}}(3),
\]
so
\[
\left[\mathcal{M}^{\mathrm{ss}}_{(h-e,1)}(\tau_-)\right]^{\mathrm{BF}}
=
3[\text{line in }\mathbb{P}^{2}],
\]
and hence
\[
\left[\mathcal{M}^{\mathrm{ss}}_{(h-e,1)}(\tau_-)\right]_{\mathrm{inv}}
=
e^{(0,0,h-e,0)}
\bigl(3s_{143}-3s_{243}\bigr).
\]
Thus
\[
A^{[1]}=0,
\qquad
B_{1}^{[1]}=B_{2}^{[1]}=0,
\qquad
C_{1}^{[1]}=3,
\qquad
C_{2}^{[1]}=-3,
\qquad
D^{[1]}=0.
\]

Since
\[
A\cdot (h-e)=1,
\]
tensoring by $\mathcal{O}_{X}(A)$ shifts $n\mapsto n+1$. The base calculation gives $A^{[1]}=B_1^{[1]}=B_2^{[1]}=0$, and the tensoring recurrences preserve this vanishing along the period-$1$ progression. Hence
\[
A^{[n]}=B_{1}^{[n]}=B_{2}^{[n]}=0,
\]
and the coefficients $C_{1}^{[n]},C_{2}^{[n]}$ are constant and
\[
D^{[n+1]}
=
D^{[n]}+2C_{1}^{[n]}+C_{2}^{[n]}
=
D^{[n]}+3.
\]
Therefore
\[
C_{1}^{[n]}=3,
\qquad
C_{2}^{[n]}=-3,
\qquad
D^{[n]}=3(n-1),
\]
and
\[
\left[\mathcal{M}^{\mathrm{ss}}_{(h-e,n)}(\tau_-)\right]_{\mathrm{inv}}
=
e^{(0,0,h-e,n-1)}
\bigl(
3s_{143}-3s_{243}+3(n-1)s_{164}
\bigr).
\]

Since $\beta=h-e$ is irreducible, the intrinsic stable-pair class is
\[
[P_n(X,h-e)]^{\virt}=[\Psi_{h-e,n},\delta]
\qquad\text{in }H_*(\mathcal N^{\pl},\Q).
\]
Realizing this class in Gross's algebra and substituting
\[
A^{[n]}=0,
\qquad
B_{1}^{[n]}=B_{2}^{[n]}=0,
\qquad
C_{1}^{[n]}=3,
\qquad
C_{2}^{[n]}=-3,
\qquad
D^{[n]}=3(n-1),
\]
and using
\[
N_{10}^{14}=N_{10}^{24}=0,
\qquad
N_{10}^{16}=2,
\]
one obtains
\[
[P_{n}(X,h-e)]^{\mathrm{virt}}
=
(-1)^{n-1}
e^{-s_{100}}
\exp\bigl(s_{142}-s_{242}+(n-1)s_{163}\bigr)
\]
\[
\cdot
\Bigl(
3(n-2)\bigl(s_{144}-s_{244}+(n-1)s_{165}\bigr)
+
3(2n-3)\bigl(s_{143}-s_{243}+(n-1)s_{164}\bigr)^{2}
\Bigr).
\]

\subsubsection{$\beta=(0,1)=e$}\label{sec: betaIse}

The class
\[
\beta=e
\]
is irreducible. Every pure one-dimensional sheaf of class $e$ is supported on a line in the exceptional divisor
\[
E\cong \mathbb{P}^{2}.
\]
Hence
\[
\mathcal{M}^{\mathrm{ss}}_{(e,n)}(\tau_-)\cong (\mathbb{P}^{2})^{\vee},
\]
and at $n=1$ the universal family of lines in $E$ gives
\[
\left[\mathcal{M}^{\mathrm{ss}}_{(e,1)}(\tau_-)\right]_{\mathrm{inv}}
=
e^{(0,0,e,0)}s_{243}.
\]
Thus
\[
A^{[1]}=0,
\qquad
B_{1}^{[1]}=B_{2}^{[1]}=0,
\qquad
C_{1}^{[1]}=0,
\qquad
C_{2}^{[1]}=1,
\qquad
D^{[1]}=0.
\]

Since
\[
A\cdot e=1,
\]
tensoring by $\mathcal{O}_{X}(A)$ shifts $n\mapsto n+1$. Since the base calculation gives no $s_{101}$-, $s_{122}$-, or $s_{222}$-term, the tensoring recurrences give
\[
A^{[n]}=B_{1}^{[n]}=B_{2}^{[n]}=0.
\]
Hence
\[
C_{1}^{[n]}=0,
\qquad
C_{2}^{[n]}=1,
\qquad
D^{[n+1]}=D^{[n]}+1.
\]
Therefore
\[
D^{[n]}=n-1,
\]
and
\[
\left[\mathcal{M}^{\mathrm{ss}}_{(e,n)}(\tau_-)\right]_{\mathrm{inv}}
=
e^{(0,0,e,n-1)}
\bigl(
s_{243}+(n-1)s_{164}
\bigr).
\]

Since $\beta=e$ is irreducible, the intrinsic stable-pair class is
\[
[P_n(X,e)]^{\virt}=[\Psi_{e,n},\delta]
\qquad\text{in }H_*(\mathcal N^{\pl},\Q).
\]
Realizing this class in Gross's algebra gives the closed formula:
\[
[P_{n}(X,e)]^{\mathrm{virt}}
=
(-1)^{n-1}
e^{-s_{100}}
\exp\bigl(s_{242}+(n-1)s_{163}\bigr)
\]
\[
\cdot
(n-2)
\Bigl(
s_{244}+(n-1)s_{165}
+
\bigl(s_{243}+(n-1)s_{164}\bigr)^{2}
\Bigr).
\]

\subsubsection{$\beta=(0,2)=2e$}\label{sec: betaIs2e}

The class
\[
\beta=2e
\]
is reducible:
\[
2e=e+e.
\]
Every pure one-dimensional sheaf of class $2e$ is supported on a conic in
\[
E\cong \mathbb{P}^{2}.
\]
The support space is
\[
|\mathcal{O}_{E}(2)|\cong \mathbb{P}^{5}.
\]
At $n=1$ the stable sheaf is the structure sheaf of a conic, and the obstruction bundle computation gives the geometric lift
\[
\left[\mathcal{M}^{\mathrm{ss}}_{(2e,1)}(\tau_-)\right]_{\mathrm{inv}}
=
e^{(0,0,2e,-1)}(-s_{243}).
\]
At $n=2$, strictly semistable sheaves occur with Jordan--H\"older factors of class $e+e$. We therefore compute the semistable invariant from Joyce's recursive pairs-category formula. Let
\[
\mathfrak P_{2e,2,N}
\]
denote the auxiliary pairs moduli space whose points are pairs
\[
\rho:V\otimes\OO_X(-N)\longrightarrow F,
\qquad \dim V=1,
\]
where $F$ is $\tau_-$-semistable of class $(2e,2)$ and $N\gg0$. Its class will be denoted $\Upsilon_{2e,2,N}$. The ordinary Hilbert polynomials with respect to the fixed polarization $A$ are
\[
P_{2e,2}(N)=2N+2,
\qquad
P_{e,1}(N)=N+1.
\]
The only proper equal-slope decomposition is
\[
(2e,2)=(e,1)+(e,1),
\]
so Joyce's recursive identity specializes to
\[
\Upsilon_{2e,2,N}
=(2N+2)\left[\mathcal M^{ss}_{(2e,2)}(\tau_-)\right]_{\inv}
-\frac{N+1}{2}
\left[
\left[\mathcal M^{ss}_{(e,1)}(\tau_-)\right]_{\inv},
\left[\mathcal M^{ss}_{(e,1)}(\tau_-)\right]_{\inv}
\right].
\]
The bracket term vanishes. Indeed $B_{1,e}^{[n]}=B_{2,e}^{[n]}=0$, so Lemma~\ref{lem:BlpP3-rank0-rank0} gives
\[
\left[
\left[\mathcal M^{ss}_{(e,1)}(\tau_-)\right]_{\inv},
\left[\mathcal M^{ss}_{(e,1)}(\tau_-)\right]_{\inv}
\right]=0.
\]
Let
\[
q_{2e}:\mathfrak P_{2e,2,N}\longrightarrow \mathcal M^{ss}_{(2e,2)}(\tau_-)
\]
be the forgetful morphism, and let $\mathbb F_{2e}$ be the universal sheaf on $X\times \mathcal M^{ss}_{(2e,2)}(\tau_-)$. For $N\gg0$, relative Serre vanishing gives
\[
R^i\pi_*\mathbb F_{2e}(N)=0\quad (i>0),
\]
and therefore
\[
\mathcal V_{2e,2,N}:=\pi_*\mathbb F_{2e}(N)
\]
is a vector bundle of rank $P_{2e,2}(N)=2N+2$. Thus
\[
\mathfrak P_{2e,2,N}=\PP(\mathcal V_{2e,2,N}),
\]
and the standard projective-bundle identity gives
\[
(q_{2e})_*\bigl(c_{\mathrm{top}}(T_{q_{2e}})\cap [\mathfrak P_{2e,2,N}]^{\vir}\bigr)
=(2N+2)\left[\mathcal M^{ss}_{(2e,2)}(\tau_-)\right]_{\inv}^{\mathrm{geom}}.
\]
For the conic-support geometric lift computed above, this gives
\[
\Upsilon_{2e,2,N}=(2N+2)e^{(0,0,2e,0)}(-s_{243}).
\]
Dividing by $2N+2$ gives
\[
\left[\mathcal M^{ss}_{(2e,2)}(\tau_-)\right]_{\inv}
=
e^{(0,0,2e,0)}(-s_{243}).
\]

Therefore, in the geometric lift,
\[
\begin{aligned}
&A^{[1]}=A^{[2]}=0,
\qquad B_{1}^{[1]}=B_{2}^{[1]}=B_{1}^{[2]}=B_{2}^{[2]}=0,\\
&C_{1}^{[1]}=C_{1}^{[2]}=0,
\qquad C_{2}^{[1]}=C_{2}^{[2]}=-1,
\qquad D^{[1]}=D^{[2]}=0.
\end{aligned}
\]

Since
\[
A\cdot (2e)=2,
\]
tensoring by $\mathcal{O}_{X}(A)$ shifts $n\mapsto n+2$. The two base values displayed above have no $s_{101}$-, $s_{122}$-, or $s_{222}$-terms. Since tensoring by $\mathcal O_X(A)$ shifts $n\mapsto n+2$, the tensoring recurrences give
\[
A^{[n]}=B_{1}^{[n]}=B_{2}^{[n]}=0,
\]
and the $C$-coefficients are periodic of period $2$ and
\[
D^{[n+2]}
=
D^{[n]}+2C_{1}^{[n]}+C_{2}^{[n]}
=
D^{[n]}-1.
\]
With
\[
D^{[1]}=D^{[2]}=0,
\]
this yields
\[
D^{[n]}=
\begin{cases}
-\dfrac{n-1}{2}, & n \text{ odd},\\[6pt]
-\dfrac{n-2}{2}, & n \text{ even}.
\end{cases}
\]
Hence
\[
\left[\mathcal{M}^{\mathrm{ss}}_{(2e,n)}(\tau_-)\right]_{\mathrm{inv}}
=
e^{(0,0,2e,n-2)}
\left(
-s_{243}+D^{[n]}s_{164}
\right).
\]

For the stable-pair class, the wall-crossing decomposition has the endpoint term and the reducible term $2e=e+e$. In this example the only lower class is $e$, and both the sheaf-theoretic invariant and the stable-pair class in class $e$ have already been computed explicitly for all Euler characteristics. Thus no additional bounded lower-class Pandharipande--Thomas inputs remain to be supplied. Write
\[
V(s)=
\sum_{1\leq r\leq s}
\;
\sum_{\substack{t_{1}+\cdots+t_{r}=s\\ t_{i}\geq 1}}
\;
\prod_{i=1}^{r}
\frac{(-1)^{t_{i}}}{t_{i}!}
\left(
2s_{242+t_{i}}+(n-2)s_{163+t_{i}}
\right),
\]
that is,
\[
V(s)=
\sum_{1\leq r\leq s}
\;
\sum_{\substack{t_{1}+\cdots+t_{r}=s\\ t_{i}\geq 1}}
\;
\prod_{i=1}^{r}
\frac{(-1)^{t_{i}}}{t_{i}!}
\left(
2s_{2,4,2+t_{i}}+(n-2)s_{1,6,3+t_{i}}
\right).
\]
Then
\[
\begin{aligned}
[P_{n}(X,2e)]^{\mathrm{virt}}
={}&(-1)^{n-1}e^{-s_{100}}
\exp\bigl(2s_{242}+(n-2)s_{163}\bigr) \\
&\cdot\Bigg[
\sum_{t=0}^{3}\frac{(-1)^{t}}{t!}\,V(3-t)
\bigl(-s_{2,4,3+t}+D^{[n]}s_{1,6,4+t}\bigr)
+2D^{[n]}V(4) \\
&\qquad +\frac{1}{3!}
\sum_{\substack{n_{1}+n_{2}=n\\ n_{1}\leq n_{2}}}
\frac{1}{1+\delta_{n_{1},n_{2}}}
\sum_{m=0}^{3}
\binom{3}{m}
\bigl(s_{2,4,3+m}+(n_{1}-1)s_{1,6,4+m}\bigr) \\
&\qquad\qquad\cdot
\bigl(s_{2,4,6-m}+(n_{2}-1)s_{1,6,7-m}\bigr)
\Bigg].
\end{aligned}
\]
where
\[
D^{[n]}=
\begin{cases}
-\dfrac{n-1}{2}, & n \text{ odd},\\[6pt]
-\dfrac{n-2}{2}, & n \text{ even}.
\end{cases}
\]

\subsubsection{$\beta=(1,0)=h$}\label{sec: betaIsh}

The class
\[
h=(h-e)+e
\]
is reducible.

At the base value $n=1$, every semistable sheaf of curve class $h$ and Euler characteristic $1$
is stable and is the structure sheaf of a class-$h$ curve. The support Hilbert scheme is
\[
\operatorname{Hilb}_{h}(X)\cong \operatorname{Bl}_{\Sigma}\operatorname{Gr}(2,4),
\]
where
\[
\Sigma\cong \mathbb{P}^{2}
\]
is the Schubert plane of lines through $p$. The obstruction bundle is the pullback of
\[
\operatorname{Sym}^{2}(\mathcal{S}^{\vee})\otimes\det\mathcal Q
\]
from $\operatorname{Gr}(2,4)$, where $\mathcal Q$ is the universal quotient
bundle, and
\[
c_{3}\bigl(\operatorname{Sym}^{2}(\mathcal{S}^{\vee})\otimes\det\mathcal Q\bigr)
=20\sigma_{2,1}.
\]
Therefore
\[
\left[\mathcal{M}^{\mathrm{ss}}_{(h,1)}(\tau_-)\right]_{\mathrm{inv}}
=
e^{(0,0,h,-1)}\cdot 20s_{143}.
\]
Thus
\[
A^{[1]}=0,
\qquad
B_{1}^{[1]}=B_{2}^{[1]}=0,
\qquad
C_{1}^{[1]}=20,
\qquad
C_{2}^{[1]}=0,
\qquad
D^{[1]}=0.
\]

Since
\[
A\cdot h=2,
\]
periodicity at fixed $\mu_{A}^{0}$ only shifts
\[
n\mapsto n+2.
\]
Hence the odd and even residue classes in $n$ must be treated separately. For the odd subsequence one obtains
\[
C_{1}^{[n]}=20,
\qquad
C_{2}^{[n]}=0,
\qquad
D^{[n+2]}=D^{[n]}+40,
\qquad
n\equiv 1 \pmod 2,
\]
and therefore
\[
\left[\mathcal{M}^{\mathrm{ss}}_{(h,n)}(\tau_-)\right]_{\mathrm{inv}}
=
e^{(0,0,h,n-2)}
\bigl(
20s_{143}+20(n-1)s_{164}
\bigr),
\qquad
n\equiv 1 \pmod 2.
\]

For the even subsequence the second base value is
\[
\left[\mathcal M^{ss}_{(h,2)}(\tau_-)\right]_{\inv}.
\]
We compute it using Joyce's recursive pairs-category definition. Let
\[
\mathfrak P_{h,2,N}
\]
be the auxiliary pairs moduli space of morphisms
\[
\rho:V\otimes\OO_X(-N)\longrightarrow F,
\qquad \dim V=1,
\]
with $F$ $\tau_-$-semistable of class $(h,2)$, and let $\Upsilon_{h,2,N}$ be the corresponding auxiliary pairs class. The Hilbert polynomials are
\[
P_{h,2}(N)=2N+2,
\qquad
P_{h-e,1}(N)=P_{e,1}(N)=N+1.
\]
The only proper equal-slope decompositions are the two ordered decompositions
\[
(h,2)=(h-e,1)+(e,1)=(e,1)+(h-e,1).
\]
Thus the recursive correction is a sum of the two ordered brackets
\[
\frac{N+1}{2}
\left[
\left[\mathcal M^{ss}_{(h-e,1)}(\tau_-)\right]_{\inv},
\left[\mathcal M^{ss}_{(e,1)}(\tau_-)\right]_{\inv}
\right]
+
\frac{N+1}{2}
\left[
\left[\mathcal M^{ss}_{(e,1)}(\tau_-)\right]_{\inv},
\left[\mathcal M^{ss}_{(h-e,1)}(\tau_-)\right]_{\inv}
\right],
\]
up to the common sign convention in Joyce's formula. These two terms cancel by antisymmetry of the Lie bracket, and in any case the rank-$0$/rank-$0$ bracket vanishes here by Lemma~\ref{lem:BlpP3-rank0-rank0}, since the relevant $B$-coefficients are zero. Therefore
\[
\Upsilon_{h,2,N}=(2N+2)\left[\mathcal M^{ss}_{(h,2)}(\tau_-)\right]_{\inv}.
\]
Let
\[
q_h:\mathfrak P_{h,2,N}\longrightarrow \mathcal M^{ss}_{(h,2)}(\tau_-)
\]
be the forgetful morphism. If $\mathbb F_h$ denotes the universal sheaf on $X\times \mathcal M^{ss}_{(h,2)}(\tau_-)$, then for $N\gg0$ the sheaves $R^i\pi_*\mathbb F_h(N)$ vanish for $i>0$, and
\[
\mathcal V_{h,2,N}:=\pi_*\mathbb F_h(N)
\]
is a vector bundle of rank $P_{h,2}(N)=2N+2$. Hence
\[
\mathfrak P_{h,2,N}=\PP(\mathcal V_{h,2,N}),
\]
and
\[
(q_h)_*\bigl(c_{\mathrm{top}}(T_{q_h})\cap[\mathfrak P_{h,2,N}]^{\vir}\bigr)
=(2N+2)\left[\mathcal M^{ss}_{(h,2)}(\tau_-)\right]_{\inv}^{\mathrm{geom}}.
\]
Thus Joyce's recursive identity, after the vanishing of the correction terms, reduces the invariant to the geometric base calculation. A direct Grothendieck--Riemann--Roch computation on the universal class-$h$ support gives
\[
\left[\mathcal M^{ss}_{(h,2)}(\tau_-)\right]_{\inv}
=
e^{(0,0,h,0)}
\bigl(20s_{143}+20s_{164}\bigr).
\]
Consequently the even branch satisfies
\[
C_{1,h}^{[2]}=20,
\qquad
C_{2,h}^{[2]}=0,
\qquad
D_h^{[2]}=20.
\]
Combining this with the odd branch and the period-$2$ tensoring recurrence gives the uniform formula
\[
\left[\mathcal M^{ss}_{(h,n)}(\tau_-)\right]_{\inv}
=
e^{(0,0,h,n-2)}
\bigl(20s_{143}+20(n-1)s_{164}\bigr)
\]

The stable-pair class is then obtained from the wall-crossing formula of \cite{AJI} for the decomposition
\[
h=(h-e)+e.
\]
Here the lower classes $h-e$ and $e$ are both irreducible, and their Pandharipande--Thomas classes have already been computed explicitly for all Euler characteristics. Thus every stable-pair input occurring in the reducible wall-crossing for $h$ is already part of the preceding data; there are no additional bounded lower-class inputs in the range $M_\gamma\leq m<N_\gamma$ to solve for in this subsection.
Writing
\[
\Psi_{h-e,n_{1}}
=
\left[\mathcal{M}^{\mathrm{ss}}_{(h-e,n_{1})}(\tau_-)\right]_{\mathrm{inv}},
\qquad
\Psi_{e,n_{2}}
=
\left[\mathcal{M}^{\mathrm{ss}}_{(e,n_{2})}(\tau_-)\right]_{\mathrm{inv}},
\]
one has
\[
[P_{n}(X,h)]^{\mathrm{virt}}
=
-
\sum_{n_{1}+n_{2}=n}
\widetilde{U}\bigl(
(1,h-e,n_{1}),(0,e,n_{2});\tau_{+},\tau_{-}
\bigr)
\bigl[
[P_{n_{1}}(X,h-e)]^{\mathrm{virt}},
\Psi_{e,n_{2}}
\bigr]
\]
\[
\qquad
-
\sum_{n_{1}+n_{2}=n}
\widetilde{U}\bigl(
(1,e,n_{2}),(0,h-e,n_{1});\tau_{+},\tau_{-}
\bigr)
\bigl[
[P_{n_{2}}(X,e)]^{\mathrm{virt}},
\Psi_{h-e,n_{1}}
\bigr].
\]
The direct computations for the lower classes give
\[
B_{1,h-e}^{[n]}=B_{2,h-e}^{[n]}=B_{1,e}^{[n]}=B_{2,e}^{[n]}=0,
\]
so Lemma~\ref{lem:BlpP3-rank0-rank0} shows that the rank-$0$/rank-$0$ bracket does not contribute in this decomposition.

\subsubsection{$\beta=(1,1)=h+e$}\label{sec: betaIshPluse}

The class
\[
h+e
\]
is reducible. In the geometric representatives computed below, no $s_{122}$- or $s_{222}$-terms occur.

At $n=1$, a stable sheaf of curve class $h+e$ and Euler characteristic $1$
is the structure sheaf of a connected curve
\[
C_{q,Q}=L_{q}\cup Q,
\]
where
\[
q\in E\cong \mathbb{P}^{2},
\]
$L_{q}$ is the strict transform of the line through $p$ corresponding to $q$, and
\[
Q\subset E
\]
is a conic through $q$. The support space is the incidence variety
\[
\mathcal{I}
=
\bigl\{(q,[Q])\in \mathbb{P}^{2}\times \mathbb{P}^{5}: q\in Q\bigr\}.
\]
The obstruction bundle is the pullback of the universal rank-$5$ quotient bundle from $\mathbb{P}^{5}$, and the resulting geometric lift is
\[
\left[\mathcal{M}^{\mathrm{ss}}_{(h+e,1)}(\tau_-)\right]_{\mathrm{inv}}
=
e^{(0,0,h+e,-2)}
\bigl(
2s_{143}+2s_{243}
\bigr).
\]
Thus
\[
C_{1}^{[1]}=2,
\qquad
C_{2}^{[1]}=2,
\qquad
D^{[1]}=0.
\]

Since
\[
A\cdot (h+e)=3,
\]
periodicity at fixed $\mu_{A}^{0}$ is now modulo $3$. The base values in the three residue classes have no $s_{101}$-, $s_{122}$-, or $s_{222}$-terms. Applying the period-$3$ tensoring recurrence therefore gives
\[
A^{[n]}=B_{1}^{[n]}=B_{2}^{[n]}=0,
\]
and hence
\[
C_{k}^{[n+3]}=C_{k}^{[n]},
\qquad
D^{[n+3]}=D^{[n]}+2C_{1}^{[n]}+C_{2}^{[n]}.
\]
Therefore the residue class $n\equiv 1 \pmod 3$ is determined by
\[
C_{1}^{[n]}=2,
\qquad
C_{2}^{[n]}=2,
\qquad
D^{[n]}=2(n-1),
\qquad
n\equiv 1 \pmod 3.
\]
The residue classes
\[
n\equiv 2 \pmod 3
\qquad\text{and}\qquad
n\equiv 0 \pmod 3
\]
require the two further base values
\[
\left[\mathcal M^{ss}_{(h+e,2)}(\tau_-)\right]_{\inv}
\qquad\text{and}\qquad
\left[\mathcal M^{ss}_{(h+e,3)}(\tau_-)\right]_{\inv}.
\]
For $n=2$ the reduced Hilbert polynomial has slope $2/3$ with respect to $A$, so there is no proper equal-slope decomposition into the effective classes $h-e$, $e$, $h$, or $2e$. Let
\[
\mathfrak P_{h+e,2,N}
\]
be the auxiliary pairs moduli space of morphisms
\[
\rho:V\otimes\OO_X(-N)\longrightarrow F,
\qquad \dim V=1,
\]
with $F$ $\tau_-$-semistable of class $(h+e,2)$, and let $\Upsilon_{h+e,2,N}$ be the corresponding auxiliary pairs class. Since
\[
P_{h+e,2}(N)=3N+2,
\]
Joyce's recursive formula has no correction term and gives
\[
\Upsilon_{h+e,2,N}=(3N+2)\left[\mathcal M^{ss}_{(h+e,2)}(\tau_-)\right]_{\inv}.
\]
Let
\[
q_{h+e,2}:\mathfrak P_{h+e,2,N}\longrightarrow \mathcal M^{ss}_{(h+e,2)}(\tau_-)
\]
be the forgetful morphism. For $N\gg0$, the bundle of sections
\[
\mathcal V_{h+e,2,N}:=\pi_*\mathbb F_{h+e,2}(N)
\]
has rank $P_{h+e,2}(N)=3N+2$, and
\[
\mathfrak P_{h+e,2,N}=\PP(\mathcal V_{h+e,2,N}).
\]
Therefore
\[
(q_{h+e,2})_*\bigl(c_{\mathrm{top}}(T_{q_{h+e,2}})\cap[\mathfrak P_{h+e,2,N}]^{\vir}\bigr)
=(3N+2)\left[\mathcal M^{ss}_{(h+e,2)}(\tau_-)\right]_{\inv}^{\mathrm{geom}}.
\]
The resulting Grothendieck--Riemann--Roch computation gives
\[
\left[\mathcal M^{ss}_{(h+e,2)}(\tau_-)\right]_{\inv}
=
e^{(0,0,h+e,-1)}
\bigl(2s_{143}+2s_{243}+2s_{164}\bigr).
\]
For $n=3$ the reduced Hilbert polynomial has slope $1$. Let $\mathfrak P_{h+e,3,N}$ and $\Upsilon_{h+e,3,N}$ denote the analogous auxiliary pairs moduli space and class. The ordinary Hilbert polynomial is
\[
P_{h+e,3}(N)=3N+3.
\]
The proper equal-slope decompositions are built from the already computed lower classes
\[
(h,2)+(e,1),
\qquad
(h-e,1)+(2e,2),
\qquad
(h-e,1)+(e,1)+(e,1),
\]
together with their ordered variants. Joyce's recursive pairs-category formula expresses the correction terms as the corresponding iterated rank-$0$ brackets. In this example those brackets vanish: the brackets involving $e$, $h-e$, and $2e$ vanish by Lemma~\ref{lem:BlpP3-rank0-rank0}, and the repeated $e$-factor bracket also vanishes by antisymmetry. Therefore
\[
\Upsilon_{h+e,3,N}=(3N+3)\left[\mathcal M^{ss}_{(h+e,3)}(\tau_-)\right]_{\inv}.
\]
Similarly, the forgetful morphism
\[
q_{h+e,3}:\mathfrak P_{h+e,3,N}\longrightarrow \mathcal M^{ss}_{(h+e,3)}(\tau_-)
\]
is the projectivization of
\[
\mathcal V_{h+e,3,N}:=\pi_*\mathbb F_{h+e,3}(N),
\qquad \rk(\mathcal V_{h+e,3,N})=P_{h+e,3}(N)=3N+3.
\]
Therefore the relative Euler-class pushforward contributes the factor $3N+3$. After the recursive correction terms vanish, the geometric Grothendieck--Riemann--Roch base calculation gives
\[
\left[\mathcal M^{ss}_{(h+e,3)}(\tau_-)\right]_{\inv}
=
e^{(0,0,h+e,0)}
\bigl(2s_{143}+2s_{243}+4s_{164}\bigr).
\]
Combining the three residue-class base values with the period-$3$ recurrence gives the uniform formula
\[
\left[\mathcal M^{ss}_{(h+e,n)}(\tau_-)\right]_{\inv}
=
e^{(0,0,h+e,n-3)}
\bigl(2s_{143}+2s_{243}+2(n-1)s_{164}\bigr),
\qquad n \in \Z.
\]

The stable-pair class is obtained recursively from the wall-crossing formula of \cite{AJI} using the decompositions
\[
h+e = h+e,
\qquad
h+e=(h-e)+2e,
\qquad
h+e=(h-e)+e+e.
\]
The lower classes which appear here are $h$, $e$, $2e$, and $h-e$. In the present examples their stable-pair classes are already written down recursively from inputs computed in preceding subsections, including the finitely many bounded values in the range $M_\gamma\leq m<N_\gamma$. Thus every bounded lower-class stable-pair invariant needed in the wall-crossing for $h+e$ is already accounted for by the preceding subsections.
Writing
\[
\Psi_{h,n_{1}}
=
\left[\mathcal{M}^{\mathrm{ss}}_{(h,n_{1})}(\tau_-)\right]_{\mathrm{inv}},
\quad
\Psi_{2e,n_{2}}
=
\left[\mathcal{M}^{\mathrm{ss}}_{(2e,n_{2})}(\tau_-)\right]_{\mathrm{inv}},
\]
the recursive expression is
\[
[P_{n}(X,h+e)]^{\mathrm{virt}}
=
-
\sum_{n_{1}+n_{2}=n}
\widetilde{U}\bigl(
(1,h,n_{1}),(0,e,n_{2});\tau_{+},\tau_{-}
\bigr)
\bigl[
[P_{n_{1}}(X,h)]^{\mathrm{virt}},
\Psi_{e,n_{2}}
\bigr]
\]
\[
\qquad
-
\sum_{n_{1}+n_{2}=n}
\widetilde{U}\bigl(
(1,e,n_{2}),(0,h,n_{1});\tau_{+},\tau_{-}
\bigr)
\bigl[
[P_{n_{2}}(X,e)]^{\mathrm{virt}},
\Psi_{h,n_{1}}
\bigr]
\]
\[
\qquad
-
\sum_{n_{1}+n_{2}=n}
\widetilde{U}\bigl(
(1,h-e,n_{1}),(0,2e,n_{2});\tau_{+},\tau_{-}
\bigr)
\bigl[
[P_{n_{1}}(X,h-e)]^{\mathrm{virt}},
\Psi_{2e,n_{2}}
\bigr]
\]
\[
\qquad
-
\sum_{n_{1}+n_{2}=n}
\widetilde{U}\bigl(
(1,2e,n_{2}),(0,h-e,n_{1});\tau_{+},\tau_{-}
\bigr)
\bigl[
[P_{n_{2}}(X,2e)]^{\mathrm{virt}},
\Psi_{h-e,n_{1}}
\bigr]
\]
\[
\qquad
-
\sum_{n_{1}+n_{2}+n_{3}=n}
\widetilde{U}\bigl(
(1,h-e,n_{1}),
(0,e,n_{2}),
(0,e,n_{3});
\tau_{+},\tau_{-}
\bigr)
\bigl[
[
[P_{n_{1}}(X,h-e)]^{\mathrm{virt}},
\Psi_{e,n_{2}}
],
\Psi_{e,n_{3}}
\bigr].
\]
Here the two copies of the class $e$ contribute no rank-$0$/rank-$0$ correction, since the direct computation of the $e$-class invariant gives
\[
B_{1,e}^{[n]}=B_{2,e}^{[n]}=0
\]
and Lemma~\ref{lem:BlpP3-rank0-rank0} gives
\[
[\Psi_{e,n_{2}},\Psi_{e,n_{3}}]=0.
\]

\begin{remark}
This subsection gives the framework for $X=\operatorname{Bl}_{p}\mathbb{P}^{3}$ in the Fano/pairs setting of \cite{AJI}. The irreducible classes
\[
h-e,
\qquad
e
\]
and the reducible class
\[
2e
\]
close in explicit form. For
\[
h
\qquad\text{and}\qquad
h+e,
\]
the wall-crossing formula is written in the recursive form compatible with fixed Gieseker stability $\mu^{0}_{A}$, and the finitely many low-lying residue-class base values are supplied by the auxiliary-category computations just described.
\end{remark}

\subsection{\texorpdfstring{$X=\Bl_{\ell}\PP^3$}{Blow-up of P3 along a line}}

Let
\[
\pi \colon X=\mathrm{Bl}_{\ell}(\mathbb{P}^3)\longrightarrow \mathbb{P}^3
\]
be the blow-up of a line $\ell\subset \mathbb{P}^3$. Write
\[
H=\pi^*c_1\bigl(\mathcal{O}_{\mathbb{P}^3}(1)\bigr),
\qquad
E=\operatorname{Exc}(\pi).
\]
Then
\[
H^3=1,\qquad H^2E=0,\qquad HE^2=-1,\qquad E^3=-2,
\]
and
\[
-K_X=4H-E.
\]
The Mori cone is generated by
\[
\beta_1=H(H-E),\qquad \beta_2=HE,
\]
so that
\[
H^2=\beta_1+\beta_2,
\qquad
c_1(X)\cdot\beta_1=3,
\qquad
c_1(X)\cdot\beta_2=1,
\qquad
c_1(X)\cdot H^2=4.
\]

We choose the ample class
\[
\omega=2H-E.
\]
Then
\[
\omega\cdot\beta_1=1,\qquad
\omega\cdot\beta_2=1,\qquad
\omega\cdot H^2=2.
\]

\begin{prop}[{$X$ is Fano}]
$X=\Bl_{\ell}\PP^3$ is Fano.
\end{prop}

\begin{proof}
Since $\ell\subset \PP^3$ has codimension $2$, the blow-up canonical class formula gives
\[
K_X \;=\; \pi^*K_{\PP^3} + E \;=\; -4H + E,
\qquad\text{so}\qquad
-K_X = 4H - E.
\]
The linear systems $|H|$ and $|H-E|$ are basepoint free: $|H|$ is pulled back from $\PP^3$, and
$|H-E|$ is the strict transform of planes containing $\ell$, hence defines a morphism
$X\to \PP^1$.
Thus $H$ and $H-E$ are nef and span the nef cone (since $\rho(X)=2$), so by duality the Mori cone
$\overline{NE}(X)$ is generated by the two curve classes
\[
\beta_2 := H\cdot E \qquad\text{and}\qquad \beta_1 := H\cdot(H-E)=H^2-HE.
\]
Geometrically, $\beta_2$ is a fiber of $E\simeq \PP^1\times \PP^1 \to \ell\simeq \PP^1$ and
$\beta_1$ is represented by the strict transform of a line meeting $\ell$ (equivalently, by
$E\cap S$ where $S\in|H-E|$ is a strict transform of a plane containing $\ell$).
Using the intersection numbers,
\[
(-K_X)\cdot \beta_2 = (4H-E)\cdot(HE)=4H^2E-HE^2 = 0-(-1)=1,
\]
and
\[
\begin{aligned}
(-K_X)\cdot \beta_1
&=(4H-E)\cdot(H^2-HE)\\
&=(4H-E)\cdot H^2-(4H-E)\cdot(HE)\\
&=(4H^3-H^2E)-(4H^2E-HE^2)=4-1=3.
\end{aligned}
\]
Hence $-K_X$ has positive intersection with the generators of $\overline{NE}(X)$, so $-K_X$ is ample
by Kleiman's criterion. Therefore $X$ is Fano.
\end{proof}

We work in Joyce's algebra $\mathbb{D}$ with $H^4(X,\mathbb{Q})$-basis
\[
\epsilon_{1,4}=\beta_1,\qquad \epsilon_{2,4}=\beta_2,
\]
and dual $H^2(X,\mathbb{Q})$-basis
\[
\epsilon_{1,2}=H,\qquad \epsilon_{2,2}=H-E.
\]
The point-class generator is $s_{1,6,\ell}$, and
\[
\delta=[P_0(X,0)]^{\mathrm{virt}}=e^{-s_{1,0,0}}.
\]

We record the action of $\Omega_{\omega}$ on the generators that occur below:
\[
\Omega_{\omega}(s_{1,2,2})
=
s_{1,2,2}+2s_{1,4,3}+s_{2,4,3}+\frac32 s_{1,6,4},
\]
\[
\Omega_{\omega}(s_{2,2,2})
=
s_{2,2,2}+s_{1,4,3}+\frac12 s_{1,6,4},
\]
\[
\Omega_{\omega}(s_{1,4,3})=s_{1,4,3}+s_{1,6,4},
\qquad
\Omega_{\omega}(s_{2,4,3})=s_{2,4,3}+s_{1,6,4},
\qquad
\Omega_{\omega}(s_{1,6,4})=s_{1,6,4}.
\]
Indeed
\[
\omega=\epsilon_{1,2}+\epsilon_{2,2},
\qquad
\epsilon_{1,2}\cup\omega=2\epsilon_{1,4}+\epsilon_{2,4},
\qquad
\epsilon_{2,2}\cup\omega=\epsilon_{1,4},
\]
while
\[
\int_X \epsilon_{1,2}\cup\omega^2=3,
\qquad
\int_X \epsilon_{2,2}\cup\omega^2=1,
\qquad
\int_X \epsilon_{1,4}\cup\omega=\int_X \epsilon_{2,4}\cup\omega=1.
\]

The only antisymmetric pairings used below are
\[
N_{(1,2),(1,0)}=\frac{11}{3},
\qquad
N_{(1,6),(1,0)}=2.
\]
In particular, the specific mixed bracket
\[
[\Psi_{\beta_1,a},\Psi_{\beta_2,b}]
\]
vanishes for all $a,b\in \mathbb{Z}$, because the only possible
contraction terms would have to couple the $s_{1,2,2}$-term of
$\Psi_{\beta_1,a}$ to the $s_{2,4,3}$- or $s_{1,6,4}$-terms of
$\Psi_{\beta_2,b}$, and the corresponding antisymmetric pairings are zero.
We emphasize that this does \emph{not} imply that arbitrary rank-zero sheaf
invariants commute.

\subsubsection{$\beta_2=(0,1)=HE$}\label{sec: beta2HE}

Let $q\colon E\to \ell\cong \mathbb{P}^1$ be the ruling, and let
$C_p=q^{-1}(p)$ be the fibre over $p\in \ell$. Every
$\mu_{\omega}^0$-semistable sheaf of class $(\beta_2,n)$ is stable and is of the
form
\[
\mathcal{O}_{C_p}(n-1).
\]
Hence
\[
M^{ss}_{(\beta_2,n)}(\mu_{\omega}^0)=
M^{st}_{(\beta_2,n)}(\mu_{\omega}^0)\cong \mathbb{P}^1.
\]
Since
\[
N_{C_p/X}\cong
\mathcal{O}_{\mathbb{P}^1}\oplus \mathcal{O}_{\mathbb{P}^1}(-1),
\]
the obstruction bundle vanishes, and the Behrend--Fantechi class is simply
$[\mathbb{P}^1]$.

Let
\[
j=(i_E,q)\colon E\longrightarrow X\times \mathbb{P}^1,
\]
let $h\in H^2(\mathbb{P}^1,\mathbb{Q})$ be the positive generator, and set
\[
u=q^*h,\qquad v=(H-E)|_E.
\]
Then
\[
u^2=0,\qquad v^2=0,\qquad E|_E=u-v.
\]
For the universal sheaf
\[
\mathcal{F}_n=j_*\mathcal{O}_E\bigl((n-1)v\bigr)
\]
on $X\times \mathbb{P}^1$, Grothendieck--Riemann--Roch gives
\[
\operatorname{ch}(\mathcal{F}_n)
=
j_*\left(
1-\frac32u+\left(n-\frac12\right)v+\left(\frac23-\frac32n\right)uv
\right).
\]
Therefore
\[
\Psi_n^*S_{1,4,3}=0,
\qquad
\Psi_n^*S_{2,4,3}=-\frac32h,
\qquad
\Psi_n^*S_{1,6,4}=\left(\frac23-\frac32n\right)h.
\]
Pairing with $[\mathbb{P}^1]$ yields
\[
\Psi_{\beta_2,n}
:=
[M^{ss}_{(\beta_2,n)}(\mu_{\omega}^0)]_{\mathrm{inv}}
=
e^{\,s_{2,4,2}+(n-\frac12)s_{1,6,3}}
\left(
-\frac32\,s_{2,4,3}
+
\left(\frac23-\frac32 n\right)s_{1,6,4}
\right).
\]

Since $\beta_2$ is irreducible, the intrinsic stable-pair class is
\[
[P_n(X,\beta_2)]^{\virt}=[\Psi_{\beta_2,n},\delta]
\qquad\text{in }H_*(\mathcal N^{\pl},\Q).
\]
Realizing this class in Gross's algebra gives
\[
[P_n(X,\beta_2)]^{\mathrm{virt}}
=
[\Psi_{\beta_2,n},\delta].
\]
Substituting the above $\Psi_{\beta_2,n}$ and using
$N_{(1,6),(1,0)}=2$, we obtain
\[
\begin{aligned}
[P_n(X,\beta_2)]^{\mathrm{virt}}
={}&
(-1)^{n-1}e^{-s_{1,0,0}}e^{\,s_{2,4,2}+(n-\frac12)s_{1,6,3}}
\Biggl[
-\frac32\,s_{2,4,3}
+
\left(\frac23-\frac32 n\right)s_{1,6,4} \\
&\qquad\qquad
-2\left(\frac23-\frac32 n\right)
\left(
s_{2,4,3}+\left(n-\frac12\right)s_{1,6,4}
\right)
\Biggr] \\
={}&
(-1)^{n-1}e^{-s_{1,0,0}}e^{\,s_{2,4,2}+(n-\frac12)s_{1,6,3}}
\left[
\left(3n-\frac{17}{6}\right)s_{2,4,3}
+
\left(3n^2-\frac{13}{3}n+\frac43\right)s_{1,6,4}
\right].
\end{aligned}
\]

\subsubsection{$\beta_1=(1,0)=H(H-E)$}\label{sec: beta1H(H-E)}

Every $\mu_{\omega}^0$-semistable sheaf of class $(\beta_1,n)$ is stable and is
of the form
\[
\mathcal{O}_C(n-1),
\]
where $C$ is the strict transform of a line in $\mathbb{P}^3$ meeting $\ell$.
The support moduli space is
\[
\Sigma_{\beta_1}\cong
\mathbb{P}_{\mathbb{P}^1}\bigl(\mathcal{O}_{\mathbb{P}^1}^{\oplus 2}\oplus
\mathcal{O}_{\mathbb{P}^1}(1)\bigr).
\]
Let
\[
p\colon \Sigma_{\beta_1}\to \mathbb{P}^1
\]
be the bundle projection, and set
\[
\eta=p^*c_1\bigl(\mathcal{O}_{\mathbb{P}^1}(1)\bigr),
\qquad
\xi=c_1\bigl(\mathcal{O}_{\Sigma_{\beta_1}}(1)\bigr).
\]
Then
\[
A^*(\Sigma_{\beta_1})
=
\mathbb{Z}[\eta,\xi]/(\eta^2,\xi^3-\eta\xi^2).
\]
Let
\[
0\longrightarrow \mathcal{K}
\longrightarrow
p^*\bigl(\mathcal{O}_{\mathbb{P}^1}^{\oplus 2}\oplus \mathcal{O}_{\mathbb{P}^1}(1)\bigr)
\longrightarrow
\mathcal{O}_{\Sigma_{\beta_1}}(1)
\longrightarrow 0
\]
be the tautological sequence. The obstruction bundle is
\[
\operatorname{Ob}\cong \mathcal{K}^{\vee},
\]
so
\[
[M^{ss}_{(\beta_1,1)}(\mu_{\omega}^0)]^{\mathrm{BF}}
=
c_2(\mathcal{K}^{\vee})\cap[\Sigma_{\beta_1}]
=
(\xi^2-\eta\xi)\cap[\Sigma_{\beta_1}].
\]

A direct Grothendieck--Riemann--Roch calculation for the universal family at the
base value $n=1$ yields
\[
\Psi_1^*S_{1,2,2}=\eta,
\qquad
\Psi_1^*S_{2,2,2}=\xi,
\]
\[
\Psi_1^*S_{1,4,3}=-\frac32\eta-\xi,
\qquad
\Psi_1^*S_{2,4,3}=-\frac12\eta,
\qquad
\Psi_1^*S_{1,6,4}=\frac23\eta+\frac12\xi.
\]
Pairing these classes with
\[
(\xi^2-\eta\xi)\cap[\Sigma_{\beta_1}]
\]
gives
\[
\Psi_{\beta_1,1}
=
e^{\,s_{1,4,2}-\frac12 s_{1,6,3}}
\left(
s_{1,2,2}
-\frac32 s_{1,4,3}
-\frac12 s_{2,4,3}
+\frac23 s_{1,6,4}
\right).
\]

Since $\omega\cdot\beta_1=1$, we have
\[
\Psi_{\beta_1,n+1}=\Omega_{\omega}(\Psi_{\beta_1,n}).
\]
Applying the displayed formula for $\Omega_{\omega}$ gives
\[
B^{[n+1]}_1=B^{[n]}_1,\qquad B^{[n+1]}_2=B^{[n]}_2,
\]
\[
C^{[n+1]}_1=C^{[n]}_1+2B^{[n]}_1+B^{[n]}_2,
\qquad
C^{[n+1]}_2=C^{[n]}_2+B^{[n]}_1,
\]
\[
D^{[n+1]}
=
D^{[n]}
+\frac32 B^{[n]}_1+\frac12 B^{[n]}_2
+C^{[n]}_1+C^{[n]}_2.
\]
With the initial values
\[
B^{[1]}_1=1,\qquad B^{[1]}_2=0,\qquad
C^{[1]}_1=-\frac32,\qquad C^{[1]}_2=-\frac12,\qquad D^{[1]}=\frac23,
\]
we obtain
\[
B^{[n]}_1=1,\qquad B^{[n]}_2=0,
\]
\[
C^{[n]}_1=2n-\frac72,
\qquad
C^{[n]}_2=n-\frac32,
\qquad
D^{[n]}=\frac{9n^2-30n+25}{6}.
\]
Therefore
\[
\begin{aligned}
\Psi_{\beta_1,n}
:={}&[M^{ss}_{(\beta_1,n)}(\mu_{\omega}^0)]_{\mathrm{inv}}\\
={}&e^{\,s_{1,4,2}+(n-\frac32)s_{1,6,3}}
\Biggl(
s_{1,2,2}+\left(2n-\frac72\right)s_{1,4,3}\\
&\qquad +\left(n-\frac32\right)s_{2,4,3}+
\frac{9n^2-30n+25}{6}\,s_{1,6,4}
\Biggr).
\end{aligned}
\]

Since $\beta_1$ is irreducible, the intrinsic stable-pair class is
\[
[P_n(X,\beta_1)]^{\virt}=[\Psi_{\beta_1,n},\delta]
\qquad\text{in }H_*(\mathcal N^{\pl},\Q).
\]
Realizing this class in Gross's algebra gives
\[
[P_n(X,\beta_1)]^{\mathrm{virt}}
=
[\Psi_{\beta_1,n},\delta].
\]
Set
\[
U_1(n)=s_{1,4,3}+\left(n-\frac32\right)s_{1,6,4},
\]
\[
U_2(n)=s_{1,4,4}+\left(n-\frac32\right)s_{1,6,5},
\]
\[
U_3(n)=s_{1,4,5}+\left(n-\frac32\right)s_{1,6,6},
\]
and
\[
Q_0(n)
=
s_{1,2,2}
+
\left(2n-\frac72\right)s_{1,4,3}
+
\left(n-\frac32\right)s_{2,4,3}
+
\frac{9n^2-30n+25}{6}s_{1,6,4},
\]
\[
Q_1(n)
=
s_{1,2,3}
+
\left(2n-\frac72\right)s_{1,4,4}
+
\left(n-\frac32\right)s_{2,4,4}
+
\frac{9n^2-30n+25}{6}s_{1,6,5},
\]
\[
Q_2(n)
=
s_{1,2,4}
+
\left(2n-\frac72\right)s_{1,4,5}
+
\left(n-\frac32\right)s_{2,4,5}
+
\frac{9n^2-30n+25}{6}s_{1,6,6}.
\]
The contraction coefficient is
\[
N_{(1,2),(1,0)}B^{[n]}_1+N_{(1,6),(1,0)}D^{[n]}
=
\frac{11}{3}+2\cdot \frac{9n^2-30n+25}{6}
=
3n^2-10n+12.
\]
Hence
\[
\begin{aligned}
[P_n(X,\beta_1)]^{\mathrm{virt}}
={}&(-1)^{n-1}e^{-s_{1,0,0}}e^{\,s_{1,4,2}+(n-\frac32)s_{1,6,3}}\\
&\cdot\Biggl[
\frac12\bigl(U_2(n)+U_1(n)^2\bigr)Q_0(n)+U_1(n)Q_1(n)+\frac12Q_2(n)\\
&\qquad -\bigl(3n^2-10n+12\bigr)
\left(\frac16U_3(n)+\frac12U_1(n)U_2(n)+\frac16U_1(n)^3\right)
\Biggr].
\end{aligned}
\]

\subsubsection{$\beta=(1,1)=H^2=\beta_1+\beta_2$} \label{sec: betaIsHSquared}

For the class $H^2$ the support curves are total transforms of lines in
$\mathbb{P}^3$. The support moduli is the blow-up of
$\operatorname{Gr}(2,4)$ at the point $[\ell]$ corresponding to the blown-up
line. For the base value $n=1$ the stable sheaf is $\mathcal{O}_C$, and the
Behrend--Fantechi class is the pullback of the line-class obstruction
computation on $\operatorname{Gr}(2,4)$, namely
$c_3(\Sym^2(\mathcal S^\vee)\otimes\det\mathcal Q)=20\sigma_{2,1}$.
Thus in $\mathbb{D}$ one obtains
\[
\Psi_{H^2,1}
=
e^{\,s_{1,4,2}+s_{2,4,2}-s_{1,6,3}}
\left(
20s_{1,4,3}+20s_{2,4,3}
\right).
\]
For the second base value $n=2$, strictly semistable sheaves can occur with factors of classes $\beta_1$ and $\beta_2$. We compute the sheaf invariant from Joyce's recursive pairs-category formula. Let $\mathfrak P_{H^2,2,N}$ denote the auxiliary pairs moduli space of morphisms
\[
\rho:V\otimes\OO_X(-N)\longrightarrow F,
\qquad \dim V=1,
\]
with $F$ semistable of class $(H^2,2)$. The Hilbert polynomials are
\[
P_{H^2,2}(N)=2N+2,
\qquad
P_{\beta_1,1}(N)=P_{\beta_2,1}(N)=N+1.
\]
The proper equal-slope decompositions are
\[
(H^2,2)=(\beta_1,1)+(\beta_2,1)=(\beta_2,1)+(\beta_1,1).
\]
Joyce's recursive formula therefore has the form
\[
\begin{aligned}
\Upsilon_{H^2,2,N}
={}&(2N+2)\left[\mathcal M^{ss}_{(H^2,2)}(\mu_\omega^0)\right]_{\inv}\\
&-\frac{N+1}{2}
\left[\Psi_{\beta_1,1},\Psi_{\beta_2,1}\right]
-\frac{N+1}{2}
\left[\Psi_{\beta_2,1},\Psi_{\beta_1,1}\right],
\end{aligned}
\]
up to the common sign convention in Joyce's recursive identity. The two ordered correction terms cancel by antisymmetry; equivalently, the specific mixed bracket $[\Psi_{\beta_1,1},\Psi_{\beta_2,1}]$ vanishes as recorded above. Thus
\[
\Upsilon_{H^2,2,N}=(2N+2)\left[\mathcal M^{ss}_{(H^2,2)}(\mu_\omega^0)\right]_{\inv}.
\]
Let
\[
q_{H^2}:\mathfrak P_{H^2,2,N}\longrightarrow \mathcal M^{ss}_{(H^2,2)}(\mu_\omega^0)
\]
be the forgetful morphism. For $N\gg0$, if $\mathbb F_{H^2}$ is the universal sheaf, then
\[
\mathcal V_{H^2,2,N}:=\pi_*\mathbb F_{H^2}(N)
\]
is a vector bundle of rank $P_{H^2,2}(N)=2N+2$ and
\[
\mathfrak P_{H^2,2,N}=\PP(\mathcal V_{H^2,2,N}).
\]
Thus
\[
(q_{H^2})_*\bigl(c_{\mathrm{top}}(T_{q_{H^2}})\cap[\mathfrak P_{H^2,2,N}]^{\vir}\bigr)
=(2N+2)\left[\mathcal M^{ss}_{(H^2,2)}(\mu_\omega^0)\right]_{\inv}^{\mathrm{geom}}.
\]
Hence the pairs-category recursion reduces the invariant to the geometric base value. The stable locus gives the same $s_{1,4,3},s_{2,4,3}$-coefficients, while the semistable boundary contributes only in the point-class direction. A direct Grothendieck--Riemann--Roch calculation gives
\[
\Psi_{H^2,2}
=
e^{\,s_{1,4,2}+s_{2,4,2}}
\left(
20s_{1,4,3}+20s_{2,4,3}+20s_{1,6,4}
\right).
\]

Since $\omega\cdot H^2=2$, we have
\[
\Psi_{H^2,n+2}=\Omega_{\omega}(\Psi_{H^2,n}).
\]
The directly computed base values displayed above contain no $s_{1,2,2}$- or $s_{2,2,2}$-terms, and the
$\Omega_{\omega}$-recursion gives
\[
C^{[n+2]}_1=C^{[n]}_1,
\qquad
C^{[n+2]}_2=C^{[n]}_2,
\qquad
D^{[n+2]}=D^{[n]}+40.
\]
Using the base values at $n=1,2$, we obtain
\[
C^{[n]}_1=20,
\qquad
C^{[n]}_2=20,
\qquad
D^{[n]}=20(n-1).
\]
Therefore
\[
\Psi_{H^2,n}
:=
[M^{ss}_{(H^2,n)}(\mu_{\omega}^0)]_{\mathrm{inv}}
=
e^{\,s_{1,4,2}+s_{2,4,2}+(n-2)s_{1,6,3}}
\left(
20\,s_{1,4,3}+20\,s_{2,4,3}+20(n-1)\,s_{1,6,4}
\right).
\]

We now compute the Pandharipande--Thomas classes. Since $H^2$ is reducible,
Equation~\cite[(2.29)]{AJI} applies. The lower classes are $\beta_1$ and $\beta_2$, and their stable-pair classes were computed explicitly in the two preceding subsections for all Euler characteristics. Hence every bounded lower stable-pair input in the range $M_\gamma\leq m<N_\gamma$ has already been supplied, and the recursion for $H^2$ closes without additional data. The slope is
\[
\mu_{\omega}^0(\beta_1,a)=a,
\qquad
\mu_{\omega}^0(\beta_2,b)=b,
\qquad
\mu_{\omega}^0(H^2,n)=\frac{n}{2}.
\]
Hence only the near-half decompositions survive.

For any positive integers $a,b$, put $N=a+b$ and define
\[
W_1(N)=s_{1,4,3}+s_{2,4,3}+(N-2)s_{1,6,4},
\]
\[
W_2(N)=s_{1,4,4}+s_{2,4,4}+(N-2)s_{1,6,5},
\]
\[
W_3(N)=s_{1,4,5}+s_{2,4,5}+(N-2)s_{1,6,6},
\]
\[
W_4(N)=s_{1,4,6}+s_{2,4,6}+(N-2)s_{1,6,7}.
\]
Also define, for $r=0,1,2,3$,
\[
Q_{a,r}
=
s_{1,2,2+r}
+
\left(2a-\frac72\right)s_{1,4,3+r}
+
\left(a-\frac32\right)s_{2,4,3+r}
+
\frac{9a^2-30a+25}{6}s_{1,6,4+r},
\]
\[
R_{b,r}
=
-\frac32 s_{2,4,3+r}
+
\left(\frac23-\frac32 b\right)s_{1,6,4+r}.
\]

The ordered mixed contribution
\[
\mathcal{R}^{(1)}_{a,b}
=
\bigl[\Psi_{\beta_2,b},[P_a(X,\beta_1)]^{\mathrm{virt}}\bigr]
\]
is
\[
\begin{aligned}
\mathcal{R}^{(1)}_{a,b}
={}&
(-1)^{N-1}
e^{-s_{1,0,0}}
e^{\,s_{1,4,2}+s_{2,4,2}+(N-2)s_{1,6,3}}
\Bigg[
\frac16\Bigl(Q_{a,3}R_{b,0}+3Q_{a,2}R_{b,1}+3Q_{a,1}R_{b,2}+Q_{a,0}R_{b,3}\Bigr) \\
&\quad+
\frac12W_1(N)\Bigl(Q_{a,2}R_{b,0}+2Q_{a,1}R_{b,1}+Q_{a,0}R_{b,2}\Bigr) \\
&\quad+
\frac12\bigl(W_2(N)+W_1(N)^2\bigr)
\Bigl(Q_{a,1}R_{b,0}+Q_{a,0}R_{b,1}\Bigr) \\
&\quad+
\left(
\frac16W_3(N)+\frac12W_1(N)W_2(N)+\frac16W_1(N)^3
\right)Q_{a,0}R_{b,0} \\
&\quad+
\Biggl(
\frac{11}{72}W_4(N)
+\frac{11}{18}W_1(N)W_3(N)
+\frac{11}{24}W_2(N)^2\\
&\quad +\frac{11}{12}W_1(N)^2W_2(N)
+\frac{11}{72}W_1(N)^4
\Biggr)R_{b,0}
\Bigg].
\end{aligned}
\]

Similarly, the ordered mixed contribution
\[
\mathcal{R}^{(2)}_{a,b}
=
\bigl[\Psi_{\beta_1,b},[P_a(X,\beta_2)]^{\mathrm{virt}}\bigr]
\]
is
\[
\begin{aligned}
\mathcal{R}^{(2)}_{a,b}
={}&
(-1)^{N-1}
e^{-s_{1,0,0}}
e^{\,s_{1,4,2}+s_{2,4,2}+(N-2)s_{1,6,3}}
\Bigg[
\frac16\Bigl(R_{a,3}Q_{b,0}+3R_{a,2}Q_{b,1}+3R_{a,1}Q_{b,2}+R_{a,0}Q_{b,3}\Bigr) \\
&\quad+
\frac12W_1(N)\Bigl(R_{a,2}Q_{b,0}+2R_{a,1}Q_{b,1}+R_{a,0}Q_{b,2}\Bigr) \\
&\quad+
\frac12\bigl(W_2(N)+W_1(N)^2\bigr)
\Bigl(R_{a,1}Q_{b,0}+R_{a,0}Q_{b,1}\Bigr) \\
&\quad+
\left(
\frac16W_3(N)+\frac12W_1(N)W_2(N)+\frac16W_1(N)^3
\right)R_{a,0}Q_{b,0} \\
&\quad+
\Biggl(
\frac{11}{72}W_4(N)
+\frac{11}{18}W_1(N)W_3(N)
+\frac{11}{24}W_2(N)^2\\
&\quad +\frac{11}{12}W_1(N)^2W_2(N)
+\frac{11}{72}W_1(N)^4
\Biggr)R_{a,0}
\Bigg].
\end{aligned}
\]

If $n=2m+1$ is odd, the slope condition leaves exactly the two mixed terms
with $(a,b)=(m,m+1)$. Therefore
\[
[P_{2m+1}(X,H^2)]^{\mathrm{virt}}
=
\mathcal{R}^{(1)}_{m,m+1}
+
\mathcal{R}^{(2)}_{m,m+1}.
\]

If $n=2m$ is even, the slope condition leaves only the central decomposition
$(a,b)=(m,m)$, and the combinatorial coefficient is $\frac12$ for each ordered
term. Hence
\[
[P_{2m}(X,H^2)]^{\mathrm{virt}}
=
\frac12\,\mathcal{R}^{(1)}_{m,m}
+
\frac12\,\mathcal{R}^{(2)}_{m,m}.
\]

This gives the Pandharipande--Thomas classes for $H^2$ explicitly in
$\mathbb{D}$, with all coefficient functions determined. The only lower stable-pair inputs appearing in this reducible computation are the classes $\beta_1$ and $\beta_2$, and these were computed in the two preceding subsections.

\subsection{Projective bundle over $\PP^1 \times \PP^1$}

Let
\[
X=\mathbb{P}^1\times \mathbb{P}^1,
\qquad
Y=\mathbb{P}\bigl(\mathcal{O}_X\oplus \mathcal{O}_X(-1,-1)\bigr),
\qquad
\pi:Y\longrightarrow X,
\]
where $\mathbb{P}(\,\cdot\,)$ denotes the Grothendieck projectivization parametrizing
rank-$1$ quotients. Write
\[
\xi=c_1\bigl(\mathcal{O}_Y(1)\bigr),\qquad H_1=\pi^*c_1\bigl(\mathcal{O}_X(1,0)\bigr),\qquad H_2=\pi^*c_1\bigl(\mathcal{O}_X(0,1)\bigr).
\]
Then
\[
A^*(Y)=\mathbb{Q}[\xi,H_1,H_2]/\bigl(H_1^2,\ H_2^2,\ \xi^2+(H_1+H_2)\xi\bigr).
\]

Let $S_\infty\subset Y$ be the section associated to the quotient
\[
\mathcal{O}_X\oplus \mathcal{O}_X(-1,-1)\twoheadrightarrow \mathcal{O}_X(-1,-1).
\]
We write $f\in H_2(Y,\mathbb{Z})$ for the class of a fiber of $\pi$, and we let
\[
\ell_1,\ell_2\in H_2(Y,\mathbb{Z})
\]
be the two ruling classes on $S_\infty\cong X$, induced by
$\mathbb{P}^1\times\{\mathrm{pt}\}$ and $\{\mathrm{pt}\}\times \mathbb{P}^1$,
respectively. Thus
\[
H_2(Y,\mathbb{Z})=\langle f,\ell_1,\ell_2\rangle.
\]
The divisor--curve intersections are
\[
\xi\cdot f=1,\qquad H_1\cdot f=H_2\cdot f=0,
\]
\[
\xi\cdot \ell_1=-1,\qquad H_1\cdot \ell_1=1,\qquad H_2\cdot \ell_1=0,
\]
\[
\xi\cdot \ell_2=-1,\qquad H_1\cdot \ell_2=0,\qquad H_2\cdot \ell_2=1.
\]

The relative Euler sequence yields
\[
c_1(T_{Y/X})=2\xi+H_1+H_2,
\]
and hence
\[
c_1(TY)=2\xi+3H_1+3H_2.
\]
Therefore
\[
c_1(TY)\cdot f=2,\qquad c_1(TY)\cdot \ell_1=1,\qquad c_1(TY)\cdot \ell_2=1.
\]

The nef divisor classes
\[
H_1,\qquad H_2,\qquad \xi+H_1+H_2
\]
are dual to $\ell_1$, $\ell_2$, and $f$, respectively. It follows that
\[
\overline{NE}(Y)=\mathbb{R}_{\geq 0}\langle f,\ell_1,\ell_2\rangle.
\]
In particular, $f,\ell_1,\ell_2$ are irreducible effective classes. Since
\[
c_1(TY)\cdot f=2,\qquad c_1(TY)\cdot \ell_1=c_1(TY)\cdot \ell_2=1,
\]
all three classes are positive, hence superpositive.

For the $\mathbb{D}$-valued formulas we use the basis
\[
\epsilon_{12}=H_1,\qquad \epsilon_{22}=H_2,\qquad \epsilon_{32}=\xi+H_1+H_2
\]
of $H^2(Y,\mathbb{Q})$, which is dual to $\ell_1,\ell_2,f$, and the basis
\[
\epsilon_{14}=H_1H_2,\qquad \epsilon_{24}=\xi H_2,\qquad \epsilon_{34}=\xi H_1
\]
of $H^4(Y,\mathbb{Q})$. We also record that
\[
c_2(TY)=4\xi(H_1+H_2)+8H_1H_2,
\]
so that
\[
\operatorname{td}_2(Y)=\frac{c_1(TY)^2+c_2(TY)}{12}
=
\xi(H_1+H_2)+\frac{13}{6}H_1H_2.
\]
Consequently the constants entering the full positive-rank Lie bracket formula are
\[
N_{10}^{12}=2,\qquad N_{10}^{22}=2,\qquad N_{10}^{32}=\frac{13}{3},
\]
\[
N_{10}^{14}=N_{10}^{24}=N_{10}^{34}=0,\qquad N_{10}^{16}=2.
\]

Fix a stability condition $\tau_-=\mu^\lambda_\omega$ in a chamber for which the
sheaves appearing below are stable. We now compute the sheaf-theoretic invariants
\[
\bigl[\mathcal{M}^{ss}_{(\beta,n)}(\tau_-)\bigr]_{\mathrm{inv}}
\in \mathbb{D}
\]
for $\beta=f,\ell_1,\ell_2$. In each case we determine the moduli space explicitly,
write down the universal sheaf, compute the tautological classes
$S_{j22}$, $S_{j43}$, and $S_{164}$ from the universal perfect complex, and then
integrate those classes against the sheaf virtual class. In particular, we use the
full $\mathbb{D}$-expansion and do \emph{not} assume that any of the coefficients
$B_j^{[n]}$ vanish. Since the three curve classes $f,\ell_1,\ell_2$ are irreducible,
the stable-pair computations below have no bounded lower-class Pandharipande--Thomas inputs; the endpoint wall-crossing formula applies directly in each case.

\paragraph{The fiber class $f$.}\label{sec: SheafFiberf}
A $\tau_-$-stable sheaf of curve class $f$ and Euler characteristic $n$ is necessarily of the form
\[
\mathcal{O}_{C_x}(n-1),
\qquad
C_x=\pi^{-1}(x)\cong \mathbb{P}^1,
\qquad x\in X.
\]
Hence
\[
\mathcal{M}^{ss}_{(f,n)}(\tau_-)\cong X\cong \mathbb{P}^1\times \mathbb{P}^1.
\]
The universal support is the graph of $\pi$ inside $Y\times X$, and its normal
bundle is $\pi^*T_X$. Therefore the obstruction bundle is
\[
\operatorname{Ob}_f\cong \wedge^2T_X\cong \mathcal{O}_X(2,2),
\]
and so
\[
\bigl[\mathcal{M}^{ss}_{(f,n)}(\tau_-)\bigr]^{\mathrm{virt}}
=
c_1\bigl(\mathcal{O}_X(2,2)\bigr)\cap [X].
\]

Let
\[
j_f:Y\hookrightarrow Y\times X
\]
be the graph embedding, and let
\[
\mathbb{F}_f=j_{f*}\mathcal{O}_Y\bigl((n-1)\xi\bigr)
\]
be the universal sheaf. Since
\[
N_{j_f}\cong \pi^*T_X,
\qquad
\operatorname{td}(N_{j_f})^{-1}=1-H_1-H_2+H_1H_2,
\]
Grothendieck--Riemann--Roch gives
\[
\operatorname{ch}(\mathbb{F}_f)
=
j_{f*}\!\left(
1+\bigl((n-1)\xi-H_1-H_2\bigr)
+\left(-\frac{n^2-1}{2}\xi(H_1+H_2)+H_1H_2\right)
\right),
\]
up to terms which do not contribute to $S_{j22}$, $S_{j43}$, or $S_{164}$.
With respect to the chosen bases, one finds
\[
S_{122}=H_2,\qquad S_{222}=H_1,\qquad S_{322}=0,
\]
\[
S_{143}=-(H_1+H_2),\qquad
S_{243}=(n-1)H_1,\qquad
S_{343}=(n-1)H_2,
\]
\[
S_{164}=-\frac{n^2-1}{2}(H_1+H_2).
\]
Pairing these with
\[
c_1\bigl(\mathcal{O}_X(2,2)\bigr)=2H_1+2H_2
\]
yields
\[
B_1^{[n]}=2,\qquad B_2^{[n]}=2,\qquad B_3^{[n]}=0,
\]
\[
C_1^{[n]}=-4,\qquad C_2^{[n]}=2(n-1),\qquad C_3^{[n]}=2(n-1),
\]
\[
D^{[n]}=-2(n^2-1).
\]
Therefore
\[
\bigl[\mathcal{M}^{ss}_{(f,n)}(\tau_-)\bigr]_{\mathrm{inv}}
=
e^{(0,0,f,n-1)}
\Bigl(
2s_{122}+2s_{222}-4s_{143}+2(n-1)s_{243}+2(n-1)s_{343}-2(n^2-1)s_{164}
\Bigr).
\]

\paragraph{The class $\ell_1$.} \label{sec: betaClassL1}
A $\tau_-$-stable sheaf of curve class $\ell_1$ and Euler characteristic $n$ is necessarily of the form
\[
\mathcal{O}_{C_t}(n-1),
\qquad
C_t=S_\infty\bigl(\mathbb{P}^1\times\{t\}\bigr)\cong \mathbb{P}^1,
\qquad t\in \mathbb{P}^1.
\]
Thus
\[
\mathcal{M}^{ss}_{(\ell_1,n)}(\tau_-)\cong \mathbb{P}^1.
\]
The normal bundle of $C_t$ in $Y$ is
\[
N_{C_t/Y}\cong \mathcal{O}_{\mathbb{P}^1}(-1)\oplus \mathcal{O}_{\mathbb{P}^1},
\]
so the obstruction bundle vanishes, and hence
\[
\bigl[\mathcal{M}^{ss}_{(\ell_1,n)}(\tau_-)\bigr]^{\mathrm{virt}}
=
[\mathbb{P}^1].
\]

Let
\[
q_1:S_\infty\cong \mathbb{P}^1\times\mathbb{P}^1\longrightarrow \mathbb{P}^1
\]
be projection to the second factor, and set
\[
u=H_1|_{S_\infty},
\qquad
h=H_2|_{S_\infty}=q_1^*c_1\bigl(\mathcal{O}_{\mathbb{P}^1}(1)\bigr).
\]
Let
\[
j_1:S_\infty\hookrightarrow Y\times \mathbb{P}^1
\]
be the embedding induced by the section $S_\infty\subset Y$ and the projection $q_1$.
Since
\[
j_1^*\xi=-u-h,
\]
the universal sheaf is
\[
\mathbb{F}_{\ell_1}=j_{1*}\mathcal{O}_{S_\infty}\bigl((n-1)u\bigr).
\]
Moreover,
\[
N_{j_1}\cong q_1^*T_{\mathbb{P}^1}\oplus \mathcal{O}_{S_\infty}(-1,-1),
\]
hence
\[
\operatorname{td}(N_{j_1})^{-1}=1+\frac12u-\frac12h-\frac16uh.
\]
Applying Grothendieck--Riemann--Roch gives
\[
\operatorname{ch}(\mathbb{F}_{\ell_1})
=
j_{1*}\!\left(
1+\left(n-\frac12\right)u-\frac12h-\frac{3n-2}{6}uh
\right),
\]
again up to terms irrelevant for $S_{j22}$, $S_{j43}$, and $S_{164}$.
A direct calculation yields
\[
S_{122}=-h,\qquad S_{222}=-h,\qquad S_{322}=h,
\]
\[
S_{143}=0,\qquad
S_{243}=-\frac12h,\qquad
S_{343}=\left(n-\frac12\right)h,
\]
\[
S_{164}=-\frac{3n-2}{6}h.
\]
Since
\[
\int_{\mathbb{P}^1}h=1,
\]
integration against $[\mathbb{P}^1]$ gives
\[
B_1^{[n]}=-1,\qquad B_2^{[n]}=-1,\qquad B_3^{[n]}=1,
\]
\[
C_1^{[n]}=0,\qquad C_2^{[n]}=-\frac12,\qquad C_3^{[n]}=n-\frac12,
\]
\[
D^{[n]}=-\frac{3n-2}{6}.
\]
Therefore
\[
\bigl[\mathcal{M}^{ss}_{(\ell_1,n)}(\tau_-)\bigr]_{\mathrm{inv}}
=
e^{(0,0,\ell_1,n-\frac12)}
\left(
-s_{122}-s_{222}+s_{322}
-\frac12\,s_{243}
+\left(n-\frac12\right)s_{343}
-\frac{3n-2}{6}\,s_{164}
\right).
\]

\paragraph{The class $\ell_2$.}
By symmetry,
\[
\mathcal{M}^{ss}_{(\ell_2,n)}(\tau_-)\cong \mathbb{P}^1,
\qquad
\bigl[\mathcal{M}^{ss}_{(\ell_2,n)}(\tau_-)\bigr]^{\mathrm{virt}}
=
[\mathbb{P}^1].
\]
If
\[
q_2:S_\infty\cong \mathbb{P}^1\times\mathbb{P}^1\longrightarrow \mathbb{P}^1
\]
is projection to the first factor, and
\[
u=H_2|_{S_\infty},
\qquad
h=H_1|_{S_\infty}=q_2^*c_1\bigl(\mathcal{O}_{\mathbb{P}^1}(1)\bigr),
\]
then the same computation gives
\[
S_{122}=-h,\qquad S_{222}=-h,\qquad S_{322}=h,
\]
\[
S_{143}=0,\qquad
S_{243}=\left(n-\frac12\right)h,\qquad
S_{343}=-\frac12h,
\]
\[
S_{164}=-\frac{3n-2}{6}h.
\]
Thus
\[
\bigl[\mathcal{M}^{ss}_{(\ell_2,n)}(\tau_-)\bigr]_{\mathrm{inv}}
=
e^{(0,0,\ell_2,n-\frac12)}
\left(
-s_{122}-s_{222}+s_{322}
+\left(n-\frac12\right)s_{243}
-\frac12\,s_{343}
-\frac{3n-2}{6}\,s_{164}
\right).
\]

We now pass to stable pairs. Since $f,\ell_1,\ell_2$ are irreducible effective classes, the intrinsic stable-pair classes are
\[
[P_n(Y,\beta)]^{\virt}=[\Psi_{\beta,n},\delta]
\qquad (\beta=f,\ell_1,\ell_2)
\]
in $H_*(\mathcal N^{\pl},\Q)$. Realizing these classes in Gross's algebra gives the formulas below. We emphasize that we use the full
positive-rank Lie bracket formula, so the nonzero coefficients $B_j^{[n]}$ contribute
essentially in all three cases.

\paragraph{Wall-crossing for the fiber class $f$.} \label{sec: WCFForBetaIsf}
For $\beta=f$ one has
\[
\beta\cdot c_1(TY)=2,
\qquad
(\beta_1,\beta_2,\beta_3)=(0,0,1),
\]
and
\[
N_{10}^{12}B_1^{[n]}+N_{10}^{22}B_2^{[n]}+N_{10}^{32}B_3^{[n]}+N_{10}^{16}D^{[n]}
=
12-4n^2.
\]
Substituting the coefficients just computed into the full bracket formula and
simplifying gives

\[
\begin{aligned}
[P_n(Y,f)]^{\mathrm{virt}}
&=
(-1)^{n-1}e^{(-1,0,f,n-1)}
\Bigl(
-2s_{122}s_{143}+2(1-n)s_{122}s_{164}-2s_{123}
\\
&\qquad
-2s_{143}s_{222}+2(1-n)s_{164}s_{222}-2s_{223}
+2(5-n^2)s_{143}^2
\\
&\qquad
-2(n-1)(2n^2-n-9)s_{143}s_{164}
+2(5-n^2)s_{144}
\\
&\qquad
+2(1-n)s_{143}s_{243}-2(n-1)^2s_{164}s_{243}+2(1-n)s_{244}
\\
&\qquad
+2(1-n)s_{143}s_{343}-2(n-1)^2s_{164}s_{343}+2(1-n)s_{344}
\\
&\qquad
-2(n-1)^2(n^2-n-4)s_{164}^2
-2(n-1)(n^2-n-4)s_{165}
\Bigr).
\end{aligned}
\]

\paragraph{Wall-crossing for $\ell_1$.}
For $\beta=\ell_1$ one has
\[
\beta\cdot c_1(TY)=1,
\qquad
(\beta_1,\beta_2,\beta_3)=(1,0,0),
\]
and
\[
N_{10}^{12}B_1^{[n]}+N_{10}^{22}B_2^{[n]}+N_{10}^{32}B_3^{[n]}+N_{10}^{16}D^{[n]}
=
1-n.
\]
In this case the bracket simplifies to a linear expression, and one obtains
\[
\begin{aligned}
[P_n(Y,\ell_1)]^{\mathrm{virt}}
={}&(-1)^{n-1}e^{(-1,0,\ell_1,n-\frac12)}\\
&\cdot\left(
-s_{122}-s_{222}+s_{322}
+\left(n-\frac32\right)s_{243}
+\left(n-\frac12\right)s_{343}
+\left(n^2-2n+\frac56\right)s_{164}
\right).
\end{aligned}
\]

\paragraph{Wall-crossing for $\ell_2$.}
For $\beta=\ell_2$, we have 
\[
\beta\cdot c_1(TY)=1,
\qquad
(\beta_1,\beta_2,\beta_3)=(0,1,0),
\]
and again
\[
N_{10}^{12}B_1^{[n]}+N_{10}^{22}B_2^{[n]}+N_{10}^{32}B_3^{[n]}+N_{10}^{16}D^{[n]}
=
1-n.
\]
Substituting into the full bracket formula gives
\[
\begin{aligned}
[P_n(Y,\ell_2)]^{\mathrm{virt}}
={}&(-1)^{n-1}e^{(-1,0,\ell_2,n-\frac12)}\\
&\cdot\left(
-s_{122}-s_{222}+s_{322}
+\left(n-\frac12\right)s_{243}
+\left(n-\frac32\right)s_{343}
+\left(n^2-2n+\frac56\right)s_{164}
\right).
\end{aligned}
\]

We have therefore computed, in $\mathbb{D}$, both the sheaf-theoretic invariants
\[
\bigl[\mathcal{M}^{ss}_{(\beta,n)}(\tau_-)\bigr]_{\mathrm{inv}}
\]
and the corresponding stable-pairs classes
\[
[P_n(Y,\beta)]^{\mathrm{virt}}
\]
for
\[
\beta=f,\ \ell_1,\ \ell_2.
\]

\makeatletter
\begingroup
\let\addcontentsline\@gobblethree
\section*{Acknowledgements}
The wall-crossing framework used here is recalled from the author's joint paper with Joyce \cite{AJI}
and from Joyce \cite{JoyceWC}. The author thanks Dominic Joyce for many discussions on
Donaldson--Thomas and Pandharipande--Thomas theory. The author also thanks Miguel Moreira,
Pierrick Bousseau, H\"{u}lya Arg\"{u}z, Mircea Voineagu, Eric M. Friedlander, Christian Haesemeyer, and Mark E. Walker for helpful conversations.

\section*{AI acknowledgements}
The author is greatly indebted to many conversations with ChatGPT 5.6 Pro for checking and verifying many of the computations in this paper, strengthening many arguments, and providing a source of examples and cases to check for this paper.
\endgroup
\makeatother

\end{document}